\documentclass[a4paper,11pt,reqno]{article}

\usepackage[utf8]{inputenc}
\usepackage[T1]{fontenc}
\usepackage{lmodern}
\usepackage[english]{babel}
\usepackage{microtype}

\usepackage{amsmath,amssymb,amsfonts,amsthm}
\usepackage{mathtools,accents}
\usepackage{mathrsfs}
\usepackage{aliascnt}
\usepackage{braket}
\usepackage{bm}

\usepackage[a4paper,margin=3cm]{geometry}
\usepackage[citecolor=blue,colorlinks]{hyperref}

\usepackage{enumerate}
\usepackage{xcolor}

\makeatletter
\g@addto@macro\@floatboxreset\centering
\makeatother

\makeatletter
\def\newaliasedtheorem#1[#2]#3{
  \newaliascnt{#1@alt}{#2}
  \newtheorem{#1}[#1@alt]{#3}
  \expandafter\newcommand\csname #1@altname\endcsname{#3}
}
\makeatother

\numberwithin{equation}{section}

\newtheoremstyle{slanted}{\topsep}{\topsep}{\slshape}{}{\bfseries}{.}{.5em}{}

\theoremstyle{plain}
\newtheorem{theorem}{Theorem}[section]
\newaliasedtheorem{proposition}[theorem]{Proposition}
\newaliasedtheorem{lemma}[theorem]{Lemma}
\newaliasedtheorem{corollary}[theorem]{Corollary}
\newaliasedtheorem{counterexample}[theorem]{Counterexample}

\theoremstyle{definition}
\newaliasedtheorem{definition}[theorem]{Definition}
\newaliasedtheorem{question}[theorem]{Question}
\newaliasedtheorem{example}[theorem]{Example}
\newaliasedtheorem{conjecture}[theorem]{Conjecture}

\theoremstyle{remark}
\newaliasedtheorem{remark}[theorem]{Remark}

\newcommand{\setQ}{\mathbb{Q}}
\newcommand{\setR}{\mathbb{R}}

\let\altphi\phi
\let\phi\varphi
\let\varphi\altphi
\let\altphi\undefined

\newcommand{\di}{\mathop{}\!\mathrm{d}}

\newcommand{\bs}{{\rm bs}}

\DeclareMathOperator{\supp}{supp}

\newcommand{\dist}{\mathsf{d}}

\newcommand{\meas}{\mathfrak{m}}

\newcommand{\Algebra}{{\mathscr A}}

\DeclareMathOperator{\RCD}{RCD}
\newfont{\tmpf}{cmsy10 scaled 2500}

\begin{document}
\title{Ricci curvature and Orientability}
\author{
Shouhei Honda
\thanks{Tohoku University, \url{shonda@m.tohoku.ac.jp}}} 
\maketitle

\begin{abstract}
In this paper we define an orientation of a measured Gromov-Hausdorff limit space of Riemannian manifolds with uniform Ricci bounds from below.
This is the first observation of orientability for metric measure spaces.
Our orientability has two fundamental properties.
One of them is the stability with respect to noncollapsed sequences. 
As a corollary we see that if the cross section of a tangent cone of a noncollapsed limit space of orientable Riemannian manifolds is smooth, then it is also orientable in the ordinary sense, which can be regarded as a new obstruction for a given manifold to be the cross section of a tangent cone.  
The other one is that there are only two choices for orientations on a limit space.
We also discuss relationships between $L^2$-convergence of orientations and convergence of currents in metric spaces.  
In particular for a noncollapsed sequence, we prove a compatibility between the intrinsic flat convergence by Sormani-Wenger, the pointed flat convergence by Lang-Wenger, and the Gromov-Hausdorff convergence, which is a generalization of a recent work by Matveev-Portegies to the noncompact case. Moreover combining this compatibility with the second property of our orientation gives an explicit formula for the limit integral current by using an orientation on a limit space.
Finally dualities between de Rham cohomologies on an oriented limit space are proven.
\end{abstract}

\tableofcontents

\section{Introduction}
\subsection{Main results}
In this paper we discuss orientability of Ricci limit spaces.
A pointed metric measure space $(X, x, \meas)$ is said to be a \textit{Ricci limit space} if there exist $n \in \mathbb{N}$, a sequence of pointed Riemannian manifolds $(X_i, x_i)$ such that $\mathrm{Ric}_{X_i} \ge -(n-1)$ and that $(X_i, x_i, \mathcal{H}^n/\mathcal{H}^n(B_1(x_i)))$ measured Gromov-Hausdorff (written by mGH, for short) converge to $(X, x, \meas)$, (denoted by $(X_i, x_i, \mathcal{H}^n/\mathcal{H}^n(B_1(x_i))) \stackrel{GH}{\to} (X, x, \meas)$, for short), where $\mathcal{H}^n$ is the $n$-dimensional Hausdorff measure (we usually fix $n$ as the dimension of a manifold).
Our goals are to define an \textit{orientation} of $(X, x, \meas)$ and to establish nice properties.

First let us recall the definition of an orientation of a smooth $n$-dimensional Riemannian manifold $M$.
We say that $M$ is \textit{orientable} if there exists a top-dimensional differential form $\omega \in L^{\infty}(\bigwedge^nT^*M)$ with the following two conditions;
\begin{enumerate}
\item{(Normalization)} $|\omega| \equiv 1$ on $M$,
\item{(Smooth regularity)} $\omega$ is smooth.
\end{enumerate}   
Then $\omega$ is said to be an orientation of $M$.
Of course there are only two choices for orientations, let us call this property the uniqueness of orientations for short.

It is well-known that this definition is equivalent to the ordinary one on a smooth manifold, i.e. there exists a smooth atlas $\{(U_i, \phi_i)\}_i$ of $M$ such that the Jacobi matrix $J(\phi_j \circ (\phi_i)^{-1})$ of each transition map $\phi_j \circ (\phi_i)^{-1}$ is positive determinant. 

Next let us discuss the Ricci limit case, $(X, x, \meas)$.
Cheeger-Colding proved in \cite{CheegerColding3} that $X$ is $\meas$-rectifiable.
This allows us to consider the vector bundle $\bigwedge^lT^*X$. In particular $l$-dimensional differential forms $\eta(z) \in \bigwedge^lT^*_zX$ on $X$ make sense for a.e. $z \in X$. Note that each fiber $\bigwedge^lT_z^*X$ has the canonical inner product.

Recently Colding-Naber proved in \cite{ColdingNaber} that there exists $k \in \{0, 1, \ldots, n\}$, denoted by $\mathrm{dim}\,X$, such that for a.e. $z \in X$ all tangent cones at $z$ of $(X, x, \meas)$ are isometric to $\mathbf{R}^k$.
This allows us to define top-dimensional differential forms $\omega(z) \in \bigwedge^kT^*_zX$ for a.e. $z \in X$.
In particular there is a top-dimensional differential form $\omega \in \bigwedge^kT^*X$ such that $|\omega(z)|=1$ for a.e. $z \in X$. Note that there are uncountably many such differential forms.
For example for any $a \in [0, \pi]$, a differential 1-form $(1_{[0, a]}-1_{(a, \pi]})\dist t$ on $([0, \pi], \mathcal{H}^1/\pi)$ gives such an example, where $1_A$ denotes the indicator function of $A$. See also subsection 6.1.

The difficulty to define the orientablity of $(X, x, \meas)$ is to find a condition of a kind of smooth regularity (2) above for such differential forms because since each fiber of $\bigwedge^kT^*X$ is well-defined only on a Borel subset of the regular set of $(X, x, \meas)$, we can not discuss the \textit{continuity} of a differential form in the ordinary sense. For example it is known that there is a noncollapsed GH-limit space of Riemannian manifolds whose setional curvature bounded below by $0$ such that the singular set of the limit space is dense. See Example (2) in page 632 of \cite{OtsuShioya} by Otsu-Shioya.

In order to overcome the difficulty we use \textit{test functions} as follows; let us denote by $\mathrm{Test}F(X)$ the set of bounded Lipschitz functions $f$ such that $f \in H^{1, 2}(X)$ and that $f$ is in the domain of the Dirichlet Laplacian $\Delta$ on $X$ with $\Delta f \in H^{1, 2}(X)$, where $H^{1, 2}(X)$ is the Sobolev space for functions on $X$. Note that $\mathrm{Test}F(X)$ is dense in $H^{1, 2}(X)$ (in particular it is also dense in $L^2(X)$).
This is a key notion in the theory of $RCD$-spaces (c.f. \cite{Gigli} by Gigli). 

We are now in a position to give the definition of the orientability of $(X, x, \meas)$ as follows;
\begin{definition}[Orientation, Definition \ref{def:ori}]
We say that a top-dimensional differential form $\omega \in L^{\infty}(\bigwedge^kT^*X)$ is an \textit{orientation} if the following two conditions hold;
\begin{enumerate}
\item{(Normalization)} $|\omega (z)|=1$ for a.e. $z \in X$, 
\item{(Regularity)} $\langle \omega, f_0\dist f_1 \wedge \cdots \wedge \dist f_k \rangle \in H^{1, 2}(X)$ for any $f_i \in \mathrm{Test}F(X)$.
\end{enumerate}
\end{definition}
The regularity condition above plays a role of a kind of the smooth regularity (2) above.
In fact we prove the following uniqueness;
\begin{theorem}[Uniqueness]\label{thm:unique}
If $\omega_1, \omega_2$ are orientations of $(X, x, \meas)$,
then we have either $\omega_1(z)=\omega_2(z)$ for a.e. $z \in X$, or $\omega_1(z)=-\omega_2(z)$ for a.e. $z \in X$.
\end{theorem}
Moreover we will prove that our orientability is compatible with the smooth case. 
For example if $(X, x, \meas)$ satisfies that $X$ is isometric to a $k$-dimensional smooth Riemannian manifold with $\meas= \int e^f\dist \mathcal{H}^k$ for some locally $H^{1, 2}$-Sobolev function $f$ on $X$, then $(X, x, \meas)$ is orientable in the sense above if and only if $X$ is orientable in the ordinary sense. See Propositions \ref{prop:comp1} and \ref{prop:comp2}. 

Let us discuss two examples of metric spaces with probability measures;
\begin{enumerate}
\item $\left( [0, \pi], \frac{1}{\pi}\mathcal{H}^1\right)$,
\item $\left( [0, \pi], \frac{1}{2}\int \sin t \dist t \right)$.
\end{enumerate}
It is easy to check that these are (collapsed) mGH-limit spaces of sequences of Riemannian metrics on the $2$-dimensional sphere $\mathbf{S}^2$ with canonical probability measures, whose sectional curvature bounded below.
In particular these are (non-pointed) Ricci limit spaces.
We will check that these are orientable in the sense above, in fact, the canonical 1-form $\omega:=\dist t$ gives an orientablity in both cases. However their proofs are different. For the first example, Fourier expansion plays a key role in the proof.
For the second one,  a key point in the proof is a fact that the capacity of the singular set, $\{0, \pi\}$, is zero (note that the capacity of $\{0, \pi\}$ in the first example is \textit{not} zero). See Remarks \ref{100} and \ref{1}.

Next we introduce the stability of orientability.
For that let us start to observe following two examples of noncollapsing/collapsing sequences; 
\begin{enumerate}
\setcounter{enumi}{1}
\setcounter{enumi}{2}
\item $(\mathbf{R}P^2, rg_{\mathbf{R}P^2}, p, \frac{1}{\mathcal{H}^2(B_1^{r\dist}(p))}\mathcal{H}^2) \stackrel{GH}{\to} (\mathbf{R}^2, g_{\mathbf{R}^2}, 0_n, \frac{1}{\mathcal{H}^2(B_1(0_n))}\mathcal{H}^2)$ as $r \uparrow \infty$, where $g_{\mathbf{R}P^2}, g_{\mathbf{R}^2}$ are canonical Riemannian metrics on $\mathbf{R}P^2, \mathbf{R}^2$, respectively,
\item let  $\{ \pm 1\}$ act on $\mathbf{S}^2(1) \times \mathbf{S}^2(1)$ by $(-1) \cdot (z, w):=(-z, -w)$, let $M:=(\mathbf{S}^2(1) \times \mathbf{S}^1(1))/\{\pm 1\}$ and let $g_{M, r}:=(g_{\mathbf{S}^2(1)}+r^2g_{\mathbf{S}^1(1)})/\{\pm 1\}$ be the canonical quotient Riemannian metric on $M$ for any $r \in (0, \infty)$, where $\mathbf{S}^2(1):=\{x \in \mathbf{R}^3;|x|=1\}$, $g_{\mathbf{S}^2(1)}$ is the canonical Riemannian metric on $\mathbf{S}^2(1)$.
Then
$(M, g_{M, r}, \frac{1}{\mathcal{H}^4(M)}\mathcal{H}^4) \stackrel{GH}{\to} (\mathbf{R}P^2, \dist_{\mathbf{R}P^2}, \frac{1}{2\pi}\mathcal{H}^2)$ as $r \downarrow 0$.
\end{enumerate}
The example 3 tells us that in general the limit space of a sequence of non-orientable spaces is \textit{not} non-orientable, i.e. the non-orientability is not stable under mGH-convergence even if the sequence is noncollapsed.
The final example tells us that if the sequence is collapsed, then in general the orientability is not stable under mGH-convergence.

The remaining case about the possible stability for orientability is that the sequence is noncollapsed, and consists of orientable spaces. The second main result is to give a positive answer to this question.
In order to give the precise statement, we recall the following; for any mGH-convergent sequence of Ricci limit spaces $(Y_i, y_i, \meas_i) \stackrel{GH}{\to} (Y, y, \meas)$, their dimensions are lower semicontinuous, i.e. $\liminf_{i \to \infty}\mathrm{dim}\,Y_i \ge \mathrm{dim}\,Y$, which was proven in \cite{Honda2}.
This allows us to define the sequence $(Y_i, y_i, \meas_i)$ to be \textit{noncollapsed} by satisfying $\lim_{i \to \infty}\mathrm{dim}\,Y_i=\mathrm{dim}\,Y$.
Note that this formulation is well-known and is equivalent to satisfy $\liminf_{i \to \infty}\mathcal{H}^n(B_1(y_i))>0$ if the sequence $(Y_i, y_i, \meas_i)$ consists of Riemannian manifolds with canonical normalized measures.

For example it was proven in \cite{KL} by Kapovitch-Li that along the interior of any limit geodesic on any Ricci limit space, same scale tangent cones gives a noncollapsed (H\"older) continuous sequence with respect to the mGH-convergence in this sense. 

The stability result is stated as follows;
\begin{theorem}[Stability, Theorem \ref{thm:stability}]\label{ah}
Let $(Y_i, y_i, \meas_i)$ be a mGH-convergent sequence of Ricci limit spaces to $(Y, y, \meas)$.
Assume that this is a noncollapsed sequence and that each $(Y_i, y_i, \meas_i)$ is orientable.
Then $(Y, y, \meas)$ is also orientable.
\end{theorem}
This stability result and the compatibility with the smooth case as mentioned show the following;
\begin{corollary}\label{cor:nonex}
Let $Z$ be a compact metric space whose Hausdorff dimension $n-k-1$.
If there exists an open subset $O$ of $Z$ such that $O$ is isometric to a non-orientable smooth (possibly incomplete) Riemannian manifold,
 then the metric space $(\mathbf{R}^k \times C(Z), (0_k, p))$ never appears as a tangent cone at a point of a noncollapsed oriented Ricci limit space, where $C(Z)$ is the metric cone over $Z$ and $p$ denotes the pole. 
\end{corollary} 
The sectional curvature version of this corollary is known, more strongly, Kapovitch proved in \cite{Kapo} that for any noncollapsed GH-limit space of $n$-dimensional Riemannian manifolds with uniform sectional curvature bounds from below, the cross section (which is the space of directions) of the tangent cone at any point in the limit space is homeomorhic to a sphere of dimension $n-1$.

However in the case of noncollapsed Ricci limit spaces, the corollary makes sense.
For example we can find in \cite{DW} by Dancer-Wang an example of an Einstein metric on $\mathbf{R}P^6 \times \mathbf{R}^4$ such that the asymptotic cone is the metric cone over $\mathbf{R}P^6 \times \mathbf{S}^3$. This observation is due to Hattori.

On the other hand Colding-Naber gave in \cite{ColdingNaber2} necessary and sufficient conditions for the GH-closure of an open smooth family $\Omega$ of closed Riemannian manifolds to be the set $\overline{\Omega}_{Y, p}$ of all cross sections of all tangent cones at some point $p$ of some noncollapsed Ricci limit space $Y$, (i.e. $\overline{\Omega}=\overline{\Omega}_{Y, p}$).
Corollary \ref{cor:nonex} can be regarded as a new obstruction for their result.
In particular $\mathbf{R}P^6 \times \mathbf{S}^3$ with any metric \textit{never} appears as the cross section of a tangent cone of a noncollapsed Ricci limit space of orientable Riemannian manifolds.

It is well-known that orientability is related to the theory of currents.
In fact, even in our setting we will establish 
a relationship between our orientability and the theory of metric currents by Ambrosio-Kirchheim \cite{AmbrosioKirchheim} (more generally, local currents by Lang \cite{Lang} and Lang-Wenger \cite{LW}).
In order to give the precise statement, for an orientation $\omega$ of $(X, x, \meas)$, let $T_{\omega}$ be a functional defined by
$$
T_{\omega}(f_0, f_1, \ldots, f_k):=\int_X\langle \omega, f_0 \dist f_1 \wedge \cdots \wedge \dist f_k \rangle \dist \meas
$$
for any Lipschitz functions $f_i$ on $X$, where one of them has a compact support. 
Note that $T_{\omega}$ is a locally integral metric current with $\partial T_{\omega} =0$ in the sense of \cite{Lang, LW} if $X$ is isometric to a $k$-dimensional smooth Riemannian manifold with $\meas =\mathcal{H}^k$. 
However for general Ricci limit spaces, $T_{\omega}$ is not an integral current. 
For example, the space  $\left( [0, \pi], \frac{1}{2}\int \sin t \dist t \right)$ in example 2 above,  for any $c \in \mathbf{R}_{>0}$, $cT_{\dist t}$ is not integral current, but it is a metric current.

Recall that it was proven in \cite{CheegerColding1} that if a sequence of $n$-dimensional Riemannian manifolds $(Z_i, z_i)$ with $\mathrm{Ric}_{Z_i} \ge -(n-1)$ GH-converge to a metric space $(Z, z)$ and the sequence is \textit{noncollapsed} (i.e. $\liminf_{i \to \infty}\mathcal{H}^n(B_1(z_i))>0$ is satisfied), then it is also a mGH-convergent sequence with respect to the $n$-dimensional Hausdorff measure, that is, $(Z_i, z_i, \mathcal{H}^n) \stackrel{GH}{\to} (Z, z, \mathcal{H}^n)$ with $\mathcal{H}^n(B_1(z))>0$. Thus we always consider $n$-dimensional Hausdorff measures $\mathcal{H}^n$ instead of normalized one $\mathcal{H}^n/\mathcal{H}^n(B_1(z_i))$ as reference measures whenever the sequence of Riemannian manifolds is noncollpased.
\begin{theorem}[Relation to metric currents, Theorem \ref{thm:noncollapsed compactness}]\label{hhgg}
Let $(Z_i, z_i, \mathcal{H}^n)$ be a sequence of oriented $n$-dimensional Riemannian manifolds with $\mathrm{Ric}_{Z_i} \ge -(n-1)$ and their orientations $\omega_i \in C^{\infty}(\bigwedge^nT^*Z_i)$, let $(Z, z, \mathcal{H}^n)$ be the noncollapsed mGH-limit space with the orientation $\omega \in L^{\infty}(\bigwedge^nT^*Z)$ associated with $\omega_i$ (note that it makes sense by Theorem \ref{ah}).
Then we see that $T_{\omega}$ is a locally integral current, that the multiplicity of $T_{\omega}$ is one, that $\partial T_{\omega}=0$, and that $T_{\omega_i}$ converge to $T_{\omega}$ in the following sense;
\begin{equation}\label{hhhhooo}
\lim_{i \to \infty}T_{\omega_i}(f_{0, i}, f_{1, i}, \ldots, f_{n, i})=T_{\omega}(f_0, f_1, \ldots, f_n)
\end{equation}
whenever the following hold;
\begin{enumerate}
\item{(Uniform convergence)} $f_{j, i} \in \mathrm{LIP}_{\mathrm{loc}}(X_i)$ converge uniformly to $f_j \in \mathrm{LIP}_{\mathrm{loc}}(X)$ on each compact subset of $X$, where $\mathrm{LIP}_{\mathrm{loc}}$ denote the set of all locally Lipschitz functions,
\item{(Uniform Lipschitz bound)} Lipschitz constants $\mathbf{Lip}(f_{j ,i}|_{B_R(x_i)})$ of $f_{j, i}$ on $B_R(x_i)$ are uniformly bounded for any $R \in (0, \infty)$, i.e. $\sup_{i, j}\mathbf{Lip}(f_{j, i}|_{B_R(x_i)})<\infty$,
\item{(Uniform compact support)}  there exist $j$ and $R_0 \in (0, \infty)$ such that $\supp f_{j, i} \subset B_{R_0}(x_i)$ for any $i$.
\end{enumerate}
\end{theorem}
This theorem with a result established in \cite{GigliMondinoSavare} by Gigli-Mondino-Savar\'e gives a compatibility between the GH-convergence and the pointed flat compactness theorem given in \cite{LW}.
Moreover applying this to the compact case gives a new approach to prove the compatibility between the GH-convergence and 
the intrinsic flat convergence introduced in \cite{SormaniWenger} by Sormani-Wenger, which was known by Matveev-Portegies in \cite{MP}. Moreover our approach gives an explicit formula of the limit integral current by the limit orientation  as in the right hand side of (\ref{hhhhooo}).
See also Theorem \ref{comcom}.

Finally we will discuss dualities of (co) homology groups for singular spaces.
It is well-known that if a smooth compact $n$-dimensional manifold $M$ is orientable, then dualities between cohomology groups, $H^{n-k}(M) \cong H^k(M)$, hold for any $k$.
However in general we cannot expect such dualities for singular spaces.
In fact although $\mathbf{S}^0*\mathbb{C}P^2$ appears as the collapsed GH-limit of a sequence of Riemannian manifolds with uniform sectional curvature bounds from below (\cite[Example 1.2]{Y} by Yamaguchi) and it is oriented (in the sense of \cite{Mit}), $H^2(\mathbf{S}^0*\mathbb{C}P^2) \not \cong H^3(\mathbf{S}^0*\mathbb{C}P^2)$, where $\mathbf{S}^0*W$ is the spherical suspension of $W$. This observation is due to \cite{Mit} by Mitsuishi.

However we can prove dualities in a special case, which includes noncollapsed GH-limits of Einstein manifolds as typical examples:
\begin{theorem}[Duality, Theorems \ref{ww} and \ref{nnm}]\label{abc}
Let $X_i$ be a sequence of oriented $n$-dimensional compact Riemannian manifolds with $|\mathrm{Ric}_{X_i}| \le n-1$ and their orientations $\omega_i \in C^{\infty}(\bigwedge^nT^*X_i)$, and let $X$ be the noncollapsed compact GH-limit space with the orientation $\omega \in L^{\infty}(\bigwedge^nT^*X)$ associated with $\omega_i$ (recall that $(X_i, \mathcal{H}^n) \stackrel{GH}{\to} (X, \mathcal{H}^n)$).
Then we have the following dualities:
\begin{enumerate}
\item $H^n_{\mathrm{dR}}(X) \cong \mathbf{R}\omega (\cong \mathbf{R} \cong H^0_{\mathrm{dR}}(X))$, where $H^k_{\mathrm{dR}}$ is the $k$-dimensional de Rham cohomology group as $RCD$-spaces introduced in \cite{Gigli}.
\item $\mathrm{Harm}_1^{\infty}(\mathcal{R}(X)) \cong H^1_{\mathrm{dR}}(X) \cong H^{n-1}_{\mathrm{dR}}(X) \cong \mathrm{Harm}_{n-1}^{\infty}(\mathcal{R}(X))$, where $\mathrm{Harm}_k^{\infty}(\mathcal{R}(X))$ is the space of bounded weakly harmonic $k$-forms $\alpha$ on the regular set $\mathcal{R}(X)$ of $X$, i.e. $\|\alpha \|_{L^{\infty}}<\infty$, $\langle \alpha, \eta \rangle \in H^{1, 2}(X)$ and 
$$
\int_X\langle \dist \alpha, \dist \eta \rangle +\langle \delta \alpha, \delta \eta \rangle \dist \mathcal{H}^n =0
$$
are satisfied for any $\eta=f_0\dist f_1 \wedge \cdots \wedge \dist f_k$, where $f_0 \in \mathrm{LIP}_c(\mathcal{R}(X))$ and any $f_i \in \mathrm{Test}F(X) (i=1, 2, \ldots, k)$.
Moreover these are finite dimensional and an isomorphism $H^1_{\mathrm{dR}}(X) \cong H^{n-1}_{\mathrm{dR}}(X)$ is given by the Hodge star operator associated with $\omega$.
\end{enumerate}
\end{theorem}
In this theorem note that it is known in \cite{CheegerColding1} that $\mathcal{R}(X)$ is an open subset of $X$ and that it is a $C^{1, \alpha}$-Riemannian manifold for any $\alpha \in (0, 1)$. Therefore the second statement makes sense.
Moreover we can check that $\omega$ is a $C^{1, \alpha}$-harmonic form on $\mathcal{R}(X)$. See Remark \ref{hoho} and Corollary \ref{coco}.
\subsection{Organization of the paper}
Let us introduce key ideas to prove Theorem \ref{thm:unique}.
Although it does not coincide with the original proof, it might be helpful to understand that for readers.

Let $\omega_1, \omega_2$ be orientations of a Ricci limit space $(X, x, \meas)$.
Then since $\omega_i$ are top-dimensional differential forms, there exists a Borel function $f:X \to \{-1, 1\}$ such that $\omega_1=f\omega_2$. Our goal is to prove that $f$ is constant.
For that, roughly speaking we will establish \textit{the continuity of $f$ along the interior of any limit geodesic $\gamma$}.
Then combining the continuity with the segment inequality on $X$ proven in \cite{CheegerColding3} shows that $f$ is constant.

In order to prove the continuity of $f$ along the interior of $\gamma$ we will first prove that for any regular point $z$ of $X$, any $\epsilon \in (0, 1)$ and any limit harmonic function $\mathbf{b}$ defined on a neighborhood of $z$,
\begin{equation}\label{quantt}
\frac{r^2}{\meas (B_r(z))}\int_{B_r(z)}|\mathrm{Hess}_{\mathbf{b}}|^2\dist \meas<\epsilon
\end{equation} 
holds if $r$ is sufficiently small, where the hessian above is taken in the sense of \cite{Honda10}.
The key point is to give a \textit{quantitative} estimate of (\ref{quantt}) (Theorem \ref{thm:quantitative bound hess}), which is justified by using a blow-up argument and the behavior of the Laplacian with respect to the mGH-convergence discussed in \cite{AmbrosioHonda} by Ambrosio with the author, and in \cite{Honda2}.

Next we will prove the compatibility between the second-order differential culculus established in \cite{Gigli, Honda10}, which allows us to prove that $\omega_i$ are differentiable for a.e. $y \in X$ in the sense of \cite{Honda10} and to give a pointwise estimate;  
\begin{equation}\label{quantt2}
|\nabla \langle \omega_i, \dist \mathbf{b}_1 \wedge \cdots \wedge \dist \mathbf{b}_k \rangle| \le \sum_{l=1}^k|\mathrm{Hess}_{\mathbf{b}_l}| \prod_{j \neq l}^k(\mathbf{Lip}\mathbf{b}_j),
\end{equation} 
where $\mathbf{b}_l$ are limit harmonic functions. 
Then combining  (\ref{quantt}) with (\ref{quantt2}), the existence of good splitting functions established in \cite{CheegerColding1} and the Poincar\'e inequality shows
\begin{equation}\label{quantt3}
\frac{1}{\meas (B_r(z))}\int_{B_r(z)}\left| f - \frac{1}{\meas (B_r(z))}\int_{B_r(z)}f \dist \meas \right| \dist \meas<\epsilon.
\end{equation} 
This quantitative estimate (\ref{quantt3}) with the uniform Reifenberg property along the interior of $\gamma$ established in \cite{ColdingNaber} yields the continuity of $f$ along the interior of $\gamma$.

Note that precise arguments above will be done by a contradiction.

The organization of the paper is as follows.
In Section 2 we recall several results on Ricci limit spaces.
In Section 3 we establish compatibilities between $L^p$-convergence of tensor fields established in \cite{Honda2} and $L^p$-convergence of derivations established in \cite{AmbrosioStraTrevisan} by Ambrosio-Stra-Trevisan.
Note that in general these are not compatible (Remark \ref{diffe}).
The compatibilities we will establish (Propositions \ref{hj} and \ref{qqssxx}) allow us to use both tools given in \cite{AmbrosioHonda, Honda2}, which will play key roles in many situations (roughly speaking, \cite{AmbrosioHonda} is about global $L^p$-objects, \cite{Honda2} is about $L^p_{\mathrm{loc}}$-objects).
In Section 4 we  prove the uniqueness of second-order differential structure of (non-compact) Ricci limit spaces by using the heat flow. In the case when the limit space is compact, this was proven in \cite{Honda3} by using Poisson's equation.
In Section 5 we prove a quantitative estimate of (\ref{quantt}).
In Section 6 we start to discuss our orientability.
Section 7 covers the proof of Theorem \ref{abc}.
Moreover for any $l \in \{n-1, n\}$ we will prove spectral convergence of the Hodge and the connection Laplacians acting on $l$-dimensional differential forms.
 
\textbf{Acknowledgement.}
The author would like to express his appreciation to Luigi Ambrosio for his warm encouragement during the stay at SNS.
He also thanks SNS for its warm hospitality and for giving him nice environment.
He is grateful to Ayato Mitsuishi for informing us of \cite{Mit, Y}. 
I would like to thank Kota Hattori for discussing on \cite{DW}. 
He is also grateful to the referee for the very thorough reading and valuable suggestions. 
Moreover he wants to thank TSIMF, BICMR and MFO for their wonderful hospitality during his stays.
Finally he acknowledges supports of the JSPS Program for Advancing Strategic
International Networks to Accelerate the Circulation of Talented Researchers, of the Grant-in-Aid for challenging Exploratory Research 26610016 and of the Grantin-Aid
for Young Scientists (B) 16K17585.
\section{Preliminaries}
We here introduce two useful notaions;
\begin{enumerate}
\item for $a, b \in \mathbf{R}$ and $\epsilon \in (0, \infty)$, we write $a=b \pm \epsilon$ if $|a-b| \le \epsilon$,
\item any function $f:(\mathbf{R}_{>0})^{k+m} \to \mathbf{R}_{\ge 0}$, satisfying that
$$
\lim_{\epsilon_1, \ldots, \epsilon_k \to 0}f(\epsilon_1, \ldots, \epsilon_k, c_1, \ldots, c_m)=0
$$
for all fixed $c_1, \ldots, c_m \in \mathbf{R}$, is denoted by $\Psi(\epsilon_1, \ldots, \epsilon_k; c_1, \ldots, c_m)$ for simplicity.
\end{enumerate} 
\subsection{Gromov-Hausdorff convergence}
Let us denote the open (closed, respectively) ball centered at a point $x$ of a metric space $X$ with the radius $r$ by $B_r(x) (\overline{B}_r(x)$, respectively). We usually denote by $\dist$ or $\dist_X$ the distance function for simplicity.
We denote by $\mathrm{LIP}(X), \mathrm{LIP}_{\mathrm{loc}}(X)$ the sets of all Lipschitz functions on $X$, all locally Lipschitz functions on $X$, respectively. Moreover let us denote by $\mathrm{LIP}_c(X)$ the set of $f \in \mathrm{LIP}(X)$ whose supports are compact.
For any $f \in \mathrm{LIP}(X)$ let $\mathbf{Lip}f$ be the global Lipschitz constant, i.e.  $\mathbf{Lip}f:=\sup_{x \neq y}|f(x)-f(y)|/\dist (x, y)$.

For two pointed geodesic spaces $(X_i, x_i) (i=1, 2)$, we say that \textit{a map $\phi$ from $B_R(x_1)$ to $X_2$ is a (pointed) $\epsilon$-Gromov-Hausdorff approximation to $B_R(x_2)$} if it holds that $|\dist_{X_1} (x, y)-\dist_{X_2} (\phi (x), \phi (y))|<\epsilon$ for any $x, y \in B_R(x_1)$, that $\dist_{X_2}(\phi(x_1), x_2)<\epsilon$ and that $B_R(x_2) \subset B_{\epsilon}(\phi (B_R(x_1)))$, where $B_{\epsilon}(A)$ denotes the open $\epsilon$-neighborhood of a subset $A$.

Throughout the paper we mainly discuss \textit{proper geodesic metric spaces}. Recall that a metric space $X$ is said to be \textit{proper} if all bounded closed subset of $X$ is compact and that $X$ is \textit{geodesic} if for all $x, y \in X$ there exists an isometric embedding $\gamma: [0, \dist (x, y)] \to X$, called a geodesic from $x$ to $y$, with $\gamma (0)=x$ and $\gamma (\dist (x, y))=y$. Moreover a pair $(X, \meas)$ of such a metric space $X$ with a Borel measure $\meas$ on $X$ satisfying $\mathrm{supp}\,\meas =X$ is called a \textit{metric measure space} in the paper, where $\mathrm{supp}\,\meas$ denotes the support of $\meas$.

We say that a sequence of pointed metric measure spaces $(Y_i, y_i, \meas_i)$ \textit{measured Gromov-Hausdorff converge to} a pointed metric measure space $(Y, y, \meas)$ if there exist sequences of positive numbers $\epsilon_i \searrow 0$, $R_i \nearrow \infty$, and of $\epsilon_i$-Gromov-Hausdorff approximations $\phi_i$ from $B_{R_i}(y_i)$ to $B_{R_i}(y)$ such that 
$$
\lim_{i \to \infty}\int_Yf\dist (\phi_i)_{\sharp}\meas_i=\int_Yf\dist \meas
$$
for any continuous function $f$ on $Y$ with compact support, where $(\phi_i)_{\sharp}\meas_i$ denotes the push-forward measure of $\meas_i$ by $\phi_i$.
Then we denote by $(Y_i, y_i, \meas_i) \stackrel{GH}{\to} (Y, y, \meas)$ the convergence for simplicity.

Moreover for a sequence $\alpha_i \in Y_i$ and a point $\alpha \in Y$ we denote $\alpha_i \stackrel{GH}{\to} \alpha$ if $\lim_{i \to \infty}\dist_Y (\phi_i(\alpha_i), \alpha) =0$.
Note that 
$$
\lim_{i \to \infty}\meas_i (B_r(\alpha_i))=\meas (B_r(\alpha))
$$
for any $r \in (0, \infty)$ and any $\alpha_i \stackrel{GH}{\to} \alpha$ if the sequence of measures $\meas_i$ have a uniform local doubling constant, where this condition is satisfied by the Bishop-Gromov inequality in the Ricci limit setting as discussed below.
Note that we do not need to consider base points if spaces we discuss are compact metric spaces.
See \cite{BuragoBuragoIvanov, CheegerColding1, Fukaya, Gromov} for details.

We say that a pointed metric measure space $(Z, z, \meas)$ is \textit{an ($n$-) Ricci limit space} if there exist a sequence of pointed $n$-dimensional complete Riemannian manifolds $(Z_i, z_i)$ with $\mathrm{Ric}_{Z_i} \ge -(n-1)$ such that 
$$
\left(Z_i, z_i, \mathcal{H}^n/\mathcal{H}^n(B_1(z_i))\right) \stackrel{GH}{\to} (Z, z, \meas).
$$
\subsection{Structure of Ricci limit spaces}
Let $(X, x, \meas)$ be a Ricci limit space.
We say that \textit{a pointed metric measure space $(Y, y, \nu)$ is a tangent cone at $z \in X$} if there exists a sequence $\epsilon_i \searrow 0$ such that 
$$
\left(X, z, \epsilon_i^{-1}\dist, \frac{\meas }{\meas (B_{\epsilon_i}(z))}\right) \stackrel{GH}{\to} (Y, y, \nu).
$$
A point $z \in X$ is said to be \textit{$k$-dimensional reular} if every tangent cone at $z$ is isometric to $(\mathbf{R}^k, 0_k, \mathcal{H}^k/\mathcal{H}^k(B_1(0_k)))$. 
Let us denote  by $\mathcal{R}^k(X)$ the set of $k$-dimensional regular points in $X$ and let $\mathcal{R}(X):=\bigcup_{1 \le k \le n}\mathcal{R}^k(X)$. As written below \textit{the dimension $\mathrm{dim}\,X$ of $(X, \dist, \meas)$} is defined by a unique $k$ such that $\meas (\mathcal{R}_k)>0$.
\begin{theorem}[Cheeger-Colding, Colding-Naber, \cite{CheegerColding1, CheegerColding3, ColdingNaber}]\label{thm:reg}
We have the following.
\begin{enumerate}
\item $\meas \left( X \setminus \mathcal{R}(X) \right)=0$.
\item Let us denote by $\mathcal{R}^k_{\tau, \delta}(X)$ the set of $z \in X$ such that 
$$
\dist_{GH}\left( (B_s(z), z), (B_s(0_k), 0_k) \right) \le \tau s
$$
for any $s \in (0, \delta ]$, where $\dist_{GH}$ is the Gromov-Hausdorff distance. Then $\mathcal{R}^k(X)=\bigcap_{\tau>0} \bigcup_{\delta >0}\mathcal{R}^k_{\tau, \delta}(X)$.
\item  There exists a unique $k$ such that $\meas \left(X \setminus \mathcal{R}^k(X)\right)=0$.
We call $k$ \textit{the dimension of $X$} and denote it $\mathrm{dim}\,X$. 
\item $X$ is (strong) $\meas$-rectifiable, i.e. there exist a countable family of Borel subsets $C_i$ of $\mathcal{R}^k(X)$ and a countable family of bi-Lipschitz embeddings $\phi_i:C_i \hookrightarrow \mathbf{R}^k$ such that $\meas (X \setminus \bigcup_iC_i)=0$ and that for any $z \in \bigcup_iC_i$ and any $\epsilon \in (0, 1)$ there exists $i$ such that $z \in C_i$ and that $\phi_i$ is a $(1 \pm \epsilon)$-bi-Lipschitz. We call $\{(C_i, \phi_i)\}_i$ a rectifiable atlas of $(X, x, \meas)$.
\end{enumerate}
\end{theorem}
By the rectifiability above, the Jacobi matrix $J(\phi_i \circ (\phi_j)^{-1})(y)$ is well-defined for a.e. $y \in \phi_j(C_i \cap C_j)$.
Using this property, the tangent bundle $TX$, more generally, the tensor bundles $T^r_sX:\bigotimes_{i=1}^rTX \otimes \bigotimes_{i=r+1}^{r+s}T^*X$ are constructed. Note that each fiber is well-defined for a.e. $z \in X$ only. 
Their important properties include; 
\begin{enumerate}
\item[$(\star)$] each fiber has the canonical inner product $\langle \cdot, \cdot \rangle$; 
\item[$(*)$] for any Sobolev function $f \in H^{1, p}(U)$ on an open subset $U$ of $X$ (see below for the definition) there exists the differential $\dist f(y) \in T_x^*X$ for a.e. $y \in U$ such that $\|f\|_{H^{1, p}}^p=\|f\|_{L^p}^p+\|\dist f\|_{L^p}^p$.
\end{enumerate}
Moreover if $f \in \mathrm{LIP}_{\mathrm{loc}}(U) \cap H^{1, p}(U)$, then 
\begin{equation}\label{new}
|\dist f|(z):=\sqrt{\langle \dist f(z), \dist f(z)\rangle}=\limsup_{y \to z}\frac{|f(y)-f(z)|}{\dist (y, z)} =:\mathrm{Lip} f(z)
\end{equation}
for a.e. $z \in U$.
Sometimes we denote by $g_X$ the metric of $TX$ and call it the \textit{Riemannian metric of} $(X, \dist, \meas)$. 
A Borel measurable function $f$ on a Borel subset $A$ (denoted by $f \in \Gamma_0(A)$ for short) of $X$ is said to be \textit{differentiable for a.e. $z \in A$} if there exists a countable family of Borel subsets $A_i$ of $A$ such that $\meas (A \setminus \bigcup_iA_i)=0$ and that each restriction $f|_{A_i}$ is Lipschitz.
Let us denote by $\Gamma_1(A)$ the set of such functions. Note that for any $f \in \Gamma_1(A)$ there exist canonical sections $\nabla f(z) \in T_zX, \dist f(z) \in T^*_zX$ for a.e. $z \in A$.

We are now in a position to introduce a second-order differential structure of $(X, x, \meas)$ given in \cite{Honda10}. 
A rectifiable atlas $\{(C_i, \phi_i)\}_i$ is said to be an \textit{(weakly) second-order differential structure of} $(X, x, \meas)$ if each coefficient of the Jacobi matrix $J(\phi_i \circ (\phi_j)^{-1})$ of each transition map $\phi_i \circ (\phi_j)^{-1}$ is in $\Gamma_1(\phi_j(C_i \cap C_j))$ whenever $\mathcal{H}^k(\phi_j(C_i \cap C_j))>0$.
\begin{theorem}\cite{Honda10}
We have the following:
\begin{enumerate}
\item Assume that a rectifiable atlas $\{(C_i, \phi_i)\}_i$ satisfies that there exist $p \in (1, \infty)$, balls $B_{r_i}(y_i)$ and Lipschitz functions $\hat{\phi}_{i, j}$ on $B_{r_i}(y_i)$ such that $C_i \subset B_{r_i}(y_i)$ and that the map $\hat{\phi}_i:=(\hat{\phi}_{i, 1}, \ldots, \hat{\phi}_{i, k})$ coincides with $\phi_i$ on $C_i$ and that for any $i, l$ with $\meas (C_i \cap C_l) >0$, $\langle \nabla \hat{\phi}_{i, j}, \nabla \hat{\phi}_{l, m} \rangle \in H^{1, p}(B_{r_i}(y_i) \cap B_{r_l}(y_m))$. Then the rectifiable atlas is a weakly second-order differential structure of $(X, \meas)$.
\item There exists a rectifiable atlas satisfying the assumption stated in (1). In particular there exists a weakly second order differential structure of $(X, \dist, \meas)$. More precisely we can take each $\hat{\phi}_i$ as a limit harmonic map. 
\end{enumerate}
\end{theorem}
It will be proven later that the second-order differential structure stated in (2) above is canonical. See Proposition \ref{hondahonda}.

We fix a second-order differential structure $\{(C_i, \phi_i)\}$. Then using this second order differential structure, we establish a second-order differential calculus on $(X, \meas)$. In particular the Levi-Civita connection $\nabla^{g_X}$ is well-defined.
In order to explain it more precisely, 
a Borel measurable vector field $V$ on $A$ (denoted by $V \in \Gamma_0(TA)$, for short) is said to be \textit{differentiable for a.e. $z \in A$} if each coefficient of $V$ expressed by each local patch $(C_j, \phi_j)$ is in $\Gamma_1(A \cap C_j)$.
Let us denote by $\Gamma_1(TA)$ the set of such vector fields.
Similarly the set $\Gamma_1(T^r_sA)$ of Borel measurable tensor fields of type $(r, s)$ on $A$, which are differentiable for a.e. $z \in A$, is well-defined. 

Then one of the main results in \cite{Honda10} is the following.
\begin{theorem}\cite{Honda10}
There exists a unique multi-linear map $\nabla^{g_X}: \Gamma_0(TA) \times \Gamma_1(TA) \to \Gamma_0(TA)$ such that the following hold (we use the standard notation; $\nabla^{g_X}_VW:=\nabla^{g_X}(V, W)$).
\begin{enumerate}
\item $\nabla^{g_X}_{f_1V_1+f_2V_2}W=f_1\nabla^{g_X}_{V_1}W+f_2\nabla^{g_X}_{V_2}W$ for any $V_i \in \Gamma_0(TA)$, any $f_i \in \Gamma_0(A)$ and any $W \in \Gamma_1(TA)$.
\item $\nabla^{g_X}_V(gW)=V(g)W+g\nabla^{g_X}_VW$ for any $V \in \Gamma_0(TA)$, any $g \in \Gamma_1(A)$, and any $W \in \Gamma_1(TA)$.
\item $\nabla^{g_X}_VW-\nabla^{g_X}_WV=[V, W]$ for any $V, W \in \Gamma_1(TA)$.
\item $V\langle W, Z \rangle =\langle \nabla^{g_X}_VW, Z\rangle +\langle W, \nabla^{g_X}_VZ\rangle$ for any $V \in \Gamma_0(TA)$ and any $W, Z \in \Gamma_1(TA)$.
\end{enumerate}
\end{theorem}
Moreover using the Levi-Civita connection with the standard way in Riemannian geometry allows us to define the covariant derivative $\nabla^{g_X}T \in \Gamma_0(T^r_{s+1}A)$ of $T \in \Gamma_1(T^r_sA)$ by satisfying that
\begin{align}\label{huuh}
&\left\langle \nabla^{g_X}T, \bigotimes_{i=1}^rV_i \otimes \bigotimes_{j=1}^{s+1}\omega_j\right\rangle \nonumber \\
&=\omega_{s+1}^*\left(\left\langle T, \bigotimes_{i=1}^rV_i \otimes \bigotimes_{j=1}^{s}\omega_j \right\rangle \right) \nonumber \\
&-\sum_{i=1}^r\left\langle T, V_1 \otimes \cdots \otimes V_{i-1} \otimes \nabla^{g_X}_{\omega_{s+1}^*}V_i \otimes V_{i+1} \otimes \cdots \otimes V_{r} \otimes \bigotimes_{j=1}^{s}\omega_j \right\rangle \nonumber \\
&-\sum_{j=1}^{s}\left\langle T, \bigotimes_{i=1}^rV_i \otimes \omega_1 \otimes \cdots \otimes \omega_{j-1} \otimes \left(\nabla^{g_X}_{\omega_{s+1}^*}\omega_j^*\right)^* \otimes \omega_{j+1} \otimes \cdots \otimes \omega_{s}\right\rangle \nonumber
\end{align}
for any $V_i \in \Gamma_1(TA)$ and $\omega_j \in \Gamma_1(T^*A)$.
Then it was proven in \cite{Honda10} that $g_X \in \Gamma_1(T^0_2X)$ with $\nabla^{g_X}g_X\equiv 0$.
A function $f \in \Gamma_1(A)$ is said to be \textit{weakly twice differentiable} (denoted by $f \in \Gamma_2(A)$ for short) if $\dist f \in \Gamma_1(T^*A)$.
Then we can define the \textit{geometric Hessian of} $f$ by $\mathrm{Hess}_f^{g_X}:=\nabla^{g_X} \dist f$ and the \textit{geometric Laplacian of} $f$ by $\Delta^{g_X}f:=-\mathrm{tr} (\mathrm{Hess}_f^{g_X})$.
By a direct calulation  for any $\omega \in \Gamma_1(\bigwedge^kT^*A)$ and any $f_i \in \Gamma_2(A)$ it is easy to check the inequality;
\begin{equation}\label{eq:fundamental}
|\nabla \langle \omega, \dist f_1 \wedge \cdots \wedge \dist f_k\rangle |(z) \le |\nabla^{g_X} \omega|(z) \prod_i |\dist f_i|(z) + \sum_i |\omega| |\mathrm{Hess}^{g_X}_{f_i}|(z)\prod_{j \neq i}|\dist f_j|(z)
\end{equation}
for a.e. $z \in A$, which will play a role in the paper.
Note that for any $\eta \in \Gamma_0(\bigwedge^lT^*Z)$, $\eta \in \Gamma_1(\bigwedge^lT^*A)$ if and only if $\langle \eta, \dist \phi_{j, i_1} \wedge \cdots \wedge \dist \phi_{j, i_l} \rangle \in \Gamma_1(A \cap C_j)$ for any $j$ and any $i_1, \ldots, i_l \in \{1, \ldots, k\}$ 
(similar statement for tensor fields also holds. In particular for any $g \in \Gamma_0(A)$, $g \in \Gamma_2(A)$ holds if and only if $g \in \Gamma_1(A)$ and $\langle \dist g, \dist \phi_{j, i} \rangle \in \Gamma_1(A \cap C_j)$ hold for all $j, i$). 

On the other hand Gigli established in \cite{Gigli} second-order differential calculus on $RCD$-spaces based on the regularity theory of the heat flow based on \cite{AmbrosioGigliSavare13, AmbrosioGigliSavare14, AmbrosioGigliMondinoRajala}.
It was proven in \cite{Honda3} that Gigli's second order differential structure and the above one are compatible on compact Ricci limit spaces.
We will generalize this compatibility to general Ricci limit spaces by using tools given in \cite{AmbrosioHonda} (Proposition \ref{prop:comp hess}).

Let us define the Sobolev spaces in the Ricci limit setting (see for instance \cite{Cheeger, Shanmugalingam, giglimemo} for more general setting).
For an open subset $U$ of $X$ and any $p \in (1, \infty)$ we define the Sobolev space $H^{1, p}(U)$ as the completion of the space of $f \in \mathrm{LIP}_{\mathrm{loc}}(U)$ satisfying $\|f\|_{L^p(U)} +\|\mathrm{Lip}f\|_{L^p(U)}<\infty$ with respect to the norm
$\|f\|_{H^{1, p}(U)}:=(\|f\|^p_{L^p(U)}+\|\mathrm{Lip}f\|_{L^p(U)}^p)^{1/p}$.
As stated in $(*)$ recall that if $f \in H^{1, p}(U)$, then $f \in \Gamma_1(U)$ with $\|f\|_{H^{1, p}(U)}=(\|f\|_{L^p(U)}^p+\|\dist f\|_{L^p(T^*U)}^p)^{1/p}$.
Let us denote by $\mathcal{D}(\Delta^{\meas}, U)$ the set of $f \in H^{1, 2}(U)$ such that there exists a (unique) $g \in L^2(U)$ such that 
$$
\int_U\langle \dist f, \dist h\rangle \dist \meas = \int_U gf\dist \meas
$$ 
for any $h \in \mathrm{LIP}_c(U)$. Then put $\Delta^{\meas} f:=g$ and call it the \textit{Dirichlet Laplacian of} $f$.
Sometimes we denote by $\Delta$ instead of $\Delta^{\meas}$ for simplicity.
See for instance section 2 of \cite{Honda2} for details in this subsection.

Finally we discuss noncollapsed Ricci limit spaces:
\begin{theorem}[Cheeger-Colding \cite{CheegerColding1}]\label{thm:noncoll}
Let $(X, x, \meas)$ be an ($n$-) Ricci limit space. Then the following five conditions are equivalent;
\begin{enumerate}
\item $\mathcal{R}^n(X) \neq \emptyset$.
\item $\mathcal{R}^k(X) = \emptyset$ for any $k < n$.
\item $\mathrm{dim}\,X=n$.
\item $\mathrm{dim}_{\mathcal{H}}\,X=n$, where $\mathrm{dim}_{\mathcal{H}}$ is the Hausdorff dimension.
\item $\meas=\mathcal{H}^n/\mathcal{H}^n(B_1(x))$
\end{enumerate}
We say that \textit{$(X, x, \meas)$ is a noncollapsed Ricci limit space} if these conditions are satisfied.
\end{theorem}
\section{$L^p$-convergence}
In this section we discuss $L^p$-convergence for functions, vector fields, and more generally, for tensor fields with respect to the mGH-convergence.
These were already discussed in \cite{AmbrosioHonda, AmbrosioStraTrevisan, Honda2}. 
More precisely, \cite{AmbrosioHonda, AmbrosioStraTrevisan} are on $RCD(K, \infty)$-spaces for global $L^p$-objects (i.e. $R=\infty$) by using the regularity theory of the heat flow and isometric embeddings to a common metric space, and \cite{Honda2} is on Ricci limit spaces  for local $L^p$-objects (i.e. $R<\infty$) with no use of such isometric embeddings.
By using a result in \cite{GigliMondinoSavare} and tools on each setting in  \cite{AmbrosioHonda, AmbrosioStraTrevisan, Honda2}, we will show several compatibilities, which play key roles in the paper.

In order to introduce precise statements, let us fix our setting as follows.
Let $(X_i, x_i, \meas_i) \stackrel{GH}{\to} (X, x, \meas)$ be a mGH-convergent sequence of Ricci limit spaces.
By the equivalence between mGH-convergence and pmG-convergence established in \cite[Theorem 3.15]{GigliMondinoSavare}, with no loss of generality we can assume that the mGH-convergence is given by isometric embeddings to a common complete separable metric space $\mathbb{X}$, i.e. there exist isometric embeddings $\psi_i: X_i \hookrightarrow \mathbb{X}$, $\psi:X \hookrightarrow \mathbb{X}$ such that $\psi_i(x_i) \to \psi (x)$ in $\mathbb{X}$ and that $(\psi_i)_{\sharp}\meas_i$ weakly converge to  $(\psi )_{\sharp}\meas$ in duality with $C_\bs(\mathbb{X})$ which is the set of all continuous functions on $\mathbb{X}$ with bounded supports, i.e.
\begin{equation}\label{i}
\lim_{i \to \infty}\int_{\mathbb{X}}\phi\dist (\psi_i)_{\sharp}\meas_i=\int_{\mathbb{X}}\phi\dist (\psi)_{\sharp}\meas
\end{equation} 
for any $\phi \in C_\bs(\mathbb{X})$.
For simplicity we identify $(X_i, x_i, \meas_i)$ with the image by $\psi_i$, i.e. $(X_i, x_i, \meas_i)=(\mathbb{X}, \psi_i(x_i), (\psi_i)_{\sharp}\meas_i)$.
Note that this identification allows us to write the convergence (\ref{i}) by
\begin{equation}
\lim_{i \to \infty}\int_{X_i}\phi \dist \meas_i=\int_X\phi \dist \meas.
\end{equation}
Let $p \in (1, \infty)$ and let us denote by $L^p(T^r_sA)$ be the set of $L^p$-tensor fields of type $(r, s)$ on a Borel subset $A$.
We first discuss the case of functions.
\subsection{Compatibility in the case of functions}
\begin{definition}[$L^p$-convergence of functions by \cite{AmbrosioStraTrevisan}]
We say that a sequence $f_i \in L^p(X_i)$ \textit{$L^p$-weakly converge to $f \in L^p(X)$ in the sense of \cite{AmbrosioStraTrevisan}} if $\sup_i\|f_i\|_{L^p}<\infty$ and $f_i\meas_i$ weakly converge to $f\meas$ in duality with $C_{bs}(\mathbb{X})$, i.e. 
\begin{equation}\label{rrt2}
\lim_{i \to \infty}\int_{X_i}\phi f_i\dist \meas_i=\int_X\phi f\dist \meas
\end{equation}
for any $\phi \in C_\bs(\mathbb{X})$.
Moreover we say that $f_i$ \textit{$L^p$-strongly converge to $f$ in the sense of \cite{AmbrosioStraTrevisan}} if it is an $L^p$-weak convergent sequence to $f$ with $\limsup_{i \to \infty}\|f_i\|_{L^p}\le \|f\|_{L^p}$. 
\end{definition}
\begin{definition}[$L^p$-convergence of functions by \cite{Honda2}]
Let $R \in (0, \infty)$.
We say that a sequence $f_i \in L^p(B_R(x_i))$ \textit{$L^p$-weakly converge to $f \in L^p(B_R(x))$ in the sense of \cite{Honda2}} if $\sup_i\|f_i\|_{L^p}<\infty$ and 
\begin{equation}\label{rrt}
\lim_{i \to \infty}\int_{B_r(y_i)}f_i\dist \meas_i=\int_{B_r(y)}f\dist \meas
\end{equation}
for any $y_i \in B_R(x_i) \stackrel{GH}{\to} y \in B_R(x)$ and any $r \in (0, \infty)$ with $\overline{B}_r(y) \subset B_R(x)$.
Moreover we say that $f_i$ \textit{$L^p$-strongly converge to $f$ in the sense of \cite{Honda2}} if it is an $L^p$-weak convergent sequence to $f$ with $\limsup_{i \to \infty}\|f_i\|_{L^p}\le \|f\|_{L^p}$. 
\end{definition}
Note that it was proven in \cite{Honda2} that this definition is equivalent to that in \cite{KS} by Kuwae-Shioya.
We check the compatibility between definitions above in the case when $p=2$ only, which is enough in the paper.
\begin{proposition}[Compatibility in the case of functions]\label{hj}
We have the following.
\begin{enumerate}
\item Let $f_i \in L^2(X_i)$ be an $L^2$-weak (or strong, respectively) convergent sequence to $f \in L^2(X)$ in the sense of \cite{AmbrosioStraTrevisan}.
Then $f_i$ $L^2$-weakly (or strongly, respectively) converge to $f$ on $B_R(x)$ for any $R \in (0, \infty)$ in the sense of \cite{Honda2}.
\item Let $R \in (0, \infty)$ and let  $f_i \in L^2(B_R(x_i))$ be an $L^2$-weak (or strong, respectively) convergent sequence to $f \in L^2(B_R(x))$ in the sense of \cite{Honda2}.
Then letting $f_i \equiv 0$ outside $B_R(x_i)$, $f_i$ $L^2$-weakly (or strongly, respectively) converge to $f$ in the sense of \cite{AmbrosioStraTrevisan}. 
\end{enumerate}
\end{proposition}
\begin{proof}
Let us check (1). 
Assume that $f_i$ is an $L^2$-weak convergent sequence in the sense of \cite{AmbrosioStraTrevisan}.
Let $y_i \in X_i \stackrel{GH}{\to} y \in X$ and let $r \in (0, \infty)$.
Then since it is easy to check that $1_{B_r(y_i)}$ $L^q$-strongly converge to $1_{B_r(y)}$ for any $q \in (1, \infty)$ in the sense of \cite{AmbrosioStraTrevisan}, \cite[(2.6)]{AmbrosioStraTrevisan} shows 
$$
\lim_{i \to \infty}\int_{X_i}f_i1_{B_r(y_i)}\dist \meas_i=\int_Xf1_{B_r(y)}\dist \meas
$$
which proves (\ref{rrt}). Thus $f_i$ $L^2$-weakly converge to $f$ on $B_R(x)$ in the sense of \cite{Honda2}.
Moreover if $f_i$ is an $L^2$-strong convergent sequence in the sense of \cite{AmbrosioStraTrevisan}, then since it is easy to check that $f_i1_{B_R(x_i)}$ is an $L^2$-weak convergent sequence to $f1_{B_R(x)}$ in the sense of \cite{AmbrosioStraTrevisan}, applying \cite[(2.6)]{AmbrosioStraTrevisan} again shows
$$
\lim_{i \to \infty}\int_{X_i}f_i1_{B_R(x_i)} \cdot f_i \dist \meas_i=\int_Xf1_{B_R(x)} \cdot f\dist \meas,
$$ 
which proves the $L^2$-strong convergence of $f_i$ in the sense of \cite{Honda2}.

Next we prove (2).
Assume that $f_i$ is an $L^2$-weak convergent sequence on $B_R(x)$ in the sense of \cite{Honda2}.
Let $\phi \in C_\bs(\mathbb{X})$ and let $\phi_i=\phi|_{X_i} \in C_\bs(X_i)$.
Then since $\phi_i$ converge uniformly to $\phi$, in particular it is an $L^q$-strong convergent sequence in the sense of \cite{Honda2} (see for instance \cite[Remark 3.8]{Honda2}).
Thus \cite[Proposition 3.27]{Honda2} yields
$$
\lim_{i \to \infty}\int_{X_i}f_i\phi_i\dist \meas_i=\int_Xf\phi \dist \meas,
$$
which proves (\ref{rrt2}). Thus $f_i$ $L^2$-weakly converge to $f$ in the sense of \cite{AmbrosioStraTrevisan}.
Similarly we have the remaining implications.
\end{proof}

\subsection{Compatibility in the case of gradient vector fields}
Let $D$ be a countable dense subset of $\mathbb{X}$ and let $\Algebra_\bs$ be the smallest set that consists of bounded Lipschitz functions on $X$ containing
\begin{equation}\label{eq:setD}
\min\{\dist(\cdot,x),k\}\quad\text{with $k\in\setQ\cap [0,\infty]$ and $x\in D$}
\end{equation}
which is a vector space over $\setQ$ and is stable under products and lattice operations.
It is a countable set and it depends only on the choice of the set $D$ (but this dependence will not be emphasized 
in our notation, since the metric space will mostly be fixed). 
Let $\Algebra_\bs$ be the subalgebra of functions with bounded support, and let $h_{\mathbb{Q}_{>0}}^{\meas}\Algebra_\bs:=\{h_t^{\meas}f; t \in \mathbb{Q}_{>0}, f \in \Algebra_\bs\}$, where $h_t^{\meas}$ (or $h_t$ for short) denotes the heat flow, i.e. it is the $L^2(X, \meas)$-gradient flow of the Cheeger energy (see for instance \cite{AmbrosioGigliSavare13, AmbrosioGigliSavare14}).
Note that $h_t^{\meas}f$ is a bounded Lipschitz function on the support $\supp \meas$ of $\meas$ for any $h_t^{\meas}f \in h_{\mathbb{Q}_{>0}}^{\meas}\Algebra_\bs$ because the Bakry-\'Emery estimate on a $\RCD(K, \infty)$-space $(Y, \dist, \meas)$ yields that (recall that $\mathrm{Test}F(Y)=\{ f \in \mathcal{D}(\Delta, Y) \cap \mathrm{LIP}(Y); \Delta f \in H^{1, 2}(Y)\}$) 
\begin{equation}\label{eq:bakry}
g \in L^2(Y) \cap L^{\infty}(Y) \Rightarrow h_tg \in \mathrm{Test}F(Y) \subset \mathcal{D}(\Delta, Y) \cap \mathrm{LIP}(Y) \quad \forall t>0,
\end{equation}
See \cite{AmbrosioGigliSavare14} for the proof (see also \cite{Savare}).

For each $h_t^{\meas}f$ we fix an extension of the function to a function in $\mathrm{LIP}_b(\mathbb{X})$ and also denote it by the same notation $h_t^{\meas}f$.
See also page 16 of \cite{AmbrosioHonda}.
\begin{definition}[$L^p$-convergence of vector fields by \cite{AmbrosioStraTrevisan}]
We say that a sequence $V_i \in L^p(TX_i)$ \textit{$L^p$-weakly converge to $V \in L^p(TX)$ in the sense of \cite{AmbrosioStraTrevisan}} if $\sup_i\|V_i\|_{L^p}<\infty$ and $\langle V_i, h_t^{\meas}f \rangle \meas_i$ weakly converge to $\langle V, \nabla h_t^{\meas}f\rangle \meas$ in duality with $C_\bs(\mathbb{X})$ for any $h_t^{\meas}f \in h_{\mathbb{Q}_{>0}}^{\meas}\Algebra_\bs$, i.e.
\begin{equation}\label{vvbb}
\lim_{i \to \infty}\int_{X_i}\phi\langle V_i, \nabla h_t^{\meas}f\rangle \dist \meas_i=\int_X\phi\langle V, \nabla h_t^{\meas}f\rangle \dist \meas
\end{equation}
for any $\phi \in C_\bs(\mathbb{X})$. 
Moreover we say that $V_i$ \textit{$L^p$-strongly converge to $V$ in the sense of \cite{AmbrosioStraTrevisan}} if it is an $L^p$-weak convergent sequence to $V$ with $\limsup_{i \to \infty}\|\langle V_i, \nabla h_t^{\meas}f\rangle \|_{L^p}\le \|\langle V, \nabla h^{\meas}_tf\rangle \|_{L^p}$ for any $h_t^{\meas}f \in h_{\mathbb{Q}_{>0}}^{\meas}\Algebra_\bs$.
\end{definition}
Let us use the following notation: $\nabla ^r_sF=\nabla F_1 \otimes \cdots \otimes \nabla F_r \otimes \dist F_{r+1} \otimes \cdots \otimes \dist F_{r+s}$ for $F=(F_1, \ldots, F_{r+s}) \in (\Gamma_1(A))^{r+s}$.
\begin{definition}[$L^p$-convergence of vector/tensor fields by \cite{Honda2}]
Let $R \in (0, \infty)$.
We say that a sequence $V_i \in L^p(T^r_sB_R(x_i))$ \textit{$L^p$-weakly converge to $V \in L^p(T^r_sB_R(x))$ in the sense of \cite{Honda2}} if $\sup_i\|V_i\|_{L^p}<\infty$ and
\begin{equation}\label{vvbb2}
\lim_{i \to \infty}\int_{B_r(y_i)}\langle V_i, \nabla^r_s \mathbf{r_{z_i}}\rangle \dist \meas_i=\int_{B_r(y)}\langle V, \nabla^r_s \mathbf{r_z}\rangle \dist \meas
\end{equation}
for any $z_i^j, y_i \in B_R(x_i) \stackrel{GH}{\to} z^j, y \in B_R(x)$, respectively, and any $r \in (0, \infty)$ with $\overline{B}_r(y) \subset B_R(x)$, where $\mathbf{r_{z_i}}:=(r_{z_i^1}, \ldots, r_{z_i^{r+s}})$, $\mathbf{r_z}:=(r_{z^1}, \ldots, r_{z^{r+s}})$ and $r_z$ is the distance function from $z$.
Moreover we say that $V_i$ \textit{$L^p$-strongly converge to $V$ in the sense of \cite{Honda2}} if it is an $L^p$-weak convergent sequence to $V$ with  $\limsup_{i \to \infty}\|V_i\|_{L^p}\le \|V\|_{L^p}$.
\end{definition}
\begin{remark}\label{diffe}
We give an example which shows that in general, the definitions above for general vector fields are \textit{not} equivalent.
Let us consider the following setting.
\begin{enumerate}
\item Let $r_i \downarrow 0$ and let $\mathbf{S}^1(r_i):=\{x \in \mathbf{R}^2; |x|=r_i\}$.
\item Define the complete separable metric $\dist$ on $Z:=\bigsqcup_i \mathbf{S}^1(r_i) \sqcup \{x_{\infty}\}$ by 
\[\dist(x, y):=
\begin{cases} 2\pi r_i+2\pi r_j \,\,\,\,\,\mathrm{if}\,x \in \mathbf{S}^1(r_i), y \in \mathbf{S}^1(r_j),\, i \neq j,\\
\dist_{\mathbf{S}^1(r_i)}(x, y) \,\,\,\,\,\,\,\,\,\,\,\,\,\mathrm{if}\, x, y \in \mathbf{S}^1(r_i), \\
2\pi r_i \,\,\,\,\,\,\,\,\,\,\,\,\,\,\,\,\,\,\,\,\,\,\,\,\mathrm{if}\, x \in \mathbf{S}^1(r_i),  y=x_{\infty}, 
\end{cases}\]
where $\dist_{\mathbf{S}^1(r_i)}$ is the standard length distance on $\mathbf{S}^1(r_i)$. 
\item Let $(X_i, \meas_i):=(\mathbf{S}^1(1) \times \mathbf{S}^1(r_i), \dist_{\mathbf{S}^1(1) \times \mathbf{S}^1(r_i)}, \frac{1}{4\pi^2r_i}\mathcal{H}^2)$, where $\dist_{\mathbf{S}^1(1) \times \mathbf{S}^1(r_i)}$ on $\mathbf{S}^1(1) \times \mathbf{S}^1(r_i)$ is the product distance, and let $(X, \meas):=(\mathbf{S}^1(1), \frac{1}{2\pi}\mathcal{H}^1 )$.
\item Let $(\mathbb{X}, \dist_{\mathbb{X}} ):=(\mathbf{S}^1(1) \times Z, \dist_{\mathbf{S}^1(1) \times Z})$ and let $\phi_i, \phi$ be canonical isometric embeddings from $X_i, X$ to $\mathbb{X}$, respectively.
\item Let $\pi_i:X_i \to \mathbf{S}^1(r_i)$ be the canonical projection, let $\eta_i$ be an harmonic $1$-form on $\mathbf{S}^1(r_i)$ with $|\eta_i| \equiv 1$, and let $\omega_i=(\pi_i)^*\eta_i$ be the induced harmonic $1$-form on $X_i$.
\end{enumerate}
Then it is easy to check that the mGH-convergence $(X_i, \meas_i) \stackrel{GH}{\to} (X, \meas)$ is given by the isometric embeddings $\phi_i, \phi$ in the mannar of \cite{GigliMondinoSavare} and that $\omega_i$ $L^2$-weakly converge to $0$ in the sense of \cite{Honda2}, but it is not an $L^2$-strong convergence in the sense of \cite{Honda2}.

From now on we check that $\omega_i$ is an $L^2$-strong convergent sequence to $0$ in the sense of \cite{AmbrosioStraTrevisan} as follows. As mentioned previously we identify $(X_i, \meas_i)$ with the image by $\phi_i$.
Thus $\mathbb{X}=\bigsqcup_iX_i \sqcup X$.

For any $f \in \mathrm{LIP}(X)$ we take an extension of $f$ to a function $\phi_f \in \mathrm{LIP}(\mathbb{X})$ by $\phi_f(y, y_i):=f(y)$, where $(y, y_i) \in \mathbf{S}^1(1) \times \mathbf{S}^1(r_i)$.
Then by letting $f_i:=\phi_f|_{X_i} \in \mathrm{LIP}(X_i)$ it is easy to check $\lim_{i \to \infty}\|\nabla f_i\|_{L^2(X_i)}=\|\nabla f\|_{L^2}$, i.e. $f_i, \nabla f_i$ $L^2$-strongly converge to $f, \nabla f$ in the sense of \cite{Honda2}, respectively.
In particular by the Rellich compactness \cite[Theorem 4.9]{Honda2}, we see that $\langle \omega_i, \dist f_i\rangle$ $L^2$-strongly converge to $0$, which means that $\omega_i$ $L^2$-strongly converge to $0$ in the sense of \cite{AmbrosioStraTrevisan}.
\end{remark}
\begin{proposition}[Compatibility in the case of  $H^{1, 2}$-gradient vector fields]\label{qqssxx}
We have the following.
\begin{enumerate}
\item Let $f_i \in H^{1, 2}(X_i)$ be an $H^{1, 2}$-weakly convergent sequence to $f \in H^{1, 2}(X)$ in the sense of \cite{AmbrosioHonda, AmbrosioStraTrevisan}, i.e. $\sup_i\|f_i\|_{H^{1, 2}}<\infty$ and $f_i$ $L^2$-weakly converge to $f$ (note that it was proven that $f_i$ $L^2_{\mathrm{loc}}$-strongly converge to $f$. See \cite[Theorem 6.3]{GigliMondinoSavare} and \cite[Theorems 5.7 and 7.4]{AmbrosioHonda}).
Then $\nabla f_i$ $L^2$-weakly converge to $\nabla f$ on $B_R(x)$ for any $R \in (0, \infty)$ in the sense of \cite{Honda2}.
Moreover if $f_i$ $H^{1, 2}$-strongly converge to $f$ in the sense of \cite{AmbrosioHonda, AmbrosioStraTrevisan}, i.e. $\limsup_{i \to \infty}\|f_i\|_{H^{1, 2}}\le \|f\|_{H^{1, 2}}$, then we see that $\nabla f_i$ $L^2$-strongly converge to $\nabla f$ in the sense of \cite{AmbrosioStraTrevisan} and that $\nabla f_i$ $L^2$-strongly converge to $\nabla f$ on $B_R(x)$ for any $R \in (0, \infty)$ in the sense of \cite{Honda2}.
\item Let $R \in (0, \infty)$ and let $f_i \in H^{1, 2}(B_R(x_i))$ be an $H^{1, 2}$-strong convergent sequence to $f \in H^{1, 2}(B_R(x))$ in the sense of \cite{Honda2}, i.e. $f_i, \nabla f_i$ $L^2$-strongly converge to $f, \nabla f$ on $B_R(x)$, respectively.
Then letting $\nabla f_i \equiv 0$ outside $B_R(x_i)$, $\nabla f_i$ $L^2$-strongly converge to $\nabla f$ in the sense of \cite{AmbrosioStraTrevisan}. 
\end{enumerate}
\end{proposition}
\begin{proof}
Let us check (1).
The weak convergence of $\nabla f_i$ in the sense of \cite{Honda2} is a direct consequence of Proposition \ref{hj} and the Rellich compactness \cite[Theorem 4.9]{Honda2}.
Thus we assume that $f_i$ $H^{1, 2}$-strongly converge to $f$ in the sense of \cite{AmbrosioHonda, AmbrosioStraTrevisan}.
Then the $L^2$-strong convergence of $\nabla f_i$ in the sense of \cite{AmbrosioStraTrevisan} follows from \cite[Theorem 5.3]{AmbrosioStraTrevisan} (or \cite[Theorem 5.6]{AmbrosioHonda}).
Moreover the continuity of gradient operators \cite[Theorem 5.7]{AmbrosioHonda} yields that $|\nabla f_i|$ $L^2$-strongly converge to $|\nabla f|$.
In particular since $1_{B_r(y_i)}|\nabla f_i|$ $L^2$-strongly converge to $1_{B_r(y)}|\nabla f|$ whenever $y_i \stackrel{GH}{\to} y$ and $r \in (0, \infty)$, we have
$$
\lim_{i \to \infty}\int_{X_i}1_{B_r(y_i)}|\nabla f_i| \cdot |\nabla f_i|\dist \meas_i=\int_X1_{B_r(y)}|\nabla f| \cdot |\nabla f|\dist \meas,
$$
which proves that $\nabla f_i$ $L^2$-strongly converge to $\nabla f$ on $B_R(x)$ for any $R \in (0, \infty)$ in the sense of \cite{Honda2}.

Next we prove (2).
Let $\phi \in C_\bs(\mathbb{X})$ and let $\phi_i:=\phi|_{X_i} \in C_\bs(X_i)$.
Then since $\phi_i$ converge uniformly to $\phi$, in particular this is an $L^q$-strong convergent sequence for any $q \in (1, \infty)$ in the sense of \cite{Honda2}.
Thus \cite[Proposition 3.48]{Honda2} yields that $\phi_i\nabla f_i$ $L^p$-strongly converge to $\phi \nabla f$ on $B_R(x)$ in the sense of \cite{Honda2}.

On the other hand let $h_t^{\meas}g \in h_{\mathbb{Q}_{>0}}^{\meas}\Algebra_\bs$, let $h_i:=h_t^{\meas}g|_{X_i} \in \mathrm{LIP}(X_i)$ and let $h=h_t^{\meas}g|_X \in \mathrm{LIP}(X)$.
Then since $h_i$ converge uniformly to $h$ with $\sup_i\mathbf{Lip}h_i<\infty$, the Rellich compactness \cite[Theorem 4.9]{Honda2} shows that $\nabla h_i$ $L^q$-weakly converge to $\nabla h$ on $B_r(x)$ in the sense of \cite{Honda2} for any $r \in (0, \infty)$ and any $q \in (1, \infty)$.
In particular
$$
\lim_{i \to \infty}\int_{X_i}\langle \phi_i\nabla f_i, \nabla h_i \rangle \dist \meas_i=\int_X\langle \phi \nabla f, \nabla h \rangle \dist \meas,
$$
which proves that $\nabla f_i$ $L^2$-weakly converge to $\nabla f$ in the sense of \cite{AmbrosioStraTrevisan}.
Since it it trivial from the assumption that $\limsup_{i \to \infty}\|\nabla f_i\|_{L^2}\le \|\nabla f\|_{L^2}$,
this completes the proof.
\end{proof}
We often say that a sequence is \textit{$L^p_{\mathrm{loc}}$-strong convergent} if it is an $L^p$-strongly convergent sequence on $B_R(x)$ for any $R \in (0, \infty)$.
\section{Uniqueness of second-order differential structure}
\subsection{Rectifiability revisited}
In this subsection we recall several rectifiability results for Ricci limit spaces.
Note that these are not new, but we need precise statements later.

For a Ricci limit space $(X, x, \meas)$ whose dimension is $k$, a locally Lipschitz map $F=(f_1, \ldots, f_l):B_R(x) \to \mathbf{R}^l$ is said to be a {$(\delta, C)$-splitting map} if $|\nabla F| \le C$ and 
$$
\frac{1}{\meas (B_R(x))}\int_{B_R(x)}\left| \langle \nabla f_i, \nabla f_j\rangle -\delta_{ij}\right|\dist  \mathcal{H}^n<\delta
$$
are satisfied.
The following was a key result in Cheeger-Colding theory.
\begin{theorem}[Cheeger-Colding]\cite{CheegerColding}\label{arrr}
Let $(M, p)$ be a pointed $n$-dimensional complete Riemannian manifold with $\mathrm{Ric}_M \ge -\delta$.
If 
$$
\dist_{\mathrm{GH}}((B_L(p), p), (B_L(0_l), 0_l))<\epsilon,
$$
then there exists a harmonic $(\Psi(\epsilon, \delta, L^{-1};n, R), C(n))$-splitting map $\mathbf{b}:=(\mathbf{b}_1, \ldots, \mathbf{b}_l):B_R(p) \to \mathbf{R}^l$.
\end{theorem}
Moreover a map $F$ is said to be a \textit{limit harmonic map} if there exist
a sequence of Riemannian manifolds $(X_i,  x_i, \meas_i)$ with $\mathrm{Ric}_{X_i}\ge -(n-1)$, and a sequence of harmonic maps $F_i:B_R(x_i) \to \mathbf{R}^k$ such that $(X_i, x_i, \meas_i) \stackrel{GH}{\to} (X, x, \meas)$ and that $F_i$ converge uniformly to $F$.   
Note that the continuity of the Laplacian with respect to the mGH-convergence \cite[Theorem 1.3]{Honda2} yields that each $f_i$ is harmonic, i.e. $F$ is also a harmonic map. 
\begin{corollary}\label{cor:delta}
Let $\tau, s \in (0, 1)$ and let $y \in \mathcal{R}_{\tau, \delta}^k(X)$.
Then for any $s \in (0, \min \{\tau^{1/2}, \delta \}]$ there exists a limit harmonic $(\Psi(s, \tau;n), C(n))$-splitting map $\mathbf{b} :B_s(y) \to \mathbf{R}^k$.
\end{corollary}
\begin{proof}
By rescaling; $\dist \mapsto \tilde{\dist}:=(\tau)^{1/2}s\dist$ we have $\dist_{\mathrm{GH}}(B_{\tau^{-1/2}}^{\tilde{\dist}}(y), y), (B_{\tau^{-1/2}}(0_k), 0_k))<\tau^{1/2}$.
By using this with Theorem \ref{arrr} and the continuity of the Laplacian with resepct to the mGH-convergence  \cite[Theorem 1.3]{Honda2} it is easy to check the assertion.
\end{proof}
Let us recall the following.
\begin{theorem}\cite[Theorem 3.4]{Honda4}\label{g}
Let $A$ be a Borel subset of $X$, let $l \in \{1, 2, \ldots, k\}$ and let $\{f_i\}_{i=1, \ldots, l}$ be a family of Lipschitz functions on $X$.
Assume that $\det (\langle \nabla f_i, \nabla f_j \rangle (z))_{ij}>0$ for a.e. $z \in A$.
Then there exist a countable family of Borel subsets $A_i$ of $A$, a family of points $x_i \in A$, and a family of points $y_{i, j} \in X$ such that the following hold;
\begin{enumerate}
\item $\meas \left( A \setminus \bigcup_iA_i\right)=0$,
\item for any $z \in \bigcup_iA_i$ and any $\epsilon \in (0, 1)$ there exists $i$ such that $z \in A_i$ and that the map $\phi_i:A_i \to \mathbf{R}^k$ defined by
$$
\phi_i:=\left((f_1, \ldots, f_l)(\langle \nabla f_j, \nabla f_k\rangle (x_i))_{jk}^{-1/2}, \dist(y_{i, 1}, \cdot), \ldots, \dist (y_{i, k-l}, \cdot) \right)
$$
is a $(1 \pm \epsilon)$-bi-Lipschitz embedding.
\end{enumerate}
\end{theorem}
\begin{remark}
From the proof of Theorem \ref{g} we can take limit harmonic functions instead of distance functions $\dist (y_{i, j}, \cdot)$.
\end{remark}
Note that the following can be checked directly along the original proof of (4) of Theorem \ref{thm:reg} by Cheeger-Colding. However for reader's convenience, we give a sketch of the proof by using results above. This will play a key role in next subsection. 
See proofs of \cite[Theorems 5.5 and 5.7]{CheegerColding3}. 
\begin{theorem}[Cheeger-Colding \cite{CheegerColding3}]\label{thm:chc}
There exists a rectifiable structure $\{(C_i, \phi_i)\}_i$ of $(X, \meas)$ such that the following hold;
\begin{enumerate}
\item each $\phi_i$ is the restriction to $C_i$ of a limit harmonic map $\tilde{\phi}_i$ defined on a ball $B_{r_i}(y_i)$ which contains $C_i$ with $|\nabla \tilde{\phi}_i| \le C(n)$.
\item for any $z \in \bigcup_iC_i$ and any $\epsilon \in (0, 1)$ there exists $i$ such that $z \in C_i$, that $\phi_i$ is a $(1 \pm \epsilon)$-bi-Lipschitz embedding, that 
\begin{equation}\label{eq:rem}
\frac{1}{\meas (B_{r_i}(y_i))}\int_{B_{r_i}(y_i)}\left| \langle \nabla \tilde{\phi}_{i, j}, \nabla \tilde{\phi}_{i, k} \rangle -\delta_{jk}\right| \dist \meas<\epsilon
\end{equation}
and that
\begin{equation}\label{p}
\frac{\meas (C_i)}{\meas (B_{r_i}(y_i))} \ge 1- \epsilon.
\end{equation}
\end{enumerate}
\end{theorem}
\begin{proof}
Let $y \in \mathcal{R}^k_{\tau, \delta}$ and let $s \in (0, \min \{\tau^{1/2}, \delta\})$.
Then by Corollary \ref{cor:delta} there exists a limit harmonic $(\Psi (\delta, \tau;n), C(n))$-splitting map $\mathbf{b}=(\mathbf{b}_1, \ldots, \mathbf{b}_k)$ on $B_s(y)$.
For this $\Psi=\Psi (\delta, \tau;n)$ let $A_{i, j}:=\{w \in B_s(y); |\langle \nabla \mathbf{b}_i, \nabla \mathbf{b}_j\rangle (y)-\delta_{ij}|<(\Psi )^{1/2}\}$ and let $A:=\bigcap_{i, j}A_{i, j}$.
Then 
\begin{equation}\label{eq:const lem2}
\frac{\meas (B_s(y) \setminus A)}{\meas (B_s(y))} \le \sum_{i, j}\frac{1}{(\Psi)^{1/2}\meas (B_s(y))}\int_{A_{i, j}} \left|\langle \nabla \mathbf{b}_i, \nabla \mathbf{b}_j\rangle -\delta_{ij}\right|\dist \meas \le n^2 (\Psi )^{1/2}.
\end{equation}
Let $z \in \mathrm{Leb}\,(1_A)\cap \bigcap_{i, j} \mathrm{Leb}\,(\langle \nabla \mathbf{b}_i, \nabla \mathbf{b}_j \rangle)$, 
where
$\mathrm{Leb}\,(g):=\{w; \lim_{r \to \infty}\frac{1}{\meas (B_r(w))}\int_{B_r(w)}|g-g(w)|\dist \meas =0\}$ for a Borel measurable function $g$.
Then applying Theorem \ref{g} for $A \cap B_r(w)$ and $\mathbf{b}|_{A \cap B_r(w)}$ for any sufficiently small $r \in (0, 1)$ yields
that there exist a countable family of Borel subsets $A_i \subset A \cap B_r(w)$ and a family of points $x_i \in A \cap B_r(w)$ such that the following hold;
\begin{enumerate}
\item $\meas ((A \cap B_r(w)) \setminus \bigcup_iA_i)=0$; 
\item for any $\epsilon \in (0, 1)$ and any $z \in \bigcup_iA_i$ there exists $i$ such that $z \in A_i$ and that the map $\phi_i:A_i \to \mathbf{R}^k$ defined by 
$$
\phi_i = \left( (\mathbf{b}_1, \ldots, \mathbf{b}_k) (\langle \nabla \mathbf{b}_l, \nabla \mathbf{b}_m\rangle (x_i) )^{-1}_{lm}\right)
$$
is a $(1 \pm \epsilon)$-bi-Lipschitz embedding.
\end{enumerate}
Since $|(\langle \nabla \mathbf{b}_i, \nabla \mathbf{b}_j\rangle )_{ij}-(\delta_{ij})_{ij}|<n^2\Psi$ on $A$ and $\epsilon, \tau, \delta, r$ are arbitrary, we conclude.
\end{proof}

\subsection{The canonical second-order differential structure}
The main technical tool we will use in this subsection is the heat flow $h_t$ associated with the Laplacian $\Delta$.
See \cite{AmbrosioGigliSavare13, AmbrosioGigliSavare14, AmbrosioGigliMondinoRajala} for details of the regularity theory.

Let us recall the definition of the Hessian of a test function defined in \cite{Gigli} by Gigli only in the Ricci limit setting (note that the Hessian in the sense of \cite{Gigli} is well-defined on $RCD(K, \infty)$-spaces).

Let $(X, x, \meas)$ be a Ricci limit space and
let $W^{2, 2}(X)$ be the set of $f \in H^{1, 2}(X)$ satisfying that there exists a unique $T \in L^2(T^0_2X)$, denoted by $\mathrm{Hess}_f^{\meas}$, such that  
\begin{align*}\label{78}
2\int_{X}g_0\left\langle T, \dist g_1 \otimes \dist g_2 \right\rangle \dist \meas &=\int_X-\langle \nabla f, \nabla g_1 \rangle (\langle \nabla g_0, \nabla g_2\rangle -g_0\Delta g_2)\dist \meas \\
&-\int_X\langle \nabla f, \nabla g_2\rangle (\langle \nabla g_0, \nabla g_1\rangle -g_0\Delta g_1) \dist \meas \\
&- \int_Xg_0 \left\langle \nabla f, \nabla \left\langle \nabla g_1, \nabla g_2 \right\rangle\right\rangle \dist \meas
\end{align*}
for any $g_i \in \mathrm{Test}F(X)$ (recall $\mathrm{Test}F(X):=\{f \in \mathcal{D}^2(\Delta, X) \cap \mathrm{LIP}(X) \cap L^{\infty}(X); \Delta f \in H^{1, 2}(X)\}$).
Then it was proven in \cite{Gigli} that $\mathrm{Test}F(X) \subset \mathcal{D}^2(\Delta, X) \subset W^{2, 2}(X)$, that $W^{2, 2}(X)$ is a Hilbert space equipped with the norm $||f||_{W^{2, 2}}:= (||f||_{H^{1, 2}}^2+||\mathrm{Hess}^{\meas}_f||_{L^2}^2)^{1/2}$, and that 
$$
\mathrm{Hess}^{\meas}_f(\nabla g_1, \nabla g_2)=\frac{1}{2}\left(\langle \nabla g_1, \nabla \langle \nabla f, \nabla g_2 \rangle \rangle +\langle \nabla g_2, \nabla \langle \nabla f, \nabla g_2 \rangle \rangle -\langle \nabla f, \nabla \langle \nabla g_1, \nabla g_2 \rangle \rangle \right)
$$
for any $f, g_i \in \mathrm{Test}F(X)$ (\cite[Proposition 3.3.22]{Gigli}).
\begin{lemma}\label{aa}
For all $f, g \in \mathcal{D}(\Delta, B_R(x)) \cap \mathrm{LIP}_{\mathrm{loc}}(B_R(x))$ we have $\langle \nabla f, \nabla g \rangle \in H^{1, 2}(B_r(x))$ for any $r \in (0, R)$.
\end{lemma}
\begin{proof}
This is a direct consequence of the regularity theory of the heat flow as follows.
By using a good cutoff on a limit space (c.f. \cite[Corollary 4.29]{Honda2}) with no loss of generality we can assume that $f, g \in \mathcal{D}(\Delta, X) \cap \mathrm{LIP}_c(X)$.
Let us consider a function $\langle \nabla h_tf, \nabla h_tg \rangle$ for $t \in (0, 1)$.
Then since $h_tf, h_tg \in \mathrm{Test}F(X)$, by Bakry-\'Emery estimates and Bochner's inequality \cite[Corollary 3.3.9]{Gigli}, we have $\langle \nabla h_tf, \nabla h_tg \rangle \in H^{1, 2}(X)$ with
\begin{align*}
\|\nabla \langle \nabla h_tf, \nabla h_tg \rangle \|_{L^2(X)} &\le \|\nabla h_tf\|_{L^{\infty}(X)}\|\mathrm{Hess}_{h_tg}^{\meas}||_{L^2(X)}+\|\nabla h_tg\|_{L^{\infty}(X)}\|\mathrm{Hess}^{\meas}_{h_tf}\|_{L^2(X)}\\
& \le e^{t(n-1)}\|\nabla f\|_{L^{\infty}(X)}\left( \int_X\left( \left(\Delta h_tf\right)^2+(n-1)|\nabla h_tf|^2\right)\dist \meas\right)^{1/2} \\
&+ e^{t(n-1)}\|\nabla g\|_{L^{\infty}(X)}\left( \int_X\left( \left(\Delta h_tg\right)^2+(n-1)|\nabla h_tg|^2\right)\dist \meas\right)^{1/2}.
\end{align*}
Then since the right hand side above is bounded with respect $t \in (0, 1)$, letting $t \downarrow 0$ gives $\langle \nabla f, \nabla g \rangle \in H^{1, 2}(X)$, which completes the proof. 
\end{proof}
\begin{proposition}[Uniqueness of second-order differential structure]\label{hondahonda}
Let $\{(C_i, \phi_i)\}_i$ be a rectifiable structure of $(X, x, \meas)$. Assume that
\begin{itemize}
\item for any $\phi_{i}$ there exist $r_i \in (0, \infty)$, $x_i \in X$ and $\tilde{\phi}_{i, j} \in \mathcal{D}^2(\Delta, B_{r_i}(x_i)) \cap \mathrm{LIP}_{\mathrm{loc}}(B_{r_i}(x_i))$ such that $C_i \subset B_{r_i}(x_i)$ and that $\tilde{\phi}_{i, j}|_{C_i} \equiv \phi_{i, j}$, where $\phi_i=(\phi_{i, 1}, \ldots, \phi_{i, k})$.
\end{itemize}
Then  $\{(C_i, \phi_i)\}_i$ is a second-order differential structure of $(X, x, \meas)$. 
\end{proposition}
\begin{proof}
It is a direct consequence of Lemma \ref{aa},  \cite[Proposition 3.25]{Honda10} and a fact that all Sobolev functions are differentiable for a.e. as mentioned in subsection 2.2. See also \cite[Theorem 4.11]{Honda3}.
\end{proof}
We call $\{(C_i, \phi_i)\}_i$ as above a \textit{canonical second-order differential structure} and always consider it whenever we discuss second-order differential calculus.
\begin{theorem}[Second-order differential structure by test functions]\label{thm:can sec}
There exists a canonical second order differential structure $\{(C_i, \phi_i)\}_i$ of $(X, x, \meas)$ such that each $\phi_{i, j}$ is the restriction of a function $\tilde{\phi}_{i, j} \in \mathrm{Test}F(X)$ to $C_i$.
\end{theorem}
\begin{proof}
Since the proof is essentially same to that of Theorem \ref{thm:chc} (or \cite[Theorem 3.4]{Honda4}) we only give a sketch of that as follows.

Fix $\epsilon \in (0, 1)$ and choose a rectifiable patch $(C_i, \phi_i)$ such that $\phi_i$ is a $(1 \pm \epsilon)$-bi-Lipschitz embedding and that  each $\phi_{i}$ is the restriction to $C_i$ of a limit harmonic map $\tilde{\phi}_{i, j}$ defined on a ball $B_{r_i}(y_i)$ satisfying (\ref{eq:rem}) and (\ref{p}).
With no loss of generality we can assume that each $\tilde{\phi}_{i, j}$ is a restriction to $B_{r_i}(y_i)$ of a function $\psi_{i, j} \in \mathrm{LIP}_c(X)$.

Then for any sufficiently small $t \in (0, 1)$ since
$$
\frac{1}{\meas (B_{r_i}(y_i))}\int_{B_{r_i}(y_i)}\left| \langle \nabla h_t\psi_{i, j}, \nabla h_t\psi_{i, k} \rangle -\delta_{jk}\right|\dist \meas<\epsilon,
$$
if let $A_t:=\bigcap_{j, k}\{z \in B_{r_i}(y_i);|\langle \nabla h_t\psi_{i, j}, \nabla h_t\psi_{i, k}\rangle (z)-\delta_{ij}|<\epsilon^{1/2}\}$, then by an argument similar to (\ref{eq:const lem2}) for some $j, k$
$$
\frac{\meas (B_{r_i}(y_i) \setminus A_t)}{\meas (B_{r_i}(y_i))} \le n^2\epsilon^{1/2}.
$$
Then applying Theorem \ref{g} as $A=A_t$ and $f_j=\psi_{i, j}$ for sufficiently small $\epsilon, t$ completes the proof.
\end{proof}
The following is a direct consequence of Theorem \ref{g}.
\begin{corollary}\label{cor:com}
If a tensor $T$ of type $(r, s)$ on a Borel subset $A$ of $X$ satisfies that $\langle T, \nabla^r_sF \rangle \in \Gamma_1(A)$ for any $F:=(F_1, \ldots, F_{r+s}) \in (\mathrm{Test}F(X))^{r+s}$, then $T \in \Gamma_1(T^r_sA)$.
In particular if $T$ is defined on a ball $B_R(y)$ satisfying that for any $F$ as above, $\langle T, \nabla^r_sF \rangle \in H^{1, p}(B_R(y))$ holds for some $p \in (1, \infty)$, then $T \in \Gamma_1(T^r_sB_R(y))$.
\end{corollary}

\begin{proposition}[Compatibility between Hessians]\label{prop:comp hess}
Let $f \in \mathcal{D}(\Delta, X)$.
Then we have the following;
\begin{enumerate}
\item $f \in \Gamma_2(X)$, i.e. $f$ is twice diferentiable for a.e. $y \in X$ with $\mathrm{Hess}_f^{g_X}(y)=\mathrm{Hess}_f^{\meas}(y)$ for a.e. $y \in X$,
\item if $(X, x, \meas)$ is a noncollapsed Ricci limit space, then $-\mathrm{tr} (\mathrm{Hess}_f^{g_X})=\Delta f$.
In particular $\mathcal{D}(\Delta, X)=H^{2, 2}(X)$.
\end{enumerate}
\end{proposition}
\begin{proof}
Since the proofs are essentially same to that of \cite[Theorem 1.9]{Honda3}, we only give a skech of the proof of (1).

Let $(X_i, x_i, \meas_i)$ be an approximate sequence of $(X, x, \meas)$, i.e. it is a sequence of $n$-dimensional Riemannian manifolds with $\mathrm{Ric}_{X_i}\ge -(n-1)$ such that 
 $(X_i, x_i, \meas_i) \stackrel{GH}{\to} (X, x, \meas)$.
We take a sequence $g_i \in L^2(X_i)$ $L^2$-strongly converging to $f$ on $X$.
Then by the $H^{1, 2}$-strong convergence for the heat flow \cite[(4.6),  Corollary 5.5]{AmbrosioHonda}, for any $t \in (0, 1)$ we see that $h_tg_i$ strongly converge to $h_tf$ in $H^{1, 2}$ and that $\Delta h_tg_i$ $L^2$-weakly converge to $\Delta h_tf$ on $X$.
In particular by the continuity of the Laplacian with respect to the mGH-convergence \cite[Theorem 1.3]{Honda2} (c.f. \cite[Theorem 4.11]{Honda3}), for any $R \in (0, \infty)$ we see that $h_tf$ is twice differentiable for a.e. $z \in X$, that $\mathrm{Hess}_{h_tf_i}^{g_{X_i}}$ $L^2$-weakly converge to $\mathrm{Hess}_{h_tf}^{g_X}$ on $B_R(x)$ and that $|\nabla h_tf|^2 \in H^{1, p_n}(B_R(x))$ with 
\begin{equation}\label{r}
\|\mathrm{Hess}_{h_tf}^{g_X}\|_{L^2(B_R(x))}+\||\nabla h_tf|^2\|_{H^{1, p_n}(B_R(x))} \le C\left(n, R, \|f\|_{L^2(B_{2R}(x))}, \|\Delta f\|_{L^2(B_{2R}(x))}\right),
\end{equation}
where $p_n=2n/(2n-1)$.

On the other hand by the $L^2$-weak continuity of Hessians \cite[Theorem 10.3]{AmbrosioHonda}, whenever $g_i \in H^{1, 2}(X_i, \dist_i, \meas_i)$ are uniformly Lipschitz and strongly converge in $H^{1, 2}$ to $g \in H^{1, 2}(X, \dist, \meas)$, $\mathrm{Hess}_{f_i}^{\meas_i}(\nabla g_i, \nabla g_i)$ $L^2$-weakly converge to $\mathrm{Hess}_f^{\meas}(\nabla g, \nabla g)$.
Note that $\mathrm{Hess}_{h_tf_i}^{g_{X_i}}=\mathrm{Hess}_{h_tf_i}^{\meas_i}$ on $X_i$ because $f_i$ is smooth, and that for any $g \in \mathrm{LIP}(X) \cap H^{1, 2}(X)$ there exists an approximation $g_i$ of $g$ as above (c.f. \cite[(10.5), Theorem 10.2]{AmbrosioHonda}).
In particular $\mathrm{Hess}_{h_tf}^{g_X}(\nabla g, \nabla g)(z)=\mathrm{Hess}_{h_tf}^{\meas}(\nabla g, \nabla g)(z)$ for a.e. $z \in X$ for any $g \in \mathrm{LIP}(X) \cap H^{1, 2}(X)$.
From the density of $\mathrm{Test}T^0_2(X)$ in $L^2(T^0_2X)$, this shows $\mathrm{Hess}^{g_X}_{h_tf}(z)=\mathrm{Hess}^{\meas}_{h_tf}(z)$ for a.e. $z \in X$ for any $t \in (0, 1)$.

For any $g \in \mathrm{Test}F(X)$ by (\ref{r}) since $\langle \nabla h_tf, \nabla g \rangle$ is unifomly bounded in $H^{1, p_n}(B_R(x))$ with respect to $t \in (0, 1)$, letting $t \downarrow 0$ shows $\langle \nabla f, \nabla g \rangle \in H^{1, p_n}(B_R(x))$. In particular Corollary \ref{cor:com} yields that $f$ is twice differentiable for a.e. $z \in X$ and that $\mathrm{Hess}_f^{g_X}(\nabla g, \nabla g)=\langle \nabla g, \nabla \langle \nabla f, \nabla g \rangle \rangle -(1/2)\langle \nabla f, \nabla |\nabla g|^2\rangle =\mathrm{Hess}_f^m(\nabla g, \nabla g)$, which completes the proof of (1).  
\end{proof}
\section{Quantitative behavior of Hessians on regular sets}
\begin{lemma}\label{lem:conv hess}
Let us consider the following setting;
\begin{enumerate}
\item let $\delta_i \searrow 0$ be a convergent sequence of positive numbers, let $L_i \nearrow \infty$ be a divergent sequence of positive numbers,
\item let $(X_i, x_i)$ be a sequence of $n$-dimensional Riemannian manifolds with $\mathrm{Ric}_{X_i} \ge -\delta_i(n-1)$,  let $(X_i, x_i, \meas_i) \stackrel{GH}{\to} (\mathbf{R}^k, 0_n, \mathcal{H}^k/\mathcal{H}^k(B_1(0_k)))$, where $\meas_i=\mathcal{H}^n/\mathcal{H}^n(B_1(x_i))$, and 
\item  let $f_i$ be smooth functions on $B_{L_i}(x_i)$ with $\sup_i\|\nabla f_i\|_{L^{\infty}(B_{L_i}(x_i))}<\infty$ and 
$$
\lim_{i \to \infty}\|\Delta f_i\|_{L^2(B_r(x_i))}=0
$$
for any $r>0$.
\end{enumerate}
Then
$$
\lim_{i \to \infty}\|\mathrm{Hess}_{f_i}\|_{L^2(B_r(x_i))}=0
$$
for any $r>0$.
\end{lemma}
\begin{proof}
With no loss of generality we can assume that there exists $f \in \mathrm{LIP}_{\mathrm{loc}}(\mathbf{R}^k)$ such that $f_i$ converges uniformly to $f$ on each compact subset. 

Then the continuity of the Laplacian with respect to the mGH-convergence \cite[Theorem 1.3]{Honda2} yields that $f$ is harmonic on $\mathbf{R}^n$ and that $\nabla f_i$ $L^2$-converge strongly to $\nabla f$ on $B_r(0_n)$ for any $r>0$.
Since $\|\nabla f\|_{L^{\infty}(\mathbf{R}^k)}\le \sup_i\|\nabla f_i\|_{L^{\infty}(B_{L_i}(x_i))}<\infty$, $f$ is a linear function on $\mathbf{R}^k$.
In particular $|\nabla f|$ is constant on $\mathbf{R}^k$.
Thus the $L^2$-strong convergence of $\nabla f_i$ implies
$$
\lim_{i \to \infty}\| |\nabla f_i|^2- |\nabla f|^2 \|_{L^2(B_r(x_i))}=0
$$
for any $r>0$. 
We now take a sequence of good cut-off functions $\phi_i$ (see \cite[Theorem 6.33]{CheegerColding}), i.e.
for any $r>0$ there exists $\phi_i \in C^{\infty}(X_i)$ such that $0 \le \phi_i \le 1$, that $\phi_i|_{B_r(x_i)}\equiv 1$, that $\phi_i|_{X_i \setminus B_{2r}(x_i)} \equiv 0$, and that $|\nabla \phi_i| + |\Delta \phi_i| \le C(n, r)$.
Then since Bochner's formula yields
$$
-\frac{1}{2}\phi_i\Delta \left(|\nabla f_i|^2-|\nabla f|^2\right) \ge \phi_i|\mathrm{Hess}_{f_i}|^2 - \phi_i\langle \nabla \Delta f_i, \nabla f_i \rangle -\delta_i\phi_i|\nabla f_i|^2,
$$
integrating this on $B_{2r}(x_i)$ gives 
$$
\|\mathrm{Hess}_{f_i}\|_{L^2(B_r(x_i))}^2 \le \int_{B_{2r}(x_i)}\left( \frac{1}{2}|\Delta \phi_i|\left| |\nabla f_i|^2 - |\nabla f|^2 \right| + |\Delta f_i| \left| \mathrm{div} (\phi_i \nabla f_i )\right|  + \delta_i |\nabla f_i|^2 \right) \dist \meas_i.
$$
Letting $i \to \infty$ with the Cauthy-Schwarz inequality completes the proof.
\end{proof}
Let $\mathcal{R}_{\tau, \delta}(X):=\bigcup_{1 \le k \le n}\mathcal{R}^k_{\tau, \delta}(X)$.
The following is the main result in this section, which will be used for harmonic functions later.
\begin{theorem}[Quantitative behavior of Hessians on regular sets]\label{thm:quantitative bound hess}
For any $L \in [1, \infty)$, any $p \in (n, \infty ]$ and any $\epsilon \in (0, 1)$, there exists $\delta:=\delta (n, p, L, \epsilon)>0$ such that the following hold;
\begin{enumerate}
\item  let $(X_i, x_i)$ be a sequence of $n$-dimensional Riemannian manifolds with $\mathrm{Ric}_{X_i} \ge -(n-1)$,  let $(X, x, \meas)$ be the mGH-limit of $(X_i, x_i, \meas_i)$,
\item let $r \in (0, 1]$, let $f_i$ be smooth functions on $B_r(x_i)$ with 
$$
\sup_i  (\| \nabla f_i\|_{L^{\infty}(B_r(x_i))} +\|\Delta f_i\|_{L^p(B_r(x_i))}) \le L,
$$
and let $f$ be the $L^2$-strong limit function on $B_r(x)$.
\end{enumerate}
Then for any $\tau, s \in (0, \delta)$, any $y \in B_{r/2}(x) \cap \mathcal{R}_{\tau, s}(X)$ and any $t \in (0, s^2)$, we have
\begin{equation}\label{eq:equi bound hess}
\frac{t^2}{\meas (B_t(y))} \int_{B_t(y)}|\mathrm{Hess}_f^{g_X}|^2\dist \meas < \epsilon.
\end{equation}
\end{theorem}
\begin{proof}
By the rescaling; $\dist \mapsto r^{-1}\dist$, $f_i \mapsto r^{-1}f_i$, with no loss of generality we can assume that $r=1$.
The proof of (\ref{eq:equi bound hess}) is done by a contradiction.
If the assertion is false, then there exist $L \in [1, \infty)$, $p \in (n, \infty]$ and $\epsilon \in (0, 1)$ such that for any $j$ there exist;
\begin{itemize}
\item a sequence of $n$-dimensional complete Riemannian manifolds $(X_{i, j}, x_{i, j}, \meas_{i, j})$ with $\mathrm{Ric}_{X_{i, j}} \ge -(n-1)$,
\item the mGH-limit space $(X_j, x_j, \meas_j)$ of $(X_{i, j}, x_{i, j}, \meas_{i, j})$,
\item a sequence of smooth functions $f_{i, j}$ on $B_1(x_{i, j})$ with $\sup_i (\|\nabla f_{i, j}\|_{L^{\infty}(B_1(x_{i, j}))} +\| \Delta f_i\|_{L^p(B_1(x_{i, j}))} )\le L$,
\item the $L^2$-strong limit function $f_j$ of $f_{i, j}$ on $B_1(x_j)$, 
\item real numbers $\tau_j, s_j \in (0, j^{-1})$ and a point $y_j \in B_{1/2}(x_j) \cap \mathcal{R}_{\tau_j, s_j}(X_j)$ such that
$$
\frac{t_j^2}{\meas_j(B_{t_j}(y_j))}\int_{B_{t_j}(y_j)}|\mathrm{Hess}_{f_j}^{g_{X_j}}|^2\dist \meas_j \ge \epsilon
$$
for some $t_j \in (0, s_j^2)$.
\end{itemize} 
Let us consider the rescaling by $t_j^{-1}$; $\dist \mapsto t_j^{-1}\dist$, $f_{i, j} \mapsto t_j^{-1}f_{i, j}$, $\meas_{i, j} \mapsto \meas_{i, j}/\meas_{i, j}(B_{t_j}(x_j))$, etc. 
Moreover we shall use the ``hat''-notation after the rescaling, i.e. $\hat{\dist}:=t_j^{-1}\dist, \hat{f}_{i, j}:=t_j^{-1}f_{i, j}$, etc for short.
Note that for a fixed $R>0$ since $Rt_j<s_j$ for any sufficiently large $j$, we have 
$\dist_{GH}\left((B_{Rt_j}(y_j), y_j), (B_{Rt_j}(0_k), 0_k)\right)<R\tau_jt_j$,
i.e. 
$$
(\hat{X}_j, \hat{\dist}, \hat{y}_j, \hat{\meas}_j) \stackrel{GH}{\to} \left(\mathbf{R}^k, 0_k, \frac{1}{\mathcal{H}^k(B_1(0_k))}\mathcal{H}^k \right)
$$ for some $k \le n$.

Then by the lower semicontinuity of $L^2$-norms of Hessians \cite[Theorem 1.3]{Honda2}, we have
$$
\liminf_{i \to \infty}\|\mathrm{Hess}_{\hat{f}_{i, j}}\|_{L^2(\hat{B}_1(y_{i, j}))} \ge \|\mathrm{Hess}_{\hat{f}_{j}}^{g_{X_j}}\|_{L^2(\hat{B}_1(y_{j}))} \ge \epsilon
$$
for $y_{i, j} \stackrel{GH}{\to} y_j$.
Thus there exists a subsequence $i(j)$ such that $(\hat{X}_{i(j), j}, \hat{\dist}, y_{i(j), j}, \hat{\meas}_{i(j), j}) \stackrel{GH}{\to} (\mathbf{R}^k, 0_k, \frac{1}{\mathcal{H}^k(B_1(0_k))}\mathcal{H}^k)$ and that 
\begin{equation}\label{eq:contradiction}
\| \mathrm{Hess}_{\hat{f}_{i(j), j}}\|_{L^2(\hat{B}_1(y_{i(j), j}))} \ge \epsilon /2.
\end{equation}
On the other hand since $\sup_j \|\nabla \hat{f}_{i(j), j}\|_{L^{\infty}(\hat{B}_{1/(2t_j)}(y_{i(j), j}))}<\infty$, $\mathrm{Ric}_{\hat{X}_{i(j), j}} \ge -\tau_jt_j^2(n-1)$ and
\begin{align*}
\int_{\hat{B}_R(y_{i(j), j})}|\Delta \hat{f}_{i(j), j}|^2\dist \meas_{i(j), j}&=\frac{(t_j)^2}{\mathcal{H}^n(B_{Rt_j}(y_{i(j), j}))}\int_{B_{Rt_j}(y_{i(j), j})}|\Delta f_{i(j), j}|^2\dist \mathcal{H}^n\\
& \le (t_j)^2 \left(\mathcal{H}^n(B_{Rt_j}(y_{i(j), j}))\right)^{-2/p}\|\Delta f_{i(j), j}\|_{L^p(B_1(x_{i(j), j}))}\\
&\le C(n) (t_j)^2 (Rt_j)^{-(2n)/p}\|\Delta f_{i(j), j}\|_{L^p(B_1(x_{i(j), j}))} \to 0
\end{align*}
as $j \to \infty$
for any $R \in (0, \infty)$,
Lemma \ref{lem:conv hess} yields
$$
\lim_{j \to \infty}\|\mathrm{Hess}_{\hat{f}_{i(j), j}}\|_{L^2(\hat{B}_1(y_{i(j), j}))}=0,
$$
which contradicts (\ref{eq:contradiction}).
Thus we have (\ref{eq:equi bound hess}).
\end{proof}
\section{Orientability of Ricci limit spaces}
\subsection{Oriented atlas}
Let $(X, x, \meas)$ be a Ricci limit space whose dimension is $k$.
\begin{definition}[Orientations as rectifiable metric measure space]
We say that a rectifiable atlas $\{(C_i, \phi_i)\}_i$ of $(X, x, \meas)$ is \textit{oriented} if 
$$
\det J\left(\phi_i \circ (\phi_j)^{-1}\right)(z)>0
$$
for a.e. $z \in \phi_j(C_i \cap C_j)$ for all $i, j$.
We say that  two oriented rectifiable atlases $\{(C_i^j, \phi_i^j)\}_i (j=1, 2)$ are \textit{equivalent} if $\{(C_i^1, \phi_i^1)\}_i \cup \{(C_i^2, \phi_i^2)\}_i$ is also oriented. We denote by $O(X, \meas)$ the set of all equivalence classes $[\{(C_i, \phi_i)\}_i]$.
\end{definition}
It is not hard to check the following (see the proof of \cite[Lemma 3.5]{Honda10}). 
\begin{lemma}\label{ee}
We have the following.
\begin{enumerate}
\item Let $\omega \in L^{\infty}(\bigwedge^kT^*X)$ with $|\omega|(z)=1$ for a.e. $z \in X$.
Then for any Borel subset $C$ of $X$ and any $\epsilon \in (0, 1)$ there exists a countable family of pairwise disjoint rectifiable patches $(C_i, \phi_i)$ such that $C_i \subset C$, that $\meas (C  \setminus \bigcup_iC_i)=0$, that the orientation of each $(C_i, \phi_i)$ is compatible with $\omega$, i.e. $\langle \omega, \dist \phi_{i, 1} \wedge \cdots \wedge \dist \phi_{i, k} \rangle (z) >0$ for a.e. $z \in C_i$, and that $\phi_i$ is an $(1 \pm \epsilon)$-bi-Lipschitz embedding. Moreover we can take each $C_i$ as a compact subset.
\item Let $\{(C_i, \phi_i)\}_i$ be a rectifiable atlas of $(X, x, \meas)$ and let $\{C_{i, j}\}_j$ be countable families of Borel subsets $C_{i, j}$ of $C_i$ with $\meas \left( C_i \setminus \bigcup_jC_{i, j}\right)=0$.
Then there exist families $\{D_{i, j}\}_j$ of Borel subsets $D_{i, j}$ of $C_{i, j}$ such that $\meas \left( C_{i, j} \setminus D_{i, j}\right)=0$ and that $\{(D_{i, j}, \phi_i|_{D_{i, j}})\}_{i, j}$ is a rectifiable atlas of $(X, x, \meas)$.
\end{enumerate}
\end{lemma}
Let us take $\omega \in L^{\infty}(\bigwedge^kT^*X)$ with $|\omega (z)| =1$ for a.e. $z \in X$.
Then we define an oriented rectifiable atlas (associated with $\omega$) as follows.
We first fix a rectifiable atlas $\{(C_i, \phi_i)\}_i$ of $(X, x, \meas)$.
Let $C_i^+:=\{z \in C_i; \langle \omega, \dist \phi_{i, 1} \wedge \cdots \dist \phi_{i, k}\rangle (z)>0\}$, let $C_i^-:=\{z \in C_i; \langle \omega, \dist \phi_{i, 1} \wedge \cdots \dist \phi_{i, k}\rangle (z)<0\}$, let $\phi_i^+:=\phi_i$ and let $\phi_i^-:=(\phi_{i, 2}, \phi_{i, 1}, \phi_{i, 3}, \phi_{i, 4}, \ldots, \phi_{i, k})$.
Then applying Lemma \ref{ee} for $\{(C_i^+, \phi_i^+)\}_i \cup \{(C_i^-, \phi_i^-)\}_i$ gives an oriented rectifiable atlas.
We denote by $\mathcal{A}_{\omega}$ the atlas. Then it is easy to check the map: $\omega \mapsto [\mathcal{A}_{\omega}]$
is well-defined from the space $\{\omega \in L^{\infty}(X); |\omega (z)|=1\,\mathrm{for}\,\mathrm{a.e.} z \in X\}$ to $O(X, \meas)$ and that it is bijective.

From this observation we see that there are uncountable many equivalence classes of oriented, rectifiable atlases.
In the next subsection we will discuss main orientability in the sense of Ricci limit space.
\subsection{Definition and Properties}
We first recall \textit{test differential forms} introduced in \cite{Gigli}: $\mathrm{TestForm}_k(X):=\{\sum_{i=1}^Nf_{0, i}\dist f_{1, i} \wedge \cdots \wedge \dist f_{k, i}; N \in \mathbb{N}, f_{j, i} \in \mathrm{Test}F(X)\}$, which is dense in $L^2(\bigwedge^kT^*X)$.
Let us reformulate the definition of orientability by using test differential forms.
Let $(X, x, \meas)$ be a Ricci limit space whose dimension is $k$
\begin{definition}[Orientability]\label{def:ori}
We say that \textit{$(X, x, \meas)$ is orientable} if there exists $\omega \in L^{\infty}\left( \bigwedge^kT^*X \right)$ such that $|\omega|(z) =1$ for a.e. $z \in X$ and that 
\begin{equation}\label{eq:reg cond}
\left\langle \omega, \eta \right\rangle \in H^{1, 2}(X)
\end{equation}
for any $\eta \in \mathrm{TestForm}_k(X)$.
Then we call $\omega$ an \textit{orientation of} $(X, x, \meas)$. 
\end{definition}
\begin{proposition}\label{rem:equiv}
Let $\omega \in L^{\infty}(\bigwedge^kT^*X)$ with $|\omega |(z)=1$ for a.e. $z \in X$.
Then $\omega$ is an orientaion of $(X, x, \meas)$ if and only if 
\begin{equation}\label{eq:char}
\langle \omega, \dist f_1 \wedge \cdots \wedge \dist f_k \rangle \in H^{1, 2}(X)
\end{equation}
for any $f_i \in \mathrm{Test}F(X)$.
\end{proposition}
\begin{proof}
It is easy to check the proof of the `only if' part.
Assume that $|\omega|(z)=1$ for a.e. $z \in X$ and (\ref{eq:char}) are satisfied.
Take a sequence $\phi_j \in \mathrm{LIP}_c(X)$ such that $0 \le \phi_j \le 1$, that $\phi_j \equiv 1$ on $B_j(x)$, that $\supp \phi_j \subset B_{j+1}(x)$, and that $|\nabla \phi_j| \le 1$.
Then since  (\ref{eq:bakry}) yields $h_t\phi_j \in \mathrm{Test}F(X)$, we see that by definition 
$$
h_t\phi_j \langle \omega, \dist f_1 \wedge \cdots \wedge \dist f_k \rangle = \langle \omega, h_t\phi_j \dist f_1 \wedge \cdots \wedge \dist f_k \rangle\in H^{1, 2}(X)  \quad \forall t>0.
$$
Thus letting $t \downarrow0$ and then letting $j \uparrow \infty$ show (\ref{eq:reg cond}).
\end{proof}
\begin{proposition}\label{prop:prop1}
Let $(X, x, \meas)$ be a Ricci limit space and let $\omega \in L^{\infty}\left(\bigwedge^kT^*X\right)$ be an oriention of $(X, x, \meas)$.
Then $\omega$ is differentiable for a.e. $z \in X$ with $\nabla^{g_X} \omega (z)= 0$ for a.e. $z \in X$.
\end{proposition}
\begin{proof}
Corollary \ref{cor:com} yields that $\omega$ is differentiable for a.e. $z \in X$.
By the definition of Levi-Civita connection for a.e. $z \in X$ we have
$$
0=\langle \nabla |\omega|^2, v \rangle (z)= 2\langle \nabla^{g_X} \omega, \omega \otimes v \rangle (z)
$$
for any $v \in T^*_zX$.
Since $\omega (z)$ is a basis of $\bigwedge^kT^*_zX$ for a.e. $z \in X$, we have $\nabla^{g_X} \omega (z)=0$ for a.e. $z \in X$.  
\end{proof}
\begin{proposition}\label{b}
Let $(X, x, \meas)$ be a Ricci limit space and let $\omega \in L^{\infty}\left(\bigwedge^kT^*X\right)$ be an orientation of $(X, x,  \meas)$.
Then
\begin{equation}\label{eq:ineq}
\left\langle \omega, \dist f_1 \wedge \cdots \wedge \dist f_k\right\rangle \in H^{1, 2}(B_r(z))
\end{equation}
for any $r<R$ and any $f_i \in \mathcal{D}(\Delta, B_R(z))$ with $|\nabla f_i| \in L^{\infty}(B_R(x))$.
\end{proposition}
\begin{proof}
By the existence of good cut-off functions on Ricci limit spaces \cite[Corollary 4.29]{Honda2} there exists a cutoff $\phi \in \mathcal{D}(\Delta, X)$ such that $0 \le \phi \le 1$, that $\phi \equiv 1$ on $B_r(z)$, that $\supp \phi \subset B_R(z)$, and that $|\nabla \phi | + |\Delta \phi| \le C(n, r, R)$.
Considering $\phi f_i$ yields that with no loss of generaliry we can assume that $f_i \in \mathcal{D}(\Delta, X) \cap \mathrm{LIP}_c(X)$.

Then the regularity of the heat flow (c.f. \cite{AmbrosioGigliSavare13, AmbrosioGigliSavare14, AmbrosioGigliMondinoRajala}) yields that $h_tf_i \in \mathrm{Test}F(X)$ for any $t>0$, that $h_t f_i \to f_i$ in $H^{1, 2}(X)$ as $t \downarrow 0$, that $\sup_{t<1}\|\nabla h_tf_i\|_{L^{\infty}(X)}<\infty$, and that $\Delta h_tf_i \to \Delta f_i$ in $L^2(X)$ as $t \downarrow 0$.
In particular Bochner's inequality on $RCD$-spaces \cite[Corollary 3.3.9]{Gigli} with Proposition \ref{prop:comp hess}
yields $\sup_{t<1}\|\mathrm{Hess}_{h_tf_i}^{g_X}\|_{L^2(X)}<\infty$.

Then by Propositions \ref{rem:equiv} and  \ref{prop:prop1} with (\ref{eq:fundamental}) since
\begin{align*}
|\nabla\langle \omega, \dist (h_tf_1) \wedge \cdots \wedge \dist (h_tf_k)\rangle | &\le |\nabla^{g_X} \omega| \prod_i |\nabla h_tf_i| + \sum_i |\mathrm{Hess}^{g_X}_{f_i}|\prod_{j \neq i}|\nabla h_tf_j| \\
&= \sum_i |\mathrm{Hess}^{g_X}_{f_i}|\prod_{j \neq i}|\nabla h_tf_j|,
\end{align*}
in particular $\sup_{t<1}\|\langle \omega, \dist (h_tf_1) \wedge \cdots \wedge \dist (h_tf_k)\rangle\|_{H^{1, 2}(X)}<\infty$. 
Since $\langle \omega, \dist (h_tf_1) \wedge \cdots \wedge \dist (h_tf_k)\rangle \to \langle \omega, \dist f_1 \wedge \cdots \wedge \dist f_k\rangle$ in $L^2(X)$ as $t \downarrow 0$, this completes the proof. 
\end{proof}
\subsection{Proof of Theorem \ref{thm:unique}}
Let us prove Theorem \ref{thm:unique}.
There exists a Borel function $f:X \to \{-1, 1\}$ such that $\omega_1(z)=f(z)\omega_2(z)$ for a.e. $z \in X$.
It suffices to prove that $f$ is constant as follows.
  
\textbf{Step 1.} Let $p \in \mathcal{R}_{\tau, s}$ with $s \le \frac{1}{2}\tau^{1/2}$. Then for any $t \in (0, s^2)$ there exists $c(t) \in \{-1, 1\}$ such that (recall the notation $\Psi$ given in the preliminaries)
\begin{equation}\label{eq:const}
\frac{1}{\meas (B_t(p))}\int_{B_t(p)}\left| f-c(t) \right| \dist \meas \le \Psi(\tau, s; n).
\end{equation}
The proof is as follows.

Let $r:=\tau^{1/2}s  \ge 2s^2$.
Corollary \ref{cor:delta} yields that there exists a limit harmonic $(\Psi (\tau, s; n), C(n))$-splitting map $\mathbf{b}:=(\mathbf{b}_1, \ldots, \mathbf{b}_k): B_r(p) \to \mathbf{R}^k$.
Then Theorem \ref{thm:quantitative bound hess} shows for any $t \in (0, s^2)$
\begin{equation}\label{rrrfff}
\frac{t^2}{\meas (B_t(p))}\int_{B_t(p)}|\mathrm{Hess}_{\mathbf{b}_i}^{g_X}|^2\dist \meas \le \Psi(\tau, s;n).
\end{equation}
Let $g_i:= \langle \omega_i, \dist \mathbf{b}_1 \wedge \cdots \wedge \dist \mathbf{b}_k\rangle$. Then applying the Poincar\'e inequality of type $(1, 2)$ for $g_i$ with (\ref{eq:fundamental}), (\ref{rrrfff}) and Proposition \ref{b} yields
\begin{equation}\label{eq:1}
\frac{1}{\meas (B_t(p))} \int_{B_t(p)}\left| g_i-\frac{1}{\meas (B_t(p))}\int_{B_t(p)}g_i\dist \meas\right| \dist \meas \le \Psi(\tau, s;n).
\end{equation}
We now fix $\Psi(\tau, s;n)$ as above and write it $\psi$ for short.
Let $A:=\bigcap_{i, j}\{z \in B_t(p); |\langle \nabla \mathbf{b}_i, \nabla \mathbf{b}_j\rangle (z)-\delta_{ij}|<(\psi)^{1/2}\}$.
Then by the same argument to (\ref{eq:const lem2}) we have 
\begin{equation}\label{eq:const lem}
\frac{\meas (B_t(p) \setminus A)}{\meas (B_t(p))}\le n^2(\psi)^{1/2}.
\end{equation}
In particular since $\dist \mathbf{b}_1 \wedge \cdots \wedge \dist \mathbf{b}_k(z)$ is a basis of $\bigwedge^kT^*_zX$ with $|\dist \mathbf{b}_1 \wedge \cdots \wedge \dist \mathbf{b}_k(z)|=1 \pm \Psi(\psi; n)$ for a.e. $z \in A$, we have $ |g_i|(z)=1 \pm \Psi(\psi; n)$ for a.e. $z \in A$.
Thus combining this with (\ref{eq:const lem}) and (\ref{eq:1}) gives
\begin{align*}
\left| \frac{1}{\meas (B_t(p))}\int_{B_t(p)}g_i\dist \meas\right| &= \frac{1}{\meas (B_t(p))}\int_{B_t(p)}|g_i|\dist \meas \pm \psi \\
&=\frac{1}{\meas (B_t(p))}\int_{A}|g_i|\dist \meas \pm \Psi(\psi ; n) = 1 \pm \Psi(\psi ; n),
\end{align*}
where we used $|g_i| \le C(n)$ on $B_r(p)$.
Thus (\ref{eq:1}) shows
\begin{equation}\label{eq:2}
\frac{1}{\meas (B_t(p))} \int_{B_t(p)}\left| g_i-c_i\right| \dist \meas \le \Psi(\psi ;n),
\end{equation} 
where $c_i \in \{-1, 1\}$ is a constant.
Therefore letting $c:=c_1c_2$ yields
\begin{align*}
\frac{1}{\meas (B_t(p))}\int_{B_t(p)}\left| f-c \right| \dist \meas &=\frac{1}{\meas (B_t(p))}\int_{B_t(p)}\left| f-c_1c_2 \right| \dist \meas \\
&=\frac{1}{\meas (B_t(p))}\int_{B_t(p)}\left| fc_2-c_1 \right| \dist \meas \\
&=\frac{1}{\meas (B_t(p))}\int_{B_t(p)}\left| g_1-c_1\right| \dist \meas \pm \Psi(\psi ; n) \le \Psi(\psi ; n).
\end{align*}

\textbf{Step 2.} If $g \in L^1(B_R(z))$ and $c_1, c_2 \in \{-1, 1\}$ satisfy
$$
\frac{1}{\meas (B_R(z))}\int_{B_R(z)}|g-c_i|\dist \meas<1,
$$
then $c_1=c_2$.

This is a direct consequence of the inequality:
$$
|c_1-c_2| \le \frac{1}{\meas (B_R(z))}\int_{B_R(z)}|g-c_1|\dist \meas + \frac{1}{\meas (B_R(z))}\int_{B_R(z)}|g-c_2|\dist \meas <2.
$$

\textbf{Step 3.} There exists $\epsilon (n, R) \in (0, 1)$ such that the following hold. Let $r \in (0, R]$, let $\gamma: [0, r] \to X$ be a minimal geodesic and let $g \in L^1(B_{2r}(\gamma (0)))$.
Assume that there exists $s \in (0, r)$ such that for any $t \in [0, r]$ there exists $c_t \in \{-1, 1\}$ such that 
$$
\frac{1}{\meas (B_s(\gamma (t)))}\int_{B_s(\gamma (t))}\left| g-c_t \right| \dist \meas <\epsilon (n, R).
$$
Then $c_r = c_0$.

The proof is as follows.

Let us take a partition $0=t_0<t_1<t_2< \cdots <t_{N-1}<t_N=r$ of $[0, r]$ with $|t_i-t_{i+1}|<s/10$.
Then Bishop-Gromov inequality yields
\begin{align*}
\frac{1}{\meas (B_{s/2}(\gamma (t_i)))}\int_{B_{s/2}(\gamma (t_i))}\left|g-c_{t_{i}}\right|\dist \meas &\le \frac{C(n, R)}{\meas (B_{s}(\gamma (t_{i})))}\int_{B_{s}(\gamma (t_{i}))}\left|g-c_{t_{i}}\right|\dist \meas \\
&<C(n, R)\epsilon (n, R)<1.
\end{align*}
On the other hand since $B_{s/2}(\gamma (t_i)) \subset B_s(\gamma (t_{i+1}))$, applying Bishop-Gromov inequality again shows
\begin{align*}
\frac{1}{\meas (B_{s/2}(\gamma (t_i)))}\int_{B_{s/2}(\gamma (t_i))}\left|g-c_{t_{i+1}}\right|\dist \meas &\le \frac{C(n, R)}{\meas (B_{s}(\gamma (t_{i+1})))}\int_{B_{s}(\gamma (t_{i+1}))}\left|g-c_{t_{i+1}}\right|\dist \meas \\
&<C(n, R)\epsilon (n, R)<1.
\end{align*}
Thus step 2 shows $c_{t_i}=c_{t_{i+1}}$, which completes the proof.

\textbf{Step 4.} To finish the proof, we assume that $f$ is not a constant.
Let $A_+:=\{f=1\}$ and let $A_-:=\{f=-1\}$. Then they have positive measures.
By \cite[Lemma 1.17]{ColdingNaber} with the Lebesgue differentiation theorem there exist $p, q \in X$, $\delta \in (0, 1)$  and a limit minimal geodesic $(\delta, \dist (p, q)+\delta) \to X$ such that $\gamma (0)=p$, that $\gamma (\dist (p, q))=q$, that the image of $\gamma$ is in $\mathcal{R}^k$, and that
$$
\lim_{r \to 0}\frac{1}{\meas (B_r(p))}\int_{B_r(p)}\left|f-1\right|\dist \meas =\lim_{r \to 0}\frac{1}{\meas (B_r(q))}\int_{B_r(q)}\left|f+1\right|\dist \meas =0.
$$
Then applying the uniform Reifenberg property along the interior of a limit minimal geodesic \cite[Theorem B.1]{ColdingNaber} to $\gamma$ yields that for any $\epsilon \in (0, 1)$ there exists $r>0$ such that $\gamma ([0, \dist (p, q)]) \subset \mathcal{R}^k_{\epsilon, r}$.
In particular step 1 shows that for any $\epsilon \in (0, 1)$ there exists $r \in (0, 1)$ such that for any $t \in [0, \dist (p, q)]$ and any $s \in (0, r)$ there exists $c(\epsilon, r, t, s) \in \{-1, 1\}$ such that
$$
\frac{1}{\meas (B_s(\gamma (t)))}\int_{B_s(\gamma (t))}\left| f- c(\epsilon, r, t, s)\right| \dist \meas <\epsilon.
$$
Thus if $\epsilon$ is sufficiently small, then step 3 implies $1=c(\epsilon, r, 0, s)=c(\epsilon, r, \dist (p, q), s)=-1$, which is a contradiction. $\,\,\,\,\,\,\,\Box$

\begin{remark}
By the proof above, it is noticed easily that we can prove Theorem \ref{thm:unique} with no use of Theorem \ref{thm:quantitative bound hess} because we only use limit harmonic $(\epsilon, C(n))$-splitting maps and know the lower semicontinuity of $L^2$-norms of Hessian \cite[Theorem 1.3]{Honda2}. However since Theorem \ref{thm:quantitative bound hess} gives local behavior of more general functions, the author believes that this has independent importance.
\end{remark}
\subsection{Stability}
In this subsection we will prove Theorem \ref{ah}. More precisely;
\begin{theorem}[Stability]\label{thm:stability}
Let $(X_i, x_i, \meas_i)$ be a mGH-convergent sequence of Ricci limit spaces to $(X, x, \meas)$.
Assume that the sequence is noncollapsed (i.e. $\mathrm{dim}\,X_i=\mathrm{dim}\,X :=k\le n$ for any sufficiently large $i$)  and that $(X_i, x_i, \meas_i)$ is orientable with an orientation $\omega_i \in L^{\infty}(\bigwedge^kT^*X_i)$. 
Then $(X, x, \meas)$ is orientable.
More precisely, there exist a subsequence $\{i(j)\}_j$ and an orientation $\omega \in L^{\infty}(\bigwedge^kT^*X)$ of $(X, x, \meas)$ such that $\omega_{i(j)}$ $L^p_{\mathrm{loc}}$-strongly converge to $\omega$ for any $p \in (1, \infty)$. Then we say that $\omega$ is associated with $\omega_{i(j)}$.
\end{theorem}
\begin{proof}
By the $L^p$-weak compactness \cite[Proposition 3.50]{Honda2} with no loss of generality we can assume that the $L^p_{\mathrm{loc}}$-weak limit $\omega \in L^{\infty}(\bigwedge^kT^*X)$ of $\omega_i$ exists for any $p \in (1, \infty)$.
Note that since the sequence $\{(X_i, x_i, \meas_i)\}_i$ is noncollapsed, $\omega$ is also a top dimensional differential form on $X$.

First let us check (\ref{eq:char}).
Let $f_i \in \mathrm{Test}F(X)$.
From the existence of approximate sequences \cite[Proposition 1.10.2]{AmbrosioHonda}, there exist sequences of $f_{j, i} \in \mathrm{Test}F(X_i)$ such that $\sup_{j, i}\|\nabla f_{j, i}\|_{L^{\infty}} <\infty$ and that $f_{j, i}, \dist f_{j, i}, \Delta f_{j, i}$ $L^2$-strongly converge to $f_j, \dist f_j, \Delta f_j$ on $X$, respectively.
Then since $\langle \omega_i, \dist f_{1, i} \wedge \cdots \wedge \dist f_{k, i} \rangle \in H^{1, 2}(X_i)$ and
\begin{align*}\label{a}
|\nabla\langle \omega_i, \dist f_{1, i} \wedge \cdots \wedge \dist f_{k, i}\rangle | &\le |\nabla^{g_{X_i}} \omega_i| \prod_j |\nabla f_{j, i}| + \sum_l |\mathrm{Hess}^{g_{X_i}}_{f_{l, i}}|\prod_{j \neq l}|\nabla f_{j, i}| \nonumber \\
&=\sum_l |\mathrm{Hess}^{g_{X_i}}_{f_{l, i}}|\prod_{j \neq l}|\nabla f_{j, i}|,
\end{align*}
we have $M:=\sup_{i, j}\| \langle \omega_i, \dist f_{1, i} \wedge \cdots \wedge \dist f_{k, i} \rangle \|_{H^{1, 2}(X_i)}<\infty$, where we used (\ref{eq:fundamental}), Proposition \ref{prop:comp hess} and Bochner's inequality (\cite[Corollary 3.3.9]{Gigli}).
Therefore the stability of Sobolev functions \cite[Theorem 1.7.4]{AmbrosioHonda} (or \cite[Theorem 7.4]{GigliMondinoSavare} or \cite[Theorem 4.9]{Honda2}) yields $\langle \omega, \dist f_1 \wedge \cdots \wedge \dist f_k \rangle \in H^{1, 2}(X)$ with $\| \langle \omega, \dist f_1 \wedge \cdots \wedge \dist f_k\rangle \|_{H^{1, 2}(X)}\le M$, which proves (\ref{eq:char}).

In order to finish the proof, it suffices to check that $\omega_i$ $L^2_{\mathrm{loc}}$-strong converge to $\omega$.  Because this with the uniform $L^{\infty}$-bound of $\omega_i$ implies the $L^p_{\mathrm{loc}}$-strong convergence of $\omega_i$ (c.f. \cite[Proposition 3.69]{Honda2}), and 
since 
$$
1=\lim_{i \to \infty}\frac{1}{\meas_i(B_r(y_i))}\int_{B_r(y_i)}|\omega_i|^2\dist \meas_i=\frac{1}{\meas (B_r(y))}\int_{B_r(y)}|\omega |^2\dist \meas
$$
for any $y_i \stackrel{GH}{\to} y$,
letting $r \downarrow 0$ with the Lebesgue differentiation theorem yields $|\omega (z)|=1$ for a.e. $z \in X$. 
It is worth pointing out that we need the noncollapsed assumption to prove the $L^2_{\mathrm{loc}}$-strong convergence of $\omega_i$ because in general it is not satisfied in the collapsed setting. For example, as in Remark \ref{diffe}, the sequence of standard orientations $\omega_i$ of $\mathbf{S}^1(1) \times \mathbf{S}^1(1/i)$ $L^2$-weak, but not strong, converge to $0 \in L^2(\mathbf{S}^1(1))$ as $i \to \infty$, which is a counter example in the collpased setting.

The proof of the $L^2_{\mathrm{loc}}$-strong convergence of $\omega_i$ is as follows.

We first recall a result given in \cite{Honda3}; let $p \in (1, \infty)$, let $f_i \in L^p(B_R(x_i))$ with $\sup_i\|f_i\|_{L^p(B_R(x_i))}<\infty$ and let $f \in L^1(B_R(x))$.
If for a.e. $z \in B_R(x)$ and any $\epsilon>0$ there exist $r \in (0, 1)$ and a convergent sequencce $z_i \stackrel{GH}{\to} z$ such that for any $t \in (0, r)$
\begin{equation}\label{eq:22}
\limsup_{i \to \infty}\frac{1}{\meas_i (B_t(z_i))}\int_{B_t(z_i)}f_i\dist \meas_i -\frac{1}{\meas (B_t(z))}\int_{B_t(z)}f\dist \meas \le \epsilon
\end{equation}
holds, then $\limsup_{i \to \infty}\|f_i\|_{L^1(B_R(x_i))}\le \|f\|_{L^1(B_R(x))}$.
It is not difficult to check this by using Vitali's covering theorem and the doubling condition.
See \cite[Proposition 3.8]{Honda3} for the detail.

Let $z \in \mathcal{R}^k(X)$ and let $z_i \stackrel{GH}{\to} z$.
Then for any $\epsilon \in (0, 1)$ and any $L \in (1, \infty)$ there exists $r \in (0, 1)$ such that for any $t \in (0, r)$
$$
\dist_{\mathrm{GH}}\left((B_{Lt}(z), z), (B_{Lt}(0_k), 0_k)\right)<\epsilon t.
$$
Fix $t \in (0, r)$.
Then by Theorem \ref{arrr} (or the proof of Corollary \ref{cor:delta}),
there exists a limit harmonic $(\Psi(\epsilon, L^{-1};n), C(n))$-splitting map $\mathbf{b}^i=(\mathbf{b}^i_1, \ldots, \mathbf{b}^i_k):B_{4t}(z_i) \to \mathbf{R}^k$.
With no loss of generality we can assume that there exists a limit $(\Psi(\epsilon, L^{-1};n), C(n))$-splitting map $\mathbf{b}:=(\mathbf{b}_1, \ldots, \mathbf{b}_k):B_{2t}(z) \to \mathbf{R}^k$ such that $\mathbf{b}^i$ converge uniformly to $\mathbf{b}$ on $B_{2t}(z)$.
We now fix $\Psi(\epsilon, L^{-1};n)$ as above and denote it by $\psi$ for short. 

Let us consider the following;
\begin{itemize}
\item let $\eta =\dist \mathbf{b}_1 \wedge \cdots \wedge \dist \mathbf{b}_k$,
\item let $\eta_i:=\dist \mathbf{b}_{1}^i \wedge \cdots \wedge \dist \mathbf{b}_{k}^i$, 
\item let $A_i:=\bigcap_{j, l}\{w \in B_t(z_i);|\langle \nabla \mathbf{b}_{j}^i, \nabla \mathbf{b}_{l}^i \rangle (w)-\delta_{j l}|<(\psi)^{1/2} \}$ and
\item let $A:=\bigcap_{j, l}\{w \in B_t(z);|\langle \nabla \mathbf{b}_j, \nabla \mathbf{b}_l \rangle (w)-\delta_{j l}|<(\psi)^{1/2} \}$.
\end{itemize}
Note that $\eta_i$ $L^2$-strongly converge to $\eta$ on $B_t(z)$, in particular $\langle \omega_i, \eta_i \rangle$ $L^2$-weakly converge to $\langle \omega, \eta \rangle$ on $B_t(z)$.
On the other hand from Proposition \ref{b} and (\ref{eq:fundamental}) we have $\sup_i\|\langle \omega_i, \eta_i\rangle \|_{H^{1, 2}(B_t(z_i))}<\infty$.
Thus combining these with the Rellich compactness \cite[Theorem 4.9]{Honda2} shows that  $\langle \omega_i, \eta_i \rangle$ $L^2$-strongly converge to $\langle \omega, \eta \rangle$ on $B_t(z)$.

Then by the same argument to (\ref{eq:const lem2}) we have 
\begin{equation}\label{eq:25}
\frac{\meas (B_t(z) \setminus A)}{\meas (B_t(z))}\le n^2(\psi)^{1/2}.
\end{equation}
Note that since $|\eta|=1 \pm \Psi(\psi ;n)$ on $A$ and $\bigwedge^kT^*_zX$ is $1$-dimensional for a.e. $z \in X$, we have $|\omega -\langle \omega, \eta \rangle \eta|<\Psi(\psi; n)$ on $A$ (of course similar statements for $\omega_i, \eta_i, A_i$ also hold for any sufficiently large $i$).
Thus for any sufficiently large $i$
\begin{align*}
\frac{1}{\meas (B_t(z))}\int_{B_t(z)}|\omega|^2\dist \meas &= \frac{1}{\meas (B_t(z))}\int_{A}|\omega|^2\dist \meas \pm \Psi(\psi; n) \\
&= \frac{1}{\meas (B_t(z))}\int_{A}\langle \omega, \eta \rangle ^2\dist \meas \pm \Psi(\psi; n) \\
&= \frac{1}{\meas (B_t(z))}\int_{B_t(z)}\langle \omega, \eta \rangle ^2\dist \meas \pm \Psi(\psi; n) \\
&= \frac{1}{\meas (B_t(z_i))}\int_{B_t(z_i)}\langle \omega_i, \eta_i \rangle ^2\dist \meas \pm \Psi(\psi; n) \\
&= \frac{1}{\meas (B_t(z_i))}\int_{A_i}\langle \omega_i, \eta_i \rangle ^2\dist \meas \pm \Psi(\psi; n) \\
&= \frac{1}{\meas (B_t(z_i))}\int_{A_i}|\omega_i|^2\dist \meas \pm \Psi(\psi; n) =1 \pm \Psi(\psi; n),
\end{align*}
which proves (\ref{eq:22}) as $f=|\omega|^2, f_i=|\omega_i|^2$.
Therefore 
$$
1=\limsup_{i \to \infty}\||\omega|^2\|_{L^1(B_R(x_i))} \le \| |\omega|^2\|_{L^1(B_R(x))}=\|\omega\|_{L^2(B_R(x))},
$$
which shows that $\omega_i$ $L^2$-strongly converge to $\omega$ on $B_R(x)$.
\end{proof}
\begin{remark}\label{rem:loc ori}
The local version of orientability can be discussed as follows.
Let $(X, x, \meas)$ be a Ricci limit space whose dimension is $k$.
We say that $(X, x, \meas)$ is \textit{locally orientable at a point $p \in X$} if there exist $r \in (0, \infty)$ and $\omega \in L^{\infty}(\bigwedge^kT^*B_r(x))$ such that $\langle \omega, \eta \rangle|_{B_r(x)} \in H^{1, 2}(B_r(x))$ for any $\eta \in \mathrm{TestForm}_k(X, \dist, \meas)$.
Then we call $\omega$ a \textit{local orientation of $(X, \dist, \meas)$ at} $p$.

From the proof of Theorem \ref{thm:unique} we can prove similar uniequeness; for two local orientations $\omega_1, \omega_2$ of $(X, \meas)$ at $p$ there exists $s>0$ such that we have either $\omega_1(z)=\omega_2(z)$ for a.e. $z \in B_s(p)$ or $\omega_1(z)=-\omega_2(z)$ for a.e. $z \in B_s(p)$.
Moreover by the proof of Theorem \ref{thm:stability} we can also prove if $(X, \meas)$ is locally orientable at $p \in X$, then all tangent cones of $(X, x, \meas)$ at $p$, whose dimension are $k$, are orientable. 
\end{remark}
We end this subsection by giving a sufficient condition for the collapsed limit space to be orientable. See subsection 6.6 for the detail of metric currents.
\begin{theorem}[Orientability to collapsed spaces]
Let $(X_i, x_i, \meas_i)$ be a sequence of $n$-dimensional Riemannian manifolds with the normalized measures satisfying $\mathrm{Ric}_{X_i} \ge -(n-1)$, let $(X, x, \meas)$ be the mGH-limit space whose dimension is $k$, and let $\omega \in L^{\infty}(\bigwedge^kT^*X)$ with $|\omega|(z)=1$ for a.e. $z \in X$.
If there exists a sequence $\omega_i \in C^{\infty}(\bigwedge^kT^*X_i)$ such that $\sup_i\|\nabla \omega_i\|_{B_R(x_i)}<\infty$ for any $R \in (0, \infty)$ and that $\omega_i$ $L^2_{\mathrm{loc}}$-strongly converge to $\omega$, then
$\omega$ is an orientation of $(X, x, \meas)$, $\omega \in \mathcal{D}^2_{\mathrm{Loc}}(\delta_k, X)$ and $T_{\omega}$ is a metric current with $\partial T_{\omega}=T_{\delta_k\omega}$.
\end{theorem}
\begin{proof}
By the proof of Proposition \ref{rem:equiv} it suffices to check that $g:=\phi \langle \omega, \dist f_1 \wedge \cdots \wedge \dist f_k \rangle$ is in $H^{1, 2}(X)$ for any $f_i \in \mathrm{Test}F(X)$ and any $\phi \in \mathrm{LIP}_c(X)$.
Then by existences of approximate sequences \cite[Theorem 4.2]{Honda4}, \cite[Proposition 10.2]{AmbrosioHonda}, there exist $R \in (0, \infty)$, a sequence $\phi_j \in \mathrm{LIP}_c(X_j)$ and a sequence $f_{i, j} \in \mathrm{Test}F(X_j)$ such that $\supp \phi_j \subset B_R(x_j)$ for any $j$, and that $\phi_j, \nabla \phi_j, f_{i, j}, \nabla f_{i, j}, \Delta f_{i, j}$ $L^2$-strongly converge to $\phi, \nabla \phi, f_i, \nabla f_i, \Delta f_i$, respectively.
Then since $\sup_j\|\mathrm{Hess}_{f_{i, j}}\|_{L^2}<\infty$, (\ref{eq:fundamental}) yields $g_j:=\phi_j\langle \omega_j, \dist f_{1, j} \wedge \cdots \wedge \dist f_{k, j} \rangle \in H^{1, 2}(X_j)$ with $\sup_j\|g_j\|_{H^{1, 2}(X_j)}<\infty$.
Thus since $g_j$ $L^2$-strongly converge to $g$, the Rellich compactness \cite[Theorem 4.9]{Honda2} (or \cite[Theorem 6.8]{GigliMondinoSavare}, \cite[Theorem 5.4]{AmbrosioHonda}) shows $g \in H^{1, 2}(X)$.
The remaining statements follows directly from Theorem \ref{thm:imply}, Lemma \ref{lem:conv} and a fact that $|\delta_k\omega_j|\le C(n)|\nabla \omega_j|$.
\end{proof}
\subsection{Compatibility with the smooth case}
Let us denote again by $(X, x, \meas)$ a Ricci limit space whose dimension is $k$.
\begin{proposition}[Compatibility; singular to smooth]\label{prop:comp1}
Let $\omega \in L^{\infty}(\bigwedge^kT^*X)$ be an orientation of $(X, x, \meas)$, and let $O$ be an open subset of $X$.
Assume that $O$ is locally isometric to a $k$-dimensional $C^1$-Riemannian manifold with $\meas \lfloor_O =e^F \dist \mathcal{H}^k$ for some $F \in \mathrm{LIP}_{\mathrm{loc}}(O)$.
Then $O$ has the orientation in the ordinary sense with respect to $\omega$.
In particular $\omega$ is continuous on $O$.
\end{proposition}
\begin{proof}
Let $p \in O$ and let $\phi:=(\phi_1, \ldots, \phi_k): B_R(p) \to \mathbf{R}^k$ be a $C^2$-coordinate chart with $\overline{B}_R(p) \subset O$. 
Note that by a direct calculation we see $\phi_i \in \mathcal{D}(\Delta^{\meas}, B_R(p))$ with $\Delta^{\meas}\phi_i=\Delta \phi_i -\langle \dist F, \dist \phi_i \rangle$, where $\Delta$ is the standard Laplacian with respect to the $C^1$-Riemannian metric on $O$. 
Then $\omega$ can be expressed by
$$
\omega=\frac{f}{|\dist \phi_1 \wedge \cdots \wedge \dist \phi_k|}\dist \phi_1 \wedge \cdots \wedge \dist \phi_k
$$
on $B_r(p)$ for any $r \in (0, R)$, where $f$ is a function on $B_R(p)$ with $|f| \equiv 1$.
Thus we have $f=\langle \omega, \dist \phi_1 \wedge \cdots \wedge \dist \phi_k\rangle /|\dist \phi_1 \wedge \cdots \wedge \dist \phi_k|$.
In particular Proposition \ref{b} yields $f \in H^{1, 2}(B_r(p))$.
Since $|f| \equiv 1$ implies $|\nabla f|(z) =0$ for a.e. $z \in B_r(p)$, the Poincare inequality shows that $f$ is constant, which completes the proof.
\end{proof}
Let us discuss the opposite implication.
The key point is to consider the Sobolev (2)-capacity of a subset $A$ of $X$, denoted by $\mathrm{Cap}_2^{\meas}(A)$.
See for instance \cite{KM, Shanmugalingam} for the definition.
We only need the following properties;
\begin{itemize}
\item for a closed subset $A$ of $X$, $\mathrm{Cap}_2^{\meas}(A)=0$ if and only if there exists a sequence $f_i \in H^{1, 2}(X) \cap \mathrm{LIP}(X)$ such that $0 \le f_i \le 1$, that for any $i$, $f_i \equiv1$ on a neighborhood of $A$, and that $f_i \to 0$ in $H^{1, 2}(X)$.
\end{itemize}
\begin{proposition}[Compatibility; smooth to singular]\label{prop:comp2}
Let $O$ be an open subset of $X$.
Assume that $\mathrm{Cap}_2^{\meas}(X \setminus O)=0$ and that $O$ is locally isometric to a $k$-dimensional $C^1$-Riemannian manifold with $\meas \lfloor_O =e^F \dist \mathcal{H}^k$ for some function $F$ on $O$ satisfying that $f|_U \in H^{1, 2}(U)$ for any relatively compact open subset $U$ of $O$.
Then $O$ is orientable in the ordinary sense if and only if $(X, x, \meas)$ is orientable.
\end{proposition}
\begin{proof}
From Proposition \ref{prop:comp1} it suffices to check `only if' part.
Let $\omega \in C^1(\bigwedge^kT^*O)$ be the canonical form defind by an orientation of $O$ with $|\omega| \equiv 1$ on $O$.
Then we first check:
\begin{equation}\label{eq:9}
\langle \omega, \phi \dist f_1 \wedge \cdots \wedge \dist f_k \rangle \in H^{1, 2}(X)
\end{equation}
for any $\phi \in \mathrm{LIP}_c(X)$ and any $f_i \in \mathrm{Test}F(X, \dist, \meas)$.
Let $R \in (0, \infty)$ with $\supp \phi \subset B_R(x)$. 
Since $\mathrm{Cap}_2^{\meas}(\overline{B}_R(x) \setminus O)=0$ there exists a sequence $\phi_i \in H^{1, 2}(X)$ such that $\phi_i \equiv 1$ on a neighborhood of $\overline{B}_R(x) \setminus O$ and that $\phi_i \to 0$ in $H^{1, 2}(X)$.
Let us consider a sequence $g_i=(1-\phi_i) \langle \omega, \phi \dist f_1 \wedge \cdots \wedge \dist f_k \rangle$.
Since $\supp g_i \subset B_R(x) \cap O$ and $\omega$ is $C^1$ on $O$, it is easy to check $g_i \in H^{1, 2}_c( B_R(x) \cap O)$ with $\sup_i\|g_i\|_{H^{1, 2}}<\infty$.
Since $g_i \to \langle \omega, \phi \dist f_1 \wedge \cdots \wedge \dist f_k \rangle$ in $L^2(X)$, this proves (\ref{eq:9}).

Then in particular $\langle \omega, \dist f_1 \wedge \cdots \wedge \dist f_k \rangle|_{B_R(x)} \in H^{1, 2}(B_R(x))$ for any $R \in (0, \infty)$.
Moreover by (\ref{eq:fundamental}) with $\nabla^{g_X}\omega\equiv 0$ on $O$ we have $\sup_{R \in [1, \infty)}\|\langle \omega, \dist f_1 \wedge \cdots \wedge \dist f_k \rangle|_{B_R(x)} \|_{H^{1, 2}(B_R(x))}<\infty$, which implies easily $\langle \omega, \dist f_1 \wedge \cdots \wedge \dist f_k \rangle \in H^{1 ,2}(X)$.
\end{proof}
\begin{corollary}
Let $(X_i, x_i, \mathcal{H}^n)$ be a sequence of noncollapsed $n$-dimensional Riemannian manifolds with $|\mathrm{Ric}_{X_i}|\le n-1$, and let $(X, x, \mathcal{H}^n)$ be the noncollapsed mGH-limit space. Then $(X, \dist, x, \mathcal{H}^n)$ is orientable if and only if $\mathcal{R}^n(X)$ is orientable in the ordinary sense.
\end{corollary}
\begin{proof}
Recall that it was proven in \cite{CheegerNaber} by Cheeger-Naber that $\mathcal{R}^n(X)$ has codimension 4 (with respect to $\mathcal{H}^n$).
In particular $\mathrm{Cap}^{\meas}_2(X \setminus \mathcal{R}^n(X))=0$ (c.f. \cite[Theorem 4.13]{KM}).
Thus the assertion follows from Proposition \ref{prop:comp2}.
\end{proof}
\subsection{Metric currents}
In this section we will establish Theorem \ref{hhgg}.
\subsubsection{Quick introduction of currents in metric spaces}
The pioneer work on currents in metric spaces was founded by Ambrosio-Kirchheim in \cite{AmbrosioKirchheim} under finite mass condition.
After that Lang and Lang-Wenger generalized in \cite{Lang, LW} this to general case, called local currents. 
Then Lang-Wenger established in \cite{LW} the pointed flat compactness for integral current spaces \cite[Theorem 1.1]{LW}.
As mentioned in the introduction since Theorem \ref{hhgg} is closely related to the pointed flat compactness, we adopt their formulation here. 
See \cite{LW} for the details of the following.

Let $Y$ be a metric space, let us denote by $\mathrm{LIP}_B(Y)$ the set of all Lipschitz functions on $Y$ with bounded supports and let $\mathrm{LIP}_{\mathrm{Loc}}(Y)$ be the set of all functions on $Y$ which are Lipschitz on each bounded subset (thus $\mathrm{LIP}_{\mathrm{loc}}(Y) = \mathrm{LIP}_{\mathrm{Loc}}(Y)$ if $Y$ is proper). For $m \in \mathbf{Z}_{\ge 0}$, put $\mathcal{D}^m(Y):=\mathrm{LIP}_B(Y) \times (\mathrm{LIP}_{\mathrm{Loc}}(Y))^m$.
For an open subset $U$ of $Y$ and a multi-linear function $T:\mathcal{D}^m(Y) \to \mathbf{R}$, let
$$
\mathbf{M}_U(T):=\sup \sum_{\lambda \in \Lambda}T(f_{\lambda}, \pi_{\lambda}),
$$
where the supremum runs over all countable sets $\Lambda$, and all $(f_{\lambda}, \pi_{\lambda}) \in \mathcal{D}^m(Y) (\pi_{\lambda}=(\pi_{\lambda, 1}, \ldots, \pi_{\lambda, m}) \in (\mathrm{LIP}_{\mathrm{Loc}}(Y))^m)$ such that $\supp f_{\lambda} \subset U$, that $\mathbf{Lip} \pi_{\lambda, l} \le 1$ and that $\sum_{\lambda}|f_{\lambda}|\le 1$.
Then for any subset $A$ of $Y$ let 
$$
\|T\|(A):=\inf_U \mathbf{M}_U(T),
$$
where the infimum runs over all open subset $U$ of $X$ with $A \subset U$.
Then 
\begin{itemize}
\item{(Push-foward)} for any map $\phi:Y \to Z$ such that $\phi$ is Lipschitz on each bounded subset of $Y$ and that $\phi^{-1}(A)$ is bounded for any bounded subset $A$ of $Z$, let us define the multi-linear functional $\phi_{\sharp}(T):\mathcal{D}^m(Z) \to \mathbf{R}$ by
$$
\phi_{\sharp}(T)(f, \pi_1, \ldots, \pi_m):=T(f\circ \phi, \pi_1 \circ \phi, \ldots, \pi_m \circ \phi ).
$$
\end{itemize}
Moreover $T$ is said to be an \textit{$m$-dimensional metric functional on $Y$} if the following two conditions hold;
\begin{enumerate}
\item{(Continuity)} we have
$$
\lim_{j \to \infty}T(f, \pi_1^j, \ldots, \pi_m^j)=T(f, \pi_1, \ldots, \pi_m),
$$
whenever $\pi_i^j$ pointwise converge to $\pi_i^j$ with $\sup_j\mathbf{Lip} (\pi_i^j|_A)<\infty$ (in particular it is uniformly convergent on each bounded subset $A$ of $Y$).
\item{(Locality)} in case $m \ge 1$, $T(f, \pi_1, \ldots, \pi_m)=0$ whenever some $\pi_i$ is constant on a neighborhood of $\supp f$.
\end{enumerate}
Assume that $T$ is a metric functional.
Then 
\begin{itemize} 
\item{(Boundary)} let us define the $(m-1)$-dimensional metric functional $\partial T$ on $Y$ by
$$
\partial T(f, \pi_1, \ldots, \pi_{m-1}):=T(\sigma, f, \pi_1, \ldots, \pi_{m-1}),
$$
where $\sigma \in \mathrm{LIP}_B(Y)$ is any function satisfying $\sigma|_{\supp f} \equiv 1$. 
\end{itemize}
In addition, a metric functional $T$ is said to be a \textit{local current} if the following holds:
\begin{enumerate}
\setcounter{enumi}{1}
\setcounter{enumi}{2}
\item{(Borel regularity)} for any $\epsilon \in (0, 1)$ and any bounded open subset $U$ of $Y$ there exists a compact subset $C$ of $U$ such that $\mathbf{M}_{U \setminus C}(T)<\epsilon$.
\end{enumerate}
Then $\|T\|$ determines a Borel regular measure on the set of all Borel subsets of $Y$ and can be characterized as the minimum Borel measure $\nu$ on $Y$ satisfying
$$
|T(f, \pi_1, \ldots, \pi_m)| \le \prod_i \mathbf{Lip} (\pi_i|_{\supp f}) \int_Y|f|\dist \nu
$$
for any $(f, \pi_1, \ldots, \pi_m) \in \mathcal{D}^m(Y)$.
See \cite[Propositions 2.2 and 2.3]{LW}.

Assume that $T$ is a local current.
Let us denote by $\mathcal{B}^{\infty}_B(Y)$ the set of all bounded Borel functions with bounded supports on $Y$.
Then it is easy to check that there is a canonical extension of $T$ as a multi-linear map: $\mathcal{B}^{\infty}_B(Y) \times (\mathrm{LIP}_{\mathrm{Loc}}(Y))^m \to \mathbf{R}$, which is also denoted by $T$ for short.
Then  
\begin{itemize}
\item{(Restriction)} for a Borel subset $A$ of $Y$ let us define a local current $T \lfloor_A$ on $A$ by 
$$
T\lfloor_A (f, \pi_1, \ldots, \pi_m):=T(1_Af, \pi_1, \ldots, \pi_m).
$$  
\end{itemize}
We say that local current $T$ is \textit{normal} if $\partial T$ is also a local current.

Finally we recall definitions of locally integer rectifiable currents and of locally integral currents.
For that let us denote by $[g]$ the canonical $k$-dimensional local current on an open subset $U$ of $\mathbf{R}^k$ defined by given $g \in L^1_{\mathrm{Loc}}(U)$, i.e.
$$
[g](f, \pi_1, \ldots, \pi_k):=\int_{U}gf \langle \dist x_1 \wedge \cdots \wedge \dist x_k, \dist \pi_1\wedge \cdots \wedge \dist \pi_k\rangle \dist \mathcal{H}^k. 
$$
\begin{definition}[Locally integer rectifiable currents and locally integral currents]\cite[Definition 2.4]{LW}
Let $S$ be an $m$-dimensional metric functional on $Y$.
\begin{enumerate}
\item{(Locally integer rectifiable current)} $S$ is said to be a \textit{locally integer rectifiable current} if the following two conditions hold;
\begin{enumerate}
\item for any bounded open subset $U$ of $Y$ and any $\epsilon \in (0, 1)$ there exist a finite family of compact subsets $\{K_i\}_{i=1}^N$ of $\mathbf{R}^m$ and a family of Lipschitz maps $\phi_i:K_i \to Y$ such that $\mathbf{M}_{U \setminus \bigcup_{i=1}^N\phi_i(K_i)}(S)<\epsilon$. In particular $S$ is a local current on $Y$,
\item for any bounded Borel subset $B$ of $Y$ and any Lipschitz map $f:B \to \mathbf{R}^m$, there exists a $\mathbf{Z}$-valued $L^1_{\mathrm{loc}}$-function $\theta$ on $\mathbf{R}^m$ such that $f_{\sharp}(S \lfloor_B)=[\theta ]$.
\end{enumerate}
\item{(Locally integral current)} $S$ is said to be a \textit{locally integral current} if $S, \partial S$ are locally integer rectifiable currents. In particular $S$ is normal.
\end{enumerate}
\end{definition}
\subsubsection{Proof of Theorem \ref{hhgg}}
Recall that $(X, x, \meas)$ is a Ricci limit space.
Note that in this subsection we may not assume that $k$ denotes the dimension.

Let $U$ be an open subset of $X$.
We say that $\eta \in L^p_{\mathrm{Loc}}(\bigwedge^kT^*U)$ (which means that $\eta$ is $L^p$-bounded on each bounded subset of $U$) \textit{is in $\mathcal{D}^p(\delta_k, U)$ (or in $\mathcal{D}^p_{\mathrm{Loc}}(\delta_k, U)$, respectively) for some $p \in [1, \infty]$} if there exists a unique $\alpha \in L^p(\bigwedge^{k-1}T^*U)$ (or  $\in L^p_{\mathrm{Loc}}(\bigwedge^{k-1}T^*U)$, respectively), denoted by $\delta_k\eta$, such that 
$$
\int_X\langle \eta, \dist f_0 \wedge \cdots \wedge \dist f_{k-1} \rangle \dist \meas = \int_X\langle \alpha, f_0 \dist f_1 \wedge \cdots \wedge \dist f_{k-1} \rangle \dist \meas
$$ 
for any $f_i \in \mathrm{LIP}_c(U)$.
Note that $\mathcal{D}^p(\delta_k, U) = \mathcal{D}^p_{\mathrm{Loc}}(\delta_k, U)$ if $U$ is bounded.
It is easy to check that if $\eta \in \mathcal{D}^p_{\mathrm{Loc}}(\delta_k, X)$, then for any $f \in \mathrm{LIP}_{\mathrm{Loc}}(X)$, $f\eta \in \mathcal{D}^p_{\mathrm{Loc}}(\delta_k, X)$ with 
\begin{equation}\label{eq:12}
\delta_k (f\eta)=f\delta_k\eta-\eta(\nabla f, \cdot).
\end{equation}
\begin{lemma}\label{o}
If $\eta \in L^1_{\mathrm{Loc}}(\bigwedge^kT^*U) \cap \mathcal{D}^p_{\mathrm{Loc}}(\delta_k, U)$, then $\delta_k \eta \in \mathcal{D}^{p}_{\mathrm{Loc}}(\delta_{k-1}, U)$ with $\delta_{k-1}(\delta_k\eta)=0$.
\end{lemma}
\begin{proof}
For any $f_i \in \mathrm{LIP}_c(U)$ taking $\phi \in \mathrm{LIP}_c(U)$ with $\phi |_{\bigcup_i \supp f_i} \equiv 1$ yields
\begin{align*}
\int_U\langle \delta_k \eta, \dist f_0 \wedge \cdots \wedge \dist f_{k-1} \rangle \dist \meas &=
\int_U\langle \delta_k \eta, \phi \dist f_0 \wedge \cdots \wedge \dist f_{k-1} \rangle \dist \meas \\
&=\int_U\langle \eta, \dist \phi \wedge \dist f_0 \wedge \cdots \wedge \dist f_{k-1} \rangle \dist \meas =0,
\end{align*} 
which completes the proof,
where we used $\dist \phi(z) =0$ for a.e. $z \in \bigcup_i\supp f_i$ in the final equality. 
\end{proof}
\begin{theorem}[From $L^1_{\mathrm{Loc}}$-forms to metric currents]\label{thm:imply}
Let $\omega \in L^1_{\mathrm{Loc}}\left( \bigwedge^kT^*X \right)$ (note that $ L^1_{\mathrm{Loc}}\left( \bigwedge^kT^*X \right)= L^1_{\mathrm{loc}}\left( \bigwedge^kT^*X \right)$ in this case).
Then the multi-linear functional $T_{\omega}: \mathcal{D}^k(X) \to \mathbf{R}$ defined by 
$$
T_{\omega}(f, \pi_1, \ldots, \pi_k):=\int_X\langle \omega, f\dist \pi_1 \wedge \cdots \wedge \dist \pi_k \rangle \dist \meas 
$$
satisfies $\|T_{\omega}\| \le |\omega |\dist \meas$ (i.e. $\|T_{\omega}\|(A) \le \int_A|\omega| \dist \meas$ for any Borel subset $A$ of $X$) and the locality condition.
Moreover we have the following;
\begin{enumerate}
\item If $k=\mathrm{dim}\,X$, then $\|T_{\omega}\| =|\omega |\dist \meas$.
\item If $\omega \in \mathcal{D}_{\mathrm{Loc}}^1(\delta_k, X)$, then $T_{\omega}$ is a locally normal current with $\partial T_{\omega} =T_{\delta_k\omega}$.
\end{enumerate}
\end{theorem}
\begin{proof}
We first check $\|T_{\omega}\| \le |\omega |\dist \meas$.
For any Borel subset $A$ of $X$ and any open subset $U$ of $X$ containing $A$, since
\begin{align*}
\sum_{\lambda \in \Lambda}T_{\omega}(f_{\lambda}, \pi_{\lambda}) &=\sum_{\lambda \in \Lambda}\int_X\langle \omega, f_{\lambda} \dist \pi_{\lambda, 1} \wedge \cdots \wedge \dist \pi_{\lambda, k}\rangle \dist \meas \\
&\le \int_U|\omega| \left(\sum_{\lambda \in \Lambda}|f_{\lambda}| \right) \dist \meas \\
&\le \int_U |\omega| \dist \meas 
\end{align*}
for any $\Lambda$ and any $(f_{\lambda}, \pi_{\lambda})$ as in the definition of $\mathbf{M}_U(T_{\omega})$, we have $\mathbf{M}_U(T_{\omega}) \le \int_U|\omega| \dist \meas$.
Since $U$ is arbitrary, the Borel regularity of $\meas$ yields $\|T_{\omega}\| \le |\omega|\dist \meas$, which implies the locality condition.

Let us prove (1).
Let $A$ be a bounded Borel subset of $X$ and let $\delta \in (0, 1)$.
Then by the rectifiablity of $(X, x, \meas)$ (c.f. Lemma \ref{ee}) it is easy to check that there exist a countable family of bounded Borel subsets $A_i$ of $A$ and a family of $(1 \pm \delta)$-bi-Lipschitz embeddings $\phi_i: A_i \hookrightarrow \mathbf{R}^k$ such that $A_i$ are pairwise disjoint with $\meas \left(A \setminus \bigsqcup A_i \right)=0$.
In particular $\int_{A \setminus \bigsqcup_iA_i}|\omega |\dist \meas=0$.
Moreover by considering the decomposition 
\begin{align*}
A_i&=\left\{ z \in A_i; \langle \omega, \dist \phi_{i, 1} \wedge \cdots \wedge \dist \phi_{i, k} \rangle (z) >0\right\} \sqcup \left\{  z \in A_i; \langle \omega, \dist \phi_{i, 1} \wedge \cdots \wedge \dist \phi_{i, k} \rangle (z)<0\right\} \\
&\sqcup \left\{ z \in A_i; \omega (z)=0\right\}
\end{align*}
with no loss of generality we can assume that $\phi_i$ is defined on $X$ as a $(1+\delta)$-Lipschitz map and that
$$
\langle \omega, \dist \phi_{i, 1} \wedge \cdots \wedge \dist \phi_{i, k} \rangle >0
$$
on $A_i$.

Let $\tilde{\phi}_i:=(1+\delta)^{-1}\phi_i$,
let $U$ be a bounded open subset of $X$ with $A \subset U$ and $\int_{U \setminus A}|\omega| \dist \meas <\delta$, let $N \in \mathbf{N}$ with $\int_{A \setminus \bigsqcup_{i=1}^NA_i}|\omega| \dist \meas<\delta$ and let $B_i$ be a compact subset of $A_i$ with $\int_{A_i \setminus B_i}|\omega| \dist \meas \le \frac{1}{2^i}\delta$.
Note that $\tilde{\phi}_i$ is $1$-Lipschitz.
Moreover we can take  $\{f_i\}_{i=1}^N \subset \mathrm{LIP}_c(U)$ such that $0 \le f_i \le 1$, that $f_i|_{B_i}\equiv 1$, that $\{\supp f_i\}_{i=1}^N$ are pairwise disjoint and that $\int_X|1_{\bigsqcup_{i=1}^NB_i}-f||\omega| \dist \meas \le \delta$, where $f:=\sum_{i=1}^Nf_i$.

Then (recall the notation $\Psi$ in the preliminaries)
\begin{align*}
\mathbf{M}_U(T_{\omega}) &\ge \sum_{i=1}^N \int_X \langle \omega, f_i\dist \tilde{\phi}_{i, 1} \wedge \cdots \wedge \dist \tilde{\phi}_{i, k} \rangle \dist \meas \\
& \ge \sum_{i=1}^N \int_{B_i} \langle \omega, \dist \tilde{\phi}_{i, 1} \wedge \cdots \wedge \dist \tilde{\phi}_{i, k} \rangle \dist \meas -\delta \\
&\ge \sum_{i=1}^N (1- \Psi(\delta; k))\int_{B_i}|\omega |\dist \meas -\delta \\
&=(1-\Psi (\delta ;k))\int_{\bigsqcup_{i=1}^N B_i }|\omega | \dist \meas -\delta \ge (1-\Psi (\delta ;k))\int_{A}|\omega | \dist \meas -\delta.
\end{align*}
Since $\|T_{\omega}\|(U \setminus A) \le \int_{U \setminus A}|\omega|\dist \meas <\delta$, 
letting $\delta \downarrow 0$ proves $\|T_{\omega}\|(A) \ge \int_A|\omega | \dist \meas$, which proves (1) (see also \cite[Lemma 4.7]{Lang} and \cite[Proposition 2.7]{AmbrosioKirchheim}).

Next we prove (2).
In order to prove that $T_{\omega}$ is a local current,
it suffices to check the continuity condition.

Let $\phi \in \mathrm{LIP}_B(X)$ and let $f_{j, i} \in \mathrm{LIP}_{\mathrm{Loc}}(X)$ be uniformly convergent sequences to $f_j \in \mathrm{LIP}_{\mathrm{Loc}}(X)$ on $B_R(x)$ with $\sup_{i, j} \mathbf{Lip}(f_{j, i}|_{B_R(x_i)})<\infty$ for any $R \in (0, \infty)$.
Note
\begin{align}\label{eq:star}
\int_X\langle \omega, \phi \dist f_{1, i} \wedge \cdots \wedge \dist f_{k, i} \rangle \dist \meas &= \int_X\langle \omega, \phi \dist f_1 \wedge \dist f_2 \wedge \cdots \wedge \dist f_k \rangle \dist \meas \nonumber \\
&+ \int_X\langle \omega, \phi \dist f_1 \wedge \dist f_2 \wedge \cdots \wedge \dist f_{k-1} \wedge \dist (f_{k, i}-f_k) \rangle \dist \meas  \nonumber \\
& \ldots  \nonumber \\
&+\int_X\langle \omega, \phi \dist f_1 \wedge \dist (f_{2, i}-f_2) \wedge \cdots \wedge \dist f_{k, i} \rangle \dist \meas \nonumber  \\
& +\int_X\langle \omega, \phi \dist (f_{1, i}-f_1) \wedge \dist f_{2, i} \wedge \cdots \wedge \dist f_{k, i} \rangle \dist \meas. \nonumber 
\end{align}
Then 
\begin{align*}
&\lim_{i \to \infty} \left|\int_X\langle \omega, \phi \dist f_1 \wedge \cdots \wedge \dist f_{j-1} \wedge \dist (f_{j, i}-f_j) \wedge \dist f_{j+1, i} \wedge \cdots \wedge \dist f_{k, i} \rangle \dist \meas\right| \\
&=\lim_{i \to \infty} \left|\int_X\langle \phi \omega, \dist f_1 \wedge \cdots \wedge \dist f_{j-1} \wedge \dist (f_{j, i}-f_j) \wedge \dist f_{j+1, i} \wedge \cdots \wedge \dist f_{k, i} \rangle \dist \meas\right| \\
&\le \lim_{i \to \infty} \left(\prod_{l < j}\mathbf{Lip}(f_l|_{\supp \phi})\right) \cdot \left(\prod_{l>j} \mathbf{Lip}(f_{l, i}|_{\supp \phi}) \right) \int_X |\delta_k(\phi \omega)||f_{j, i}-f_j|\dist \meas=0,
\end{align*}
where we used (\ref{eq:12}).
In particular
\begin{equation}\label{eq:20}
\lim_{i \to \infty}\int_X\langle \omega, \phi_i\dist f_{1, i} \wedge \cdots \wedge \dist f_{k, i} \rangle \dist \meas = \int_X\langle \omega, \phi\dist f_{1} \wedge \cdots \wedge \dist f_{k} \rangle \dist \meas,
\end{equation}
which is the desired continuity property. Therefore $T_{\omega}$ is a local current. 
Then it is easy to check $\partial T_{\omega}=T_{\delta_k\omega}$.
Moreover applying the above for $\delta_k\omega$ with Lemma \ref{o} shows that $\partial T_{\omega}$ is also a local current, which completes the proof. 
\end{proof}
\begin{remark}\label{100}
It is easy to see that for any $n \ge 2$ the space $(X, \meas):=([0, \pi], \frac{1}{\int_0^{\pi}\sin ^{n-1}t \dist t}\int \sin^{n-1} t\dist t)$ is the collapsed mGH-limit space of a sequence $(\mathbf{S}^n, g_i, \frac{1}{\mathcal{H}^n(\mathbf{S}^n)}\mathcal{H}^n)$, where $g_i$ is a sequence of Riemannian metrics on the $n$-dimensional unit sphere $\mathbf{S}^n$ whose sectional curvature is bounded below by $1$ (see for instance \cite[Remark 1.10.6]{AmbrosioHonda} in the case when $n=2$).
Then $\mathrm{Cap}^{\meas}_2(\{0, 1\})=0$ because if let $f_k$ be Lipschitz functions on $[0, \pi]$ defined by
$$
f_k(t):=
\begin{cases} 1 \,\,\,\,\,\,\,\,\,\,\,\,\,\,\,\,\,\,\,\,\,\,\,\,\mathrm{if}\,t\in  [0, 1/k], \\
1-\frac{\log (tk)}{\log 2} \,\,\,\,\,\,\mathrm{if}\,t \in (1/k, 2/k], \\
0 \,\,\,\,\,\,\,\,\,\,\,\,\,\,\,\,\,\,\,\,\,\,\,\,\mathrm{if}\,t \in (2/k, \pi],
\end{cases}
$$
then it is easy to check that $f_k \to 0$ in $H^{1, 2}(X)$, which implies $\mathrm{Cap}^{\meas}_2(\{0, 1\})=0$.
Thus by Proposition \ref{prop:comp2}, $(X, \meas)$ is orientable.
Let $\omega$ be the canonical orientation, i.e. $\omega:=\dist t$.
Then integration by parts shows $\omega \in \mathcal{D}^{\infty}(\delta_1, [0, \pi])$ with $\delta_1 \omega=-(n-1)(\tan t)^{-1}$.
In particular by Theorem \ref{thm:imply}, $T_{\omega}$ is a normal current with $\partial T_{\omega}=[-(n-1)(\tan t)^{-1}]$ on $[0, \pi]$.
\end{remark}
\begin{remark}\label{1}
It is easy to see that $(X, \meas):=([0, \pi], \frac{1}{\pi}\mathcal{H}^1)$ is the collapsed mGH-limit space of a sequence  $(\mathbf{S}^2, g_i, \frac{1}{\mathcal{H}^2(\mathbf{S})}\mathcal{H}^2)$ whose sectional curvature is nonnegative.
We can check that $(X, \meas)$ is orientable as follows.
Note that $\mathrm{Cap}^{\meas}_2(\{0, 1\}) \neq 0$ and that the eigenvalue of $\Delta$ is of the Neumann problem, i.e. $\{f_i(t):=\sqrt{2}\cos(it)\}_i$ are all eigenfunctions of $\Delta$, in particular this gives an orthonormal basis in $L^2(X)$ and a basis in $H^{1, 2}(X)$.

Let $\omega:=\dist t$.
For any $f \in \mathrm{Test}F(X)$ let $f=\sum_i a_if_i$ in $H^{1, 2}(X)$, i.e. $a_i=\int_Xff_i\dist \meas$.
Note that since $f \in \mathcal{D}(\Delta, X)$, we have $\Delta f=\sum_ii^2a_if_i$ in $L^2(X)$.
In particular $L:=\sum_i(i^2 a_i)^2<\infty$. Let $g_n=\sum_{i=1}^na_i \langle \omega, \dist f_i \rangle \in C^{\infty}([0, \pi])$.
Then
\begin{align*}
\|\dist g_n\|_{L^2}^2&=\left\|\sum_{i=1}^na_i\Delta f_i\right\|_{L^2}^2 =\sum_{i=1}^n(i^2a_i)^2 \le L.
\end{align*}
Thus since $g_n \to \langle \omega, \dist f \rangle $ in $L^2(X)$, we have $\langle \omega, \dist f \rangle \in H^{1, 2}(X)$, which implies that $\omega$ is an orientation of $(X, \dist, \meas)$.
Moreover since 
\begin{equation}\label{mmmm}
\int_0^{\pi}\langle \omega, \dist h \rangle \dist \meas =\frac{1}{\pi}\left(h(\pi)-h(0)\right) 
\end{equation}
for any $h \in \mathrm{LIP}(X, \dist)$,
$\omega$ is \textit{not} in $\mathcal{D}^1(\delta_1, X)$.
In particular $\omega$ is not in $H^{1, 2}_H(T^*X)$ (see \cite{Gigli} or Section 7 for the definition of Sobolev spaces $H^{1, 2}_H$ for differential forms).
However by (\ref{mmmm}) we can check directly that $T_{\omega}$ is a metric current with $\partial T_{\omega}=\frac{1}{\pi}(\delta_{\pi}-\delta_0)$, where $\delta_t$ is the Dirac measure centered on $t$.
\end{remark}
\begin{lemma}\label{lem:conv}
Let $(X_i, x_i, \meas_i) \stackrel{GH}{\to} (X, x, \meas)$ be a convergent sequence of ($n$-)Ricci limit spaces, let $p \in (1, \infty]$, and let $\omega_i \in  L^p(\bigwedge^kT^*B_R(x_i)) \cap \mathcal{D}^p(\delta_k, B_R(x_i)) (i=1, 2, \ldots )$ with $\sup_{i<\infty}(\|\omega_i \|_{L^p}+\|\delta_k\omega_i \|_{L^p})<\infty$.
Then there exist a subsequence $i(j)$ and $\omega \in  L^p(\bigwedge^kT^*B_R(x)) \cap \mathcal{D}^p(\delta_k, B_R(x))$ such that $\omega_{i(j)}, \delta_k\omega_{i(j)}$ $L^p$-weakly converge to $\omega, \delta_k\omega$ on $B_R(x)$, respectively.
\end{lemma}
\begin{proof}
By the $L^p$-weak compactness, there exist a subsequence $i(j)$, $\omega \in  L^p(\bigwedge^kT^*B_R(x))$ and $\eta \in  L^p(\bigwedge^{k-1}T^*B_R(x))$ such that $\omega_{i(j)}, \delta_k\omega_{i(j)}$ $L^p$-weakly converge to $\omega, \eta$ on $B_R(x)$, respectively.
Let $f_i \in \mathrm{LIP}_c(B_R(x)) (i=0, 1, \ldots, k)$. Then from the existence of an approximate sequence \cite[Theorem 4.2]{Honda4}, there exist sequences of $f_{l, i(j)} \in \mathrm{LIP}_c(B_R(x_i))$ such that $f_{l, i(j)}, \dist f_{l, i(j)}$ $L^q$-strongly converge to $f_l, \dist f_l$ on $B_R(x)$ for any $q \in (1, \infty)$ with $\sup_{j}\mathbf{Lip} f_{l, i(j)}<\infty$. 
Then since 
$$
\int_{X_{i(j)}}\langle \omega_{i(j)},\dist f_{1, i(j)} \wedge \cdots \wedge \dist f_{k, i(j)} \rangle \dist \meas_{i(j)} = \int_{X_{i(j)}}\langle \delta_k\omega_{i(j)}, f_{1, i(j)} \dist f_{2, i(j)} \wedge \cdots \wedge \dist f_{k, i(j)} \rangle \dist \meas_{i(j)},
$$
letting $j \to \infty$ yields $\omega \in \mathcal{D}^p(\delta_k, B_R(x))$ with $\delta_k\omega=\eta$.
\end{proof}
\begin{corollary}\label{co:stab}
Let $(X_i, x_i, \meas_i) \stackrel{GH}{\to} (X, x, \meas)$ be a convergent sequence of ($n$-)Ricci limit spaces, let $p \in (1, \infty]$ and let $\omega_i \in L^p_{\mathrm{Loc}}\left(\bigwedge^kT^*X_i\right) (=L^p_{\mathrm{loc}}\left(\bigwedge^kT^*X_i\right))$ be an $L^p_{\mathrm{loc}}$-weakly convergent sequence to $\omega \in L^p_{\mathrm{Loc}}\left(\bigwedge^kT^*X\right)(=L^p_{\mathrm{loc}}\left(\bigwedge^kT^*X\right))$.
Then we have the following.
\begin{enumerate}
\item If for any $i$, $T_{\omega_i}$ is a local current on $X_i$ with $\sup_i\|\partial T_{\omega_i}\|(B_R(x_i))<\infty$ for any $R \in (0, \infty)$, then
$T_{\omega_i}$ converge to $T_{\omega}$ in the following sense;
\begin{equation}\label{eq:15}
\lim_{i \to \infty}T_{\omega_i}(f_{0, i}, f_{1, i}, \ldots, f_{k, i})=T_{\omega}(f_0, f_1, \ldots, f_k).
\end{equation}
whenever $f_{j, i}$ converge uniformly to $f_j$ on $B_R(x)$ with $\sup_i\mathbf{Lip}(f_{j, i}|_{B_R(x_i)})<\infty$ for any $R \in (0, \infty)$ and there exist $j$ and $R_0 \in (0, \infty)$ such that $\supp f_{j, i} \subset B_{R_0}(x_i)$ for any $i$.
\item If for any $i$, $\omega_i \in \mathcal{D}^p_{\mathrm{Loc}}(\delta_k, X_i)$ with $\sup_i\|\delta_k\omega_i\|_{L^p(\bigwedge^kT^*B_R(x_i))}<\infty$ for any $R \in (0, \infty)$, then we see that $T_{\omega_i}, T_{\omega}$ are locally  normal currents, that $\omega \in \mathcal{D}^p_{\mathrm{Loc}}(\delta_k, X)$ and that $T_{\omega_i}, \partial T_{\omega_i}$ converge to $T_{\omega}, \partial T_{\omega}$ in the sense of (\ref{eq:15}), respectively.
\end{enumerate}
\end{corollary}
\begin{proof}
Let us prove (\ref{eq:15}).
By using cut-off functions, with no loss of generality we can assume that $\supp f_{j, i} \subset B_{R_0}(x_i)$ for any $i, j$.
Then since $h_tf_{j, i}$ are uniformly Lipschitz functions and $L^q_{\mathrm{loc}}$-strongly converge to $h_tf_j$ for any $q \in (1, \infty)$ and any $t>0$ (c.f. \cite[Corollary 5.5]{AmbrosioHonda}), we have 
\begin{equation}\label{eq:20}
\lim_{i \to \infty}T_{\omega_i}(f_{0, i}, h_tf_{1, i}, \ldots, h_tf_{k, i})=T_{\omega}(f_0, h_tf_1, \ldots, h_tf_k).
\end{equation}
Then by an argument similar to the proof of (2) of Theorem \ref{thm:imply} (by using $\partial T_{\omega}$ instead of $\delta_k\omega$) we have
\begin{equation}\label{eq:21}
\lim_{t \downarrow 0}\left( \limsup_{i \to \infty}\left|  T_{\omega_i}(f_{0, i}, h_tf_{1, i}, \ldots, h_tf_{k, i})-T_{\omega_i}(f_{0, i}, f_{1, i}, \ldots, f_{k, i}) \right|\right)=0.
\end{equation}
Thus combining (\ref{eq:20}) with (\ref{eq:21}) shows (\ref{eq:15}).

Moreover (2) follows from (1), Theorem \ref{thm:imply} and Lemma \ref{lem:conv}.
\end{proof}
Let us denote by $D:=D(X, x, \meas )$ the set of points $z \in X$ such that the limit
$$
\lim_{r \downarrow 0}\frac{\meas (B_r(z))}{r^k}
$$
exists, and is positive and finite, where $k$ is the dimension of $(X, x, \meas)$.
Then for any $z \in D$ we put 
$$
g(z):=g_{(X, x, \meas)}(z)=\lim_{r \downarrow 0}\frac{\meas (B_r(z))}{\mathcal{H}^k(B_r(0_k))}.
$$
Recall that it is proven by Cheeger-Colding that $\meas (X \setminus D)=0$ and that if $(X, x, \meas )$ is a noncollapsed ($n$-)Ricci limit space, then we see that $X=D$, that $g \le \frac{1}{\mathcal{H}^n(B_1(x))}$ and that $g(z)=\frac{1}{\mathcal{H}^n(B_1(x))}$ if and only if $z \in \mathcal{R}^n(X)$, in particular $g(z)=\frac{1}{\mathcal{H}^n(B_1(x))}$ for a.e. $z \in X$ (recall that $\meas =\mathcal{H}^n/\mathcal{H}^n(B_1(x))$). See \cite[Theorems 3.1 and 5.9]{CheegerColding1} and \cite[Theorems 3.23 and 4.6]{CheegerColding3}.
Note that $\meas$ and $\mathcal{H}^k$ are mutually absolutely continuous on $D$.
For example, for $([0, \pi], \dist, \frac{1}{\int_0^{\pi}\sin ^{n-1}\dist t}\int \sin^{n-1} t\dist t)$ as in Remark \ref{100},
we see that $D([0, \pi], \dist, \frac{1}{\int_0^{\pi}\sin ^{n-1}\dist t}\int \sin^{n-1} t\dist t)=(0, 1)$ and that $g(t)= \frac{1}{\int_0^{\pi}\sin ^{n-1}t \dist t} \sin^{n-1}t$.
\begin{theorem}[Push-forward formula]\label{thm:push}
Let $k$ denote the dimension of $(X, x, \meas)$, let $\omega \in L^{\infty}(\bigwedge^kT^*X)$ with $|\omega |(z)=1$ for a.e. $z \in X$, let $C$ be a Borel subset of $D(X, x, \meas)$ and let $\phi:C \hookrightarrow \mathbf{R}^k$ be a bi-Lipschitz embedding.
Assume that the orientation of $(C, \phi)$ is compatible with $\omega$, (recall that this means $\langle \omega, \dist \phi_1 \wedge \cdots \wedge \dist \phi_k\rangle(z) >0$ for a.e. $z \in C$).
Then 
$$
\phi_{\sharp}\left(T_{\omega}\lfloor_C\right)=\left[ 1_{\phi (C)} g \circ \phi^{-1}\right].
$$
\end{theorem}
\begin{proof}
Fix $\epsilon \in (0, 1)$. By Lemma \ref{ee} there exists a coutable pairwise disjoint rectifiable patches $(C_i, \psi_i)$ such that $C_i \subset C$, that $\meas (C  \setminus \bigcup_iC_i)=0$, that the orientation of each $(C_i, \psi_i)$ is compatible with $\omega$, and that $\psi_i$ is a $(1 \pm \epsilon)$-bi-Lipschitz embedding.
Then for any $(f, \pi ) \in \mathcal{D}^k(\mathbf{R}^k)$, we have
\begin{align*}
&\phi_{\sharp}\left(T_{\omega}\lfloor_C\right)(f, \pi)\\
&=\int_{C}\langle \omega, \dist (\pi_1 \circ \phi) \wedge \cdots \wedge \dist (\pi_k \circ \phi) \rangle f \circ \phi \dist \meas \\
&=\int_{C}\langle \omega, \dist \phi_1 \wedge \cdots \wedge \dist \phi_k \rangle \det J(\pi)(\phi )f \circ \phi \dist \meas \\
&=\sum_i\int_{C_i}\langle \omega, \dist \phi_1 \wedge \cdots \wedge \dist \phi_k \rangle  \det J(\pi) (\phi )f \circ \phi \dist \meas \\
&=\sum_i\int_{C_i}\langle \omega, \dist \psi_{i, 1} \wedge \cdots \wedge \dist \psi_{i, k} \rangle \det J(\phi \circ \psi_i^{-1})(\psi ) \det J(\pi)(\phi )f \circ \phi \dist \meas \\
&=\left(1 \pm \Psi (\epsilon; k)\right)\sum_i\int_{C_i} \det J(\phi \circ \psi_i^{-1})(\psi ) \det J(\pi)(\phi )f \circ \phi \dist \meas \\
&=\left(1 \pm \Psi (\epsilon; k)\right)\sum_i\int_{C_i} \det J(\phi \circ \psi_i^{-1})(\psi ) \det J(\pi )(\phi )f \circ \phi g \dist \mathcal{H}^k \\
&=\left(1 \pm \Psi (\epsilon; k)\right)\sum_i\int_{\psi_i (C_i)} \det J(\phi \circ \psi_i^{-1})(\psi ) \det J(\pi)(\phi \circ \psi_i^{-1})f( \phi \circ \psi_i^{-1})  g (\psi_i^{-1})\dist \mathcal{H}^k \\
&=\left(1 \pm \Psi (\epsilon; k)\right)\sum_i\int_{\phi (C_i)}\det J(\pi) f  g (\phi^{-1})\dist \mathcal{H}^k \\
&=\left(1 \pm \Psi (\epsilon; k)\right)\int_{\bigcup_i\phi (C_i)}\det J(\pi ) f  g (\phi^{-1})\dist \mathcal{H}^k \\
&=\left(1 \pm \Psi (\epsilon; k)\right)\int_{\phi (C)}\det J(\pi ) f  g (\phi^{-1})\dist \mathcal{H}^k
\end{align*}
which completes the proof because $\epsilon$ is arbitrary.
\end{proof}
As a summary of this subsection we have the following.
\begin{theorem}[Stability of canonical currents for noncollapsed sequences]\label{thm:noncollapsed compactness}
Let $(X_i, x_i, \mathcal{H}^n)$ be a sequence of $n$-dimensional oriented Riemannian manifolds with their orientations $\omega_i \in C^{\infty}(X_i)$ satisfying that $\mathrm{Ric}_{X_i} \ge -(n-1)$ and $\mathcal{H}^n(B_1(x_i)) \ge v>0$. 
Then there exist a subsequence $i(j)$, the noncollapsed Ricci limit space $(X, x, \mathcal{H}^n)$ of $(X_{i(j)}, x_{i(j)}, \mathcal{H}^n)$ and an orientation $\omega \in L^{\infty}(\bigwedge^nT^*X)$ of $(X, x, \mathcal{H}^n)$ such that the following hold.
\begin{enumerate}
\item $\omega \in \mathcal{D}^{\infty}_{\mathrm{Loc}}(\delta_n, X)$ with $\delta_n\omega=0$.
\item $T_{\omega}$ is a locally integral current with $\partial T_{\omega} =0$.
\item $\|T_{\omega}\| = \mathcal{H}^n$ on the set of all Borel subsets of $X$. 
\item For any Borel subset $C$ of $X$ and any bi-Lipschitz embedding $\phi: C \hookrightarrow \mathbf{R}^n$ we have \begin{equation}\label{vh}
\phi_{\sharp}(T_{\omega}\lfloor_C)=[1_{\phi (C^+)}-1_{\phi (C^-)}],
\end{equation}
where $C^{\pm}$ are Borel subsets of $C$ such that the orientation of $C^+$ (or $C^-$, respectively) is (not, respectively) compatible with $\omega$ and that $\mathcal{H}^n(C \setminus (C^+ \cup C^-))=0$.
In particular $T_{\omega}$ has multiplicity one in the following sense; for any pairwise disjoint rectifiable atlas $\{(C_i, \phi_i)\}_i$ of $(X, x, \mathcal{H}^n)$ satisfying that the orientation of each patch $(C_i, \phi_i)$ is compatible with $\omega$, we have 
\begin{equation}\label{ppii}
T_{\omega}=\sum_i(\phi_i^{-1})_{\sharp}(1_{\phi_i (C_i)})
\end{equation}
and 
\begin{equation}\label{ppi1}
\|T_{\omega}\|(A)=\sum_i \|(\phi_i^{-1})_{\sharp}(1_{\phi_i(C_i)})\|(A)
\end{equation}
for any Borel subset $A$ of $X$.
\item $\omega_{i(j)}$ $L^p_{\mathrm{loc}}$-strongly converge to $\omega$ for any $p \in (1, \infty)$. 
\item $T_{\omega_{i(j)}}$ converge to $T_{\omega}$ in the sense of (\ref{eq:15}).
\end{enumerate} 
\end{theorem}
\begin{proof}
It suffices to prove (2) and (4).
We first check (4).

Note that (\ref{vh}) is a direct consequence of Theorem \ref{thm:push}.
Then we have 
\begin{align*}
T_{\omega}(f, \pi)&=\int_X\langle \omega, f\dist \pi_1 \wedge \cdots \wedge \dist \pi_n\rangle \dist \mathcal{H}^n\\
&=\sum_i\int_{C_i}\langle \omega, f\dist \pi_1 \wedge \cdots \wedge \dist \pi_n\rangle \dist \mathcal{H}^n\\
&=\sum_i(\phi_i)_{\sharp}(T_{\omega}\lfloor_{C_i})(f \circ \phi_i^{-1}, \pi \circ \phi_i^{-1}) \\
&=\sum_i[1_{\phi_i(C_i)}](f \circ \phi_i^{-1}, \pi \circ \phi_i^{-1}) =\sum_i(\phi_i^{-1})_{\sharp}(1_{\phi_i(C_i)})(f, \pi),
\end{align*}
which proves (\ref{ppii}).
(\ref{ppi1}) follows from (\ref{ppii}) and a fact that $\|T\lfloor_A\|=\|T\|\lfloor_A$ (see \cite[(2.8)]{LW} and the proof of \cite[Theorem 8.3]{Lang}).
Thus we have (4).

Then (2) is a direct consequence of (\ref{ppii}), (\ref{ppi1}) and \cite[Theorem 8.3]{Lang}.
\end{proof}
\subsection{Compatibility with convergence of metric currents}
Let us explain a relationship between Theorem \ref{thm:noncollapsed compactness}, the weak convergence of currents given in \cite{AmbrosioKirchheim} by Ambrosio-Kirchheim, the pointed flat compactness theorem given in \cite{LW} by Lang-Wenger, and the intrinsic flat convergence defined in \cite{SormaniWenger} by Sormani-Wenger.

Wenger defined in \cite{wenger} the \textit{flat distance}, denoted by $\dist_{\mathcal{F}}^Z$, between two $m$-dimensional integral currents $S, T$ on a metric space $Z$, which is a generalization of Federer-Fleming's flat distance on Euclidean spaces to arbitrary metric spaces as follows;
$$
\dist_{\mathcal{F}}^Z(S, T):=\inf \{\|U\|(Z)+\|V\|(Z)\},
$$
where the infimum runs over all $m$-dimensional integral currents $U$ on $Z$ and all $(m+1)$-dimensional integral currents  $V$ on $Z$ with $S-T=U +\partial V$.
He also showed that the convergence of a sequence of integral currents $T_i$ on $Z$ to an integral current $T$ on $Z$ with respect to $\dist_{\mathcal{F}}^Z$ implies the weak convergence, in another word, the pointwise convergence holds;
$$
\lim_{i \to \infty}T_i(f, \pi)=T(f, \pi)
$$
for any bounded Lipschitz function $f$ on $Z$ and any $\pi_i \in \mathrm{LIP}(Z)$.
Moreover the reverse implication also holds under mild additional assumptions. See \cite[Theorems 1.2 and 1.4]{wenger}.

Sormani-Wenger defined in \cite{SormaniWenger} the distance $\dist_{\mathcal{F}}$, so called the \textit{intrinsic flat distance}, between two integral current spaces $(X, T), (Y, S)$ as follows;
$$
\dist_{\mathcal{F}}\left((X, T), (Y, S)\right):=\inf \dist_{\mathcal{F}}^Z\left(\phi_{\sharp}(X, T), \psi_{\sharp}(Y, S)\right),
$$
where the infimum runs over all isometric embeddings to a metric space $Z$; $\phi:X \hookrightarrow Z$, $\psi:Y \hookrightarrow Z$.
They gave fundamental properties of the convergence, which include that a sequence of integral current spaces $(X_i, T_i)$ converge to an integral current space $(X, T)$ with respect to the intrinsic flat distance if and only if 
there exist a complete separable metric space $Z$ and a sequence of isometric embeddings $\phi_i:X_i \hookrightarrow Z$, $\phi:X \hookrightarrow Z$ such that 
\begin{equation}\label{hhuuii1}
\lim_{i \to \infty}\dist_{\mathcal{F}}^Z\left((\phi_i)_{\sharp}(X_i, T_i), \phi_{\sharp}(X, T)\right)=0.
\end{equation}
See \cite[Theorem 4.2]{SormaniWenger}.
Note that Wenger also proved in \cite{wenger2} a compactness that a sequence of integral current spaces $\{(X_i, T_i)\}_i$ with $\sup_i(\|T_i\|(X_i)+\|\partial T_i\|(X_i))<\infty$ and $\sup_i \mathrm{diam}\,(\supp T_i)<\infty$ has a convergent subsequence with respect to $\dist_{\mathcal{F}}$, where $\supp T:=\supp \|T\|$. See \cite[Theorem 1.2]{wenger2}.

On the other hand Lang-Wenger introduced the following convergence; a sequence of pointed $m$-dimensional locally integral current spaces $(X_i, x_i, T_i)$ \textit{converge in the pointed flat sense to} a pointed $m$-dimensional locally integral current space $(X, x, T)$ if there exist a complete metric space $Z$, and a sequence of isometric embeddings $\phi_{i}:X_i \hookrightarrow Z$, $\phi:X \hookrightarrow Z$ such that $\lim_{i \to \infty}\phi_{i}(x_{i})=\phi (x)$ and that $(\phi_{i})_{\sharp}T_{i}$ converge to $\phi_{\sharp}T$ in the local flat topology in $Z$, i.e. for any bounded closed subset $B$ of $Z$ there exists a sequence of $(m+1)$-dimensional integral currents $S_{i}$ on $Z$ such that $\lim_{i \to \infty}(\| \phi_{\sharp}T-(\phi_{i})_{\sharp}T_{i}-\partial S_{i}\|(B)+\|S_{i}\|(B)) =0$.
Then in particular the weak convergence, i.e. the pointwise convergence is satisfied;
\begin{equation}\label{hhuuii}
\lim_{i \to \infty}\left((\phi_{i})_{\sharp}T_{i}\right) (f, \pi)=\left(\phi_{\sharp}T\right) (f, \pi)
\end{equation}
for any $(f, \pi) \in \mathcal{D}^m(Z)$.
They also established in \cite{LW} a compactness that a sequence of pointed $m$-dimensional locally integral current spaces $(X_i, x_i, T_i)$ with $\sup_i(\|T_i\|(B_R(x_i)) +\|\partial T_i\|(B_R(x_i)) )<\infty$ for any $R \in (0, \infty)$ has a convergent subsequence in the pointed flat sense above and showed the uniqueness of such limits.  See \cite[Theorem 1.1 and Proposition 1.2]{LW}.
In particular note that if  a sequence $(X_i, x_i, T_i)$ satisfies $\sup_i (\|T_i\|(X_i)+\|\partial T_i\|(X_i))<\infty$, $x_i \in \supp T_i$ and $\sup_i\mathrm{diam}\,(\supp T_i)<\infty$, then (after dropping the information of base points) the pointed flat convergence coincides with that with respect to $\dist_{\mathcal{F}}$. 

Let us turn to the mGH-convergence.
Gigli-Mondino-Savar\'e defined in \cite{GigliMondinoSavare} a new notion of convergence for metric measure spaces, so called the \textit{pointed measured Gromov convergence} (written by the pmG-convergence, for short), as follows;
a sequence of pointed metric measure spaces $(X_i, x_i, \meas_i)$ \textit{pmG-converge to} a pointed metric measure space $(X, x, \meas)$ if there exist a complete separable metric space $Z$ and a sequence of isometric embeddings $\phi_i:X_i \hookrightarrow Z$, $\phi:X \hookrightarrow Z$ such that 
\begin{equation}\label{hhuuii2}
\lim_{i \to \infty}\phi(x_i)=\phi(x), \lim_{i \to \infty}\int_Zf \dist (\phi_i)_{\sharp}\meas_i=\int_Zf\dist \phi_{\sharp}\meas
\end{equation}
for any continuous function $f$ on $Z$ with bounded support.
They gave fundamental properties of the convergence (\ref{hhuuii2}), which include that for a sequence of pointed metric measure spaces $(X_i, x_i, \meas_i)$ with a locally uniform doubling condition, $(X_i, x_i, \meas_i)$ pmG-converge to a pointed metric measure space $(X, x, \meas)$ if and only if it is a mGH-convergent sequence.
See \cite[Theorem 3.15]{GigliMondinoSavare}.

We are now in a position to introduce the relationship between convergence above.
Let us consider a noncollapsed sequence of $n$-dimensional oriented Riemmanian manifolds 
\begin{equation}\label{kkkk}
(X_i, x_i, \mathcal{H}^n) \stackrel{GH}{\to} (X, x, \mathcal{H}^n)
\end{equation}
with their orientations $\omega_i \in C^{\infty}(\bigwedge^nT^*X_i)$ satisfying $\mathrm{Ric}_{X_i}\ge-(n-1)$ as in Theorem \ref{thm:noncollapsed compactness}.
With no loss of generality we can assume that there exists an orientation $\omega$ of $(X, x, \mathcal{H}^n)$ associated with $\omega_i$. 

The first compatibility result between the intrinsic flat convergence and the mGH-convergence is given in \cite{SormaniWenger2} by Sormani-Wenger for compact manifolds with nonnegative Ricci curvature. 
After that Munn proved in \cite{Mike} similar compatibility for compact manifolds with bounded Ricci curvature.
More recently Matveev-Portegies showed in \cite{MP} the compatibility for compact manifolds with a uniform Ricci bound from below.

In our setting (\ref{kkkk}), Theorem \ref{thm:noncollapsed compactness} can be regarded as a generalization of their compatibilities above to the noncompact case as follows; 

By the equivalence between the pmG-convergence and the mGH-convergence in this setting, with no loss of generality we can assume that (\ref{hhuuii2}) is satisfied. Then for any $(f, \pi) \in \mathcal{D}^n(Z)$ by letting $g_i:=f \circ \phi_i \in \mathrm{LIP}_B(X_i), \hat{\pi}_{i, j}:=\pi_j \circ \phi_i \in \mathrm{LIP}_{\mathrm{Loc}}(X_i), g:=f \circ \phi \in \mathrm{LIP}_B(X)$ and $\hat{\pi}_j :=\pi_j \circ \phi \in \mathrm{LIP}_{\mathrm{Loc}}(X)$, Theorem \ref{thm:noncollapsed compactness} yields 
$$
\lim_{i \to \infty}T_{\omega_i}(g_i, \hat{\pi}_{1, i}, \ldots, \hat{\pi}_{n, i})=T_{\omega}(g, \hat{\pi}_1, \ldots, \hat{\pi}_n), 
$$
which is equivalent to (\ref{hhuuii}).
As a summary we have the following compatibility;
\begin{theorem}[Compatibility with convergence of metric currents]\label{comcom}
Let $(X_i, x_i, \mathcal{H}^n)$ be a noncollapsed mGH-convergent sequence of oriented Riemannian manifolds with their orientations $\omega_i$ to $(X, x, \mathcal{H}^n)$ with the orientation $\omega$ associated with $\omega_i$, and let $(Y, y, T)$ be the limit space of $(X_i, x_i, T_{\omega_i})$ in the sense of the weak convergence (\ref{hhuuii}).
Then $(X, x, T_{\omega})$ is isometric to $(Y, y, T)$, i.e. there exists an isometry $\phi:X \to \supp T$ such that $\phi (x)=y$ and that $\phi_{\sharp}T_{\omega}=T$.
\end{theorem}
\begin{proof}
It follows from the observation above, Theorem \ref{thm:noncollapsed compactness} and \cite[Proposition 1.2]{LW}.
\end{proof}
\section{Duality and Spectral convergence}
In this section we prove Theorem \ref{abc}.
For that we first give a quick introduction of the Hodge theory for $RCD$-metric measure spaces established in \cite{Gigli} by Gigli.
See \cite{AmbrosioGigliSavare13, AmbrosioGigliSavare14, LottVillani, Sturm06a, Sturm06b} for the detail of the study of metric measure spaces with Ricci curvature bounded from below.

Let $(Y, \meas)$ be an $RCD(K, \infty)$-space.
For any $\eta \in \mathrm{TestForm}_l(Y)$ there exists a unique $\alpha \in L^2(\bigwedge^{l-1}T^*Y)$, denoted by $\delta \eta$, such that 
$$
\int_Y\langle \alpha, \beta \rangle \dist \meas =\int_Y\langle \eta, \dist \beta \rangle \dist \meas
$$
for any $\beta \in \mathrm{TestForm}_{l-1}(Y)$, where
$$
\dist (f_0\dist f_1 \wedge \cdots \wedge \dist f_l) = \dist f_0 \wedge \dist f_1 \wedge \cdots \wedge \dist f_l
$$
for any $f_i \in \mathrm{Test}F(Y)$.
Then let us denote by $H^{1, 2}_H(\bigwedge^lT^*Y)$ the completion of $\mathrm{TestForm}_l(Y)$
with respect to the norm $\|\omega \|_{H^{1, 2}_H}:=(\|\omega \|_{L^2}^2 + \|\dist \omega\|^2_{L^2}+\|\delta \omega \|_{L^2}^2)^{1/2}$.

Note that $\dist \eta \in H^{1, 2}_H(\bigwedge^{l+1}T^*Y)$ and $\delta \eta \in L^2(\bigwedge^{l-1}T^*Y)$ are well-defined for any $\eta \in H^{1, 2}_H(\bigwedge^lT^*Y)$.
Then the \textit{$l$-dimensional de Rham cohomology} $H^l_{\mathrm{dR}}(Y)$ is defined by
$$
H^l_{\mathrm{dR}}(Y):=\frac{\mathrm{Ker} \left( \dist: H^{1, 2}_H(\bigwedge^lT^*Y) \to H^{1, 2}_H(\bigwedge^{l+1}T^*Y) \right)}{\overline{\mathrm{Im} \left(  \dist: H^{1, 2}_H(\bigwedge^{l-1}T^*Y) \to H^{1, 2}_H(\bigwedge^{l}T^*Y)\right)}},
$$
where the closure in the denominator is in the $L^2$-sense.
Then the Hodge theorem is satisfied; the canonical map $\eta \mapsto [\eta ]$ from the space of harmonic $l$-forms
$$
\mathrm{Harm}_l(Y):=\left\{\eta \in H^{1, 2}_H(\bigwedge^lT^*Y); \dist \eta=0, \delta \eta=0\right\}
$$
to $H^l_{\mathrm{dR}}(Y)$ gives an isomorphism. See \cite[Theorem 3.5.15]{Gigli}.

Let us denote by $\mathcal{D}(\Delta_{H, l}, Y)$ the set of $\eta \in H^{1, 2}_H(\bigwedge^lT^*Y)$ such that there exists a unique $\alpha \in L^2(\bigwedge^lT^*Y)$, denoted by $\Delta_{H, 1}\eta$, such that 
$$
\int_Y\langle \alpha, \beta \rangle \dist \meas=\int_Y\langle \dist \eta, \dist \beta \rangle + \langle \delta \eta, \delta \beta \rangle \dist \meas
$$
for any $\beta \in H^{1, 2}_H(\bigwedge^lT^*Y)$. 
We call $\Delta_{H, l}$ the \textit{($l$-dimensional) Hodge Laplacian}.
Note that $\eta \in \mathcal{D}(\Delta_{H, l}, X)$ with $\Delta_{H, l}\eta =0$ if and only if $\eta \in \mathrm{Harm}_l(Y)$.

Similarly we can define the Sobolev space $H^{1, 2}_C(\bigwedge^lT^*Y)$ for $l$-forms by the completion of $\mathrm{TestForm}_l(Y)$ with respect to the norm  $\|\omega \|_{H^{1, 2}_C}:=(\|\omega \|_{L^2}^2 +\|\nabla \omega \|_{L^2}^2)^{1/2}$ and define the connection Laplacian, denoted by $\Delta_{C, l}$, acting on $l$-forms. See also \cite{Honda3}.

If $(Y, \meas)$ is a compact Ricci limit space with $\mathrm{diam}\,Y>0$ (which is also an $RCD(-(n-1), \infty)$-space), then the spectrum of $\Delta_{H, 1}$ is discrete and unbounded. 
In particular $H^1_{\mathrm{dR}}(Y)$ is finite dimensional. 
See \cite[Remark 4.10]{Honda6}.

The following is a main result of \cite{Honda6}, which will also play a key role in this section.
\begin{theorem}[Spectral convergence of $\Delta_{H, 1}, \Delta_{C, 1}$]\cite[Theorems 1.1 and 1.2]{Honda6}\label{thm:spectral conv}
Let $X_i$ be a sequence of $n$-dimensional compact Riemannian manifolds with $|\mathrm{Ric}_{X_i}| \le n-1$, and let $X$ be the noncollapsed GH-limit space (recall that then $(X_i, \mathcal{H}^n) \stackrel{GH}{\to} (X, \mathcal{H}^n)$ is satisfied).
Then spectral convergence of $\Delta_{H, 1}, \Delta_{C, l}$ hold for any $l \in \mathbb{N}$, i.e. 
$$
\lim_{i \to \infty}\lambda^{H, 1}_k(X_i)=\lambda_k^{H, 1}(X)<\infty
$$
and
$$
\lim_{i \to \infty}\lambda_k^{C, l}(X_i)=\lambda_k^{C, l}(X)<\infty
$$
hold for any $k$,
where $\lambda_k^{H, 1}, \lambda_k^{C, l}$ are their $k$-th eigenvalues of $\Delta_{H, 1}, \Delta_{C, l}$, respectively.
\end{theorem}
Note that the $L^2$-strong convergence of eigenforms is also established. In particular all eigenforms have quantitative $L^{\infty}$-bounds. See \cite[Theorem 1.4]{Honda6}.

Let us give a characterization of harmonic $1$-forms. Under the same notation as in the theorem above, recall again that $\mathcal{R}(X)$ is an open subset of $X$ and is a $C^{1, \alpha}$-Riemannian manifold for any $\alpha \in (0, 1)$ (\cite[Theorem 7.2]{CheegerColding}), and that as mentioned in subsection 6.5, since the Hausdorff dimension of the singular set $X \setminus \mathcal{R}(X)$ is at most $4$ (\cite[Theorem 1.4]{CheegerNaber}), the capacity of the singular set is zero; $\mathrm{Cap}^{\mathcal{H}^n}_2(X \setminus \mathcal{R}(X))=0$. See Theorem \ref{abc} for the definition of $\mathrm{Harm}_l^{\infty}$.
\begin{proposition}\label{prop:lin}
Let $X$ be as in Theorem \ref{thm:spectral conv}.
Then $\mathrm{Harm}_1^{\infty}(\mathcal{R}(X))$
coincides with $\mathrm{Harm}_1(X)$.
\end{proposition}
\begin{proof}
It is not difficult to check this by the argument same to the proof of \cite[Proposition 4.5]{Honda6}.
We give a sketch of the proof for reader's convenience.

First we prove $\mathrm{Harm}_1(X) \subset \mathrm{Harm}_1^{\infty}(\mathcal{R}(X))$.
Let $\omega \in \mathrm{Harm}_1(X)$. By Theorem \ref{thm:spectral conv}, there exist a sequence $\omega_i \in C^{\infty}(T^*X_i)$ and a sequence $\lambda_i \in \mathbf{R}_{\ge 0}$ such that $\sup_i\|\omega_i\|_{L^{\infty}}<\infty$, that $\lambda_i \to 0$, that $\Delta_{H, 1}\omega_i =\lambda_i\omega_i$ and that $\omega_i, \dist \omega_i, \delta \omega_i$ $L^2$-strongly converge to $\omega, \dist \omega, \delta \omega$, respectively. 
On the other hand by the existence of an approximate sequence \cite[Theorem 1.11]{Honda3} (or Corolalry \ref{vf}), for any $\eta=f_0\dist f_1 \wedge \cdots \wedge \dist f_k$, where $f_0 \in \mathrm{LIP}_c(\mathcal{R}(X))$ and $f_i \in \mathrm{Test}F(X)(i=1, 2, \ldots, k)$, there exists a smooth approximate sequence $\eta_i \in C^{\infty}(T^*X_i)$ such that $\eta_i, \dist \eta_i, \delta \eta_i$ $L^2$-strongly converge to $\eta, \dist \eta, \delta \eta$, respectively.
Since 
\begin{equation}\label{bbbbbb}
\int_{X_i}\langle \dist \omega_i, \dist \eta_i\rangle +\langle \delta \omega_i, \delta \eta_i\rangle \dist \mathcal{H}^n=\lambda_i\int_{X_i}\langle \omega_i, \eta_i \rangle \dist \mathcal{H}^n,
\end{equation}
and $\langle \omega_i, \eta_i \rangle \in C^{\infty}(X_i)$ $L^2$-converge strongly to $\langle \omega, \eta \rangle$ with $\sup_i\|\langle \omega_i, \eta_i\rangle \|_{H^{1, 2}}<\infty$, where we used that Bochner's inequality implies $\sup_i(\|\nabla \omega_i\|_{L^2}+\|\nabla \eta_i\|_{L^2})<\infty$, 
letting $i \to \infty$ in (\ref{bbbbbb}) with the Rellich compactness \cite[Theorem 4.9]{Honda2} yields $\omega \in \mathrm{Harm}_1^{\infty}(\mathcal{R}(X))$, i.e. $ \mathrm{Harm}_1(X) \subset \mathrm{Harm}_1^{\infty}(\mathcal{R}(X))$.

Next let $\omega \in  \mathrm{Harm}_1^{\infty}(\mathcal{R}(X))$.
Then the condition that $\langle \omega, \eta \rangle \in H^{1, 2}(X)$ for any $\eta=f_0\dist f_1 \wedge \cdots \wedge \dist f_k$, where $f_0 \in \mathrm{LIP}_c(\mathcal{R}(X))$ and $f_i \in \mathrm{Test}F(X)(i=1, 2, \ldots, k)$, 
yields that $\omega$ is a locally $H^{1, 2}$-Sobolev 1-form on $\mathcal{R}(X)$, i.e.
for any open subset $U$ of $\mathcal{R}(X)$ with $\overline{U} \subset \mathcal{R}(X)$, $\omega$ can be written by
$\omega=\sum_{i=1}^Nf_{0, i}\dist f_{1, i}$
on $U$ for some $f_{0, i} \in H^{1, 2}(U)$ and some $f_{1, i} \in C^{\infty}(U)$ (c.f. the proof of \cite[Proposition 4.5]{Honda6}).  

Since $\mathrm{Cap}^{\mathcal{H}^n}_2(X \setminus \mathcal{R}(X))=0$, there exists a sequence of $\phi_i \in \mathrm{LIP}_c(\mathcal{R}(X))$ such that $0 \le \phi_i \le 1$ and that $\phi_i \to 1$ in $H^{1, 2}(X)$.
Then it is easy to check that $\phi_i\omega \in H^{1, 2}_H(T^*X)$, that $\sup_i\|\phi_i\omega\|_{H^{1, 2}_H}<\infty$ and that $\phi_i\omega \to \omega$ in $L^2(T^*X)$. 
Thus we have $\omega \in H^{1, 2}_H(T^*X)$.
Moreover by $\dist (\phi_i \omega)=\dist \phi_i \wedge \omega$ and $\delta (\phi_i\omega) =-\langle \dist \phi_i, \omega \rangle$, the convergence of $\phi_i \omega$ to $\omega$ in $L^2(X)$ implies $\dist \omega=0$ and $\delta \omega=0$, 
which shows $\mathrm{Harm}_1(X) \subset  \mathrm{Harm}_1^{\infty}(\mathcal{R}(X))$, i.e. $\mathrm{Harm}_1(X)=\mathrm{Harm}_1^{\infty}(\mathcal{R}(X))$.
\end{proof}
\begin{remark}\label{hoho}
In Theorem \ref{thm:spectral conv} we can also check $C^{1, \alpha}$-convergence of eigenforms on each compact subset of $\mathcal{R}(X)$ for any $\alpha \in (0, 1)$.
In particular combining this with an argument similar to the proof of Proposition \ref{prop:lin} yields that $\mathrm{Harm}_1(X)$ also coincides with the set of bounded $C^{1, \alpha}$-harmonic 1-forms on $\mathcal{R}(X)$.
\end{remark}
\begin{theorem}[Rellich compactness for differential forms]\label{relli}
Let $X_i, X$ be as in Theorem \ref{thm:spectral conv} and let $\omega_i \in H^{1, 2}_H(T^*X_i)$ be a sequence with $\sup_i\|\omega_i\|_{H^{1, 2}_H}<\infty$.
Then there exist a subsequence $i(j)$ and $\omega \in H^{1, 2}_H(T^*X)$ such that $\omega_{i(j)}$ $L^2$-strongly converge to $\omega$ and that $\dist \omega_{i(j)}, \delta \omega_{i(j)}$ $L^2$-weakly converge to $\dist \omega, \delta \omega$, respectively.
Similar statements for $H^{1, 2}_C$-differential $k$-forms hold for any $k$.
\end{theorem}
\begin{proof}
We only give a proof in the $H^{1, 2}_H$-case because the proof in the other case is similar.

Since Bochner's formula yields $\sup_i\|\omega_i\|_{H^{1, 2}_C}<\infty$, by $W^{1, 2}_H$-stability results \cite[Theorems 1.12, 6.9, 7.8 and Proposition 7.1]{Honda3}, with no loss of generality we can assume that there exists the $L^2$-strong limit $\omega$ of $\omega_i$ such that $\omega \in W^{1, 2}_H(T^*X)$ and that $\dist \omega_i, \delta \omega_i$ $L^2$-weakly converge to $\dist \omega, \delta \omega$, respectively (see \cite[Definition 3.5.13]{Gigli} for the definition of $W^{1, 2}_H$-Sobolev spaces for differential forms. Note that in general $H^{1, 2}_H \subset W^{1, 2}_H$ holds).
Thus it suffices to check $\omega \in H^{1, 2}_H(T^*X)$.

Let $\eta_i \in \mathcal{D}(\Delta_{H, 1}, X), \eta_{j, i} \in C^{\infty}(T^*X_j)$ satisfy $\Delta_{H, 1}\eta_i=\lambda_i^{H, 1}(X)\eta_i, \Delta_{H, 1}\eta_{j, i}=\lambda_i^{H, 1}(X_j)\eta_{j, i}$ and $\|\eta_i\|_{L^2}=\|\eta_{j, i}\|_{L^2}=1$.
In particular $\{\eta_i\}_i, \{\eta_{j, i}\}_i$ are orthonormal bases in $L^2(T^*X)$, in $L^2(T^*X_j)$, respectively.
By Theorem \ref{thm:spectral conv} with no loss of generality we can assume that $\eta_{j, i}$ $L^2$-strongly converge to $\eta_j$.
Put $\omega =\sum_i a_i \eta_i$ in $L^2(T^*X)$ and $\omega_j=\sum_ia_{j, i}\eta_{j, i}$ in $L^2(T^*X_j)$.
Then for any $N \in \mathbb{N}$ letting $\omega^N:=\sum_{i=1}^Na_i\eta_i \in H^{1, 2}_H(T^*X)$ and  $\omega^N_j=\sum_{i=1}^Na_{j, i}\eta_{j, i}$ show
\begin{align*}
\|\omega^N\|_{H^{1, 2}_H}^2&= \sum_{i=1}^N(a_i)^2\left(1 + \int_X |\dist \eta_i|^2+|\delta \eta_i|^2 \dist \mathcal{H}^n\right) \\
&=\sum_{i=1}^N(a_i)^2(1+\lambda_i^{H, 1}(X)) \\
&=\lim_{j \to \infty}\sum_{i=1}^N(a_{j, i})^2(1+\lambda_i^{H, 1}(X_j))\\
&=\lim_{j \to \infty}\|\omega_j^N\|_{H^{1, 2}_H}^2 \le \liminf_{j \to \infty}\|\omega_j\|_{H^{1, 2}_H}^2<\infty,
\end{align*}
where we used the convergence $a_{j, i}=\int_{X_j}\langle \omega_j,  \eta_{j, i}\rangle \dist \mathcal{H}^n \to \int_X\langle \omega, \eta_i\rangle \dist \mathcal{H}^n=a_i$ as $j \to \infty$ and a fact that $\omega_j =\sum_ia_{j, i}\eta_{j, i}$ in $H^{1, 2}_H(T^*X_j)$.
Thus letting $N \to \infty$ gives $\omega \in H^{1, 2}_H(T^*X)$.
\end{proof}
Finally we give convergence of heat flows $h_t^{H, 1}, h^{C, l}_t$ associated with $\Delta_{H, 1}, \Delta_{H, l}$, respectively. See pages 133 and 154 of \cite{Gigli} for definitions.
\begin{corollary}[Convergence of heat flows]\label{vf}
Let $X_i, X$ be as in Theorem \ref{thm:spectral conv} and let $h_t^{H, 1}$ be the heat flow associated to $\Delta_{H, 1}$.
Then for any $t \in (0, \infty)$ and any $L^2$-strong convergent sequence of $\omega_i \in L^2(T^*X_i)$ to $\omega \in L^2(T^*X)$ we see that $h_t^{H, 1}\omega_i, \dist h_t^{H, 1}\omega_i, \delta h_t^{H, 1}\omega_i$ $L^2$-strongly converge to $h_t^{H, 1}\omega, \dist h_t^{H, 1}\omega, \delta h_t^{H, 1}\omega$, respectively and that $\Delta_{H, 1}h_t^{H, 1}\omega_i$ $L^2$-weakly converge to $\Delta_{H, 1}h_t^{H, 1}\omega$.
Similar results for $\Delta_{C, l}$ also hold for any $l$.
\end{corollary}
\begin{proof}
It is a direct consequence of the continuity of the Hodge Laplacian with respect to the mGH-convergence \cite[Theorem 7.17]{Honda3}, the equivalence between the spectral convergence and the convergence of the heat flow \cite[Theorem 2.4]{KS}, and Theorem \ref{relli}.
\end{proof}
\begin{remark}\label{kl}
By the spectral convergence of the connection Laplacian acting on tensor fields \cite[Theorem 1.2]{Honda6}, Theorem \ref{relli} and Corollary \ref{vf} not only hold for differential forms but also hold for all tensor fields.
\end{remark}
\subsection{Duality between $H^0_{\mathrm{dR}}$ and $H_{\mathrm{dR}}^n$}
For an $RCD(K, \infty)$-space $(Y, \meas)$ and given $f \in H^{1, 2}(Y)$, $f$ is harmonic on $Y$ if and only if $\dist f =0$.
In particular if $f$ is harmonic, then $f$ has a Lipschitz representative with the Lipschitz constant $0$, i.e. $f$ is constant.
Combining this with the Hodge theorem shows $H^0_{\mathrm{dR}}(Y) \cong \mathbf{R}$.

In this subsection we prove a part of Theorem \ref{abc};
\begin{theorem}[Duality between $H^0_{\mathrm{dR}}$ and $H_{\mathrm{dR}}^n$]\label{ww}
Under the same assumptions as in Theorem \ref{thm:spectral conv}, if each $X_i$ is orientable,
then $H^n_{\mathrm{dR}}(X) \cong \mathbf{R}\omega \cong \mathbf{R}$, where $\omega$ is an orientation of $(X, \mathcal{H}^n)$.
\end{theorem}
\begin{proof}
Let $\omega_i$ be an orientation of $X_i$. Then by Theorem \ref{thm:stability} with no loss of generality there exists the orientation $\omega \in L^{\infty}(\bigwedge^nT^*X)$ associated with $\omega_i$.
Moreover since the Hodge star operator $\star_{j, i}: \bigwedge^iT_x^*X_j \to \bigwedge^{n-i}T^*_xX_j$ can be regarded as $\star_{j, i} \in H^{1, 2}_C(T^i_{n-i}X_j)$ with $\nabla \star_{j, i}=0$, by Remark \ref{kl}, with no loss of generality we can assume that there exists $\star_i \in H^{1, 2}_C(T^i_{n-i}X)$ with $\nabla \star_i=0$ such that $\star_{j, i}$ $L^2$-strongly converge to $\star_i$.
Note that $\star_i$ is also the Hodge star operator from $\bigwedge^iT^*X$ to $\bigwedge^{n-i}T^*X$ associated with $\omega$.

Then the proof consists of 4-steps as follows;

\textbf{Step 1.} We see $\omega \in H^{1, 2}_H(\bigwedge^nT^*X) \cap H^{1, 2}_C(\bigwedge^nT^*X)$ with $\dist \omega=0$, $\delta \omega=0$ and $\nabla \omega =0$.

Because
since $\nabla \omega_i \equiv 0, \dist \omega_i \equiv0$ and $\delta \omega_i \equiv 0$, Theorem \ref{relli} (or Theorem \ref{thm:spectral conv}) yields that $\omega \in H^{1, 2}_C(\bigwedge^nT^*X) \cap H^{1, 2}_H(\bigwedge^nT^*X)$ with $\nabla \omega =0$, $\dist \omega=0$ and $\delta \omega=0$.

\textbf{Step 2.} For a Borel function $f$ on $X$, the following three conditions are equivalent:
\begin{enumerate}
\item $f\omega \in H^{1, 2}_H(\bigwedge^nT^*X)$,
\item $f\omega \in H^{1, 2}_C(\bigwedge^nT^*X)$,
\item $f \in H^{1, 2}(X)$.
\end{enumerate}
Moreover if the above hold, then $\dist (f\omega ) =\dist f \wedge \omega$ and $\nabla (f\omega)=\omega \otimes \dist f$.

The proof is as follows.
Let $f\omega \in H^{1, 2}_H(\bigwedge^nT^*X)$.
Then by the existence of an approximate sequence \cite[Theorem 3.5]{Honda6} (or using Corollary \ref{vf}), there exists a smooth approximation $\eta_j \in C^{\infty}(\bigwedge^nT^*X_j)$ such that $\eta_j, \dist \eta_j, \delta \eta_j$ $L^2$-strongly converge to $f\omega, \dist (f\omega ), \delta (f\omega )$, respectively. 
Since $\Delta_{H, n}=\Delta_{C, n}$ on $X_j$, we have $\sup_j\|\eta_j\|_{H^{1, 2}_C}<\infty$. 
In particular since $|\nabla \langle \eta_j, \omega_j \rangle| \le |\nabla \eta_j|$, the Rellich compactness \cite[Theorem 4.9]{Honda2} with the $L^2$-strong convergence of $\langle \eta_j, \omega_j \rangle$ to $\langle f\omega, \omega \rangle =f$ yields $f \in H^{1, 2}(X)$, which completes the proof of the implication from (1) to (3). 
Next we take smooth approximation $f_j \in C^{\infty}(X_j)$ such that $f_j, \dist f_j$ $L^2$-strongly converge to $f, \dist f$, respectively (c.f. \cite[Theorem 4.2]{Honda4}). Then letting $j \to \infty$ in the identity $\dist (f_j\omega_j ) =\dist f_j \wedge \omega_j$ shows $\dist (f\omega ) =\dist f \wedge \omega$.

Similarly we can easily prove the remaining implications and $\nabla (f\omega)=\omega \otimes \dist f$ via taking smooth approximation.

\textbf{Step 3.} For any $\eta \in H^{1, 2}_H(\bigwedge^nT^*X)$, we have $\delta \eta=(-1)^n\star_{n-1}^{-1}\dist \star_n \eta$.

Take a smooth approximation $\eta_j \in C^{\infty}(\bigwedge^nT^*X_j)$ such that $\eta_j, \dist \eta_j, \delta \eta_j$ $L^2$-strongly converge to $\eta, \dist \eta, \delta \eta$, respectively.
Since
\begin{equation}\label{wwe}
\delta \eta_j=(-1)^n\star_{j, n-1}^{-1}\dist \star_{j, n}\eta_j
\end{equation} 
and the left hand side above is $L^2$-strong convergent sequence, $\dist \star_{j, n}\eta_j$ is also $L^2$-strong convergent sequence.
Since $\star_{j, n}\eta_j$ $L^2$-strong converge to $\star_n\eta$, Theorem \ref{relli} shows that $L^2$-strong limit of $\dist \star_{j, n}\eta_j$ is $\dist \star_n\eta$.
Thus letting $j \to \infty$ in (\ref{wwe}) completes the proof.

\textbf{Step 4.} For a Borel function $f$ on $X$, the following are equivalent:
\begin{enumerate}
\item $f\omega \in \mathcal{D} (\Delta_{H, n}, X)$,
\item $f\omega \in \mathcal{D} (\Delta_{C, n}, X)$,
\item $f \in \mathcal{D}(\Delta, X)$.
\end{enumerate}
Moreover if the above holds, then $\Delta_{H, n}(f\omega)=\Delta_{C, n}(f\omega )=(\Delta f) \omega$.

The proof is as follows.
Assume $f\omega \in \mathcal{D} (\Delta_{H, n}, X)$.
Note that step 2 yields $f \in H^{1, 2}(X)$.
For any $g \in \mathrm{Test}F(X)$ since step 2 shows $g\omega \in H^{1, 2}_H(\bigwedge^nT^*X)$, we have by step 3
\begin{align}\label{kkkkkkk}
\int_X\langle \Delta_{H, n}(f\omega), g\omega \rangle \dist \mathcal{H}^n&=\int_X\langle \dist (f\omega), \dist (g\omega) \rangle + \langle \delta (f\omega), \delta (g\omega) \rangle \dist \mathcal{H}^n\nonumber \\
&=\int_X\langle \star_{n-1}^{-1}\dist \star_n (f\omega), \star_{n-1}^{-1}\dist \star_n (g\omega) \rangle \dist \mathcal{H}^n\nonumber \\
&=\int_X\langle \dist f, \dist g \rangle \dist \mathcal{H}^n.
\end{align}
On the other hand letting $\Delta_{H, n}(f\omega)=h\omega$ for some $h \in L^2(X)$, (\ref{kkkkkkk}) implies $f \in \mathcal{D}(\Delta, X)$ with $\Delta f=h$.

Similarly we can easily prove the remaining implications.

Step 4 and the Hodge theorem for $RCD$-spaces \cite[Theorem 3.5.15]{Gigli} complete the proof of Theorem \ref{ww}.
\end{proof}
\subsection{Duality between  $H^1_{\mathrm{dR}}$ and $H_{\mathrm{dR}}^{n-1}$}
In this section we prove the remaining statements in Theorem \ref{abc};
\begin{theorem}[Duality between  $H^1_{\mathrm{dR}}$ and $H_{\mathrm{dR}}^{n-1}$]\label{nnm}
Under the same assumption as in Theorem \ref{ww}, we have $H^1_{\mathrm{dR}}(X) \cong H^{n-1}_{\mathrm{dR}}(X)$.
\end{theorem}
\begin{proof}
Let us use the same notation as in the proof of Theorem \ref{ww}.

\textbf{Step 1.} For an $(n-1)$-form $\eta$ on $X$, the following are equivalent;
\begin{enumerate}
\item $\eta \in H^{1, 2}_H(\bigwedge^{n-1}T^*X)$,
\item $\eta \in H^{1, 2}_C(\bigwedge^{n-1}T^*X)$,
\item $\star_{n-1}\eta \in H^{1, 2}_H(T^*X)$ (recall that it was proven in \cite[Theorem 7.12]{Honda3} that $H^{1, 2}_C(T^*X)=H^{1, 2}_H(T^*X)$).
\end{enumerate}
Moreover if the above hold, then $\delta \eta =(-1)^{n-1}\star_{n-2}^{-1} \dist \star_{n-1} \eta$.

The proof is as follows.
Assume $\eta \in H^{1, 2}_H(\bigwedge^{n-1}T^*X)$. Then take a smooth approximation $\eta_i \in C^{\infty}(\bigwedge^{n-1}T^*X_i)$ such that $\eta_i, \dist \eta_i, \delta \eta_i$ $L^2$-strongly converge to $\eta, \dist \eta, \delta \eta$.
Since $\|\dist \eta_i\|_{L^2}^2+\|\delta \eta_i\|_{L^2}^2=\|\dist \star_{n-1}\eta_i\|_{L^2}^2+\|\delta \star_{n-1}\eta_i\|^2_{L^2}$, we have
$\sup_i\|\eta_i\|_{H^{1, 2}_C}=\sup_i\|\star_{n-1}\eta_i\|_{H^{1, 2}_C}<\infty$, where we used the Bochner inequality.
In particular 
Theorem \ref{relli} yields that $\eta \in H^{1, 2}_C(\bigwedge^{n-1}T^*X)$, that $\star_{n-1}\eta \in H^{1, 2}_H(T^*X)$, and that $\dist \star_{i, n-1}\eta_i, \delta \star_{i, n-1}\eta_i$ $L^2$-weakly converge to $\dist \star_{n-1}\eta, \delta \star_{n-1}\eta$, respectively.
Thus letting $i \to \infty$ in the identity $\delta \eta_i=(-1)^{n-1}\star_{i, n-2}^{-1}\dist \star_{i, n-1}\eta_i$ shows  $\delta \eta =(-1)^n\star_{n-2}^{-1} \dist \star_{n-1} \eta$.
These complete the proof of the implication from (1) to (2) and (3).
Similarly the remaining implications are easily checked.

\textbf{Step 2.} For an $(n-1)$-from $\eta$ on $X$, the following are equivalent;
\begin{enumerate}
\item $\eta \in \mathcal{D}(\Delta_{H, n-1}, X)$,
\item $\eta \in \mathcal{D}(\Delta_{C, n-1}, X)$,
\item $\star_{n-1}\eta \in \mathcal{D}(\Delta_{H, 1}, X)$ (recall that it was proven in \cite[Theorem 1.1]{Honda6} that this is equivalent to $\star_{n-1}\eta \in \mathcal{D}(\Delta_{C, 1}, X)$).
\end{enumerate}
Moreover if the above hold, then $\Delta_{H, n-1}\eta=\Delta_{C, n-1}\eta +\star_{n-1}^{-1}\mathrm{Ric}_X((\star_{n-1}\eta)^*, \cdot )$, $\Delta_{H, 1}\star_{n-1}\eta=\star_{n-1}\Delta_{H, n-1}\eta$, $\star_{n-1}\Delta_{C, n-1}\eta =\Delta_{C, 1}\star_{n-1}\eta$, where $\mathrm{Ric}_X \in L^{\infty}(T^0_2X)$ is the Ricci curvature defined in \cite[Theorem 1.1]{Honda6}.

Since the proof is essentially same to that of step 4 in the proof of Theorem \ref{ww}, we skip the proof.

Then Theorem \ref{nnm} follows from step 2 and the Hodge theorem for $RCD$-spaces \cite[Theorem 3.5.15]{Gigli}.
\end{proof}
Remark \ref{hoho} and proofs of Theorems \ref{ww}, \ref{nnm} yield the following:
\begin{corollary}\label{coco}
$\mathrm{Harm}_{n-1}(X)$ coincides with $\mathrm{Harm}_{n-1}^{\infty}(\mathcal{R}(X))$.
Moreover it is also isometric to the space of bounded $C^{1, \alpha}$-harmonic $(n-1)$-forms on $\mathcal{R}(X)$.
\end{corollary}
Finally by Theorem \ref{thm:spectral conv} and the proof of Theorem \ref{nnm}, we have the following;
\begin{corollary}
Under the same assumption as in Theorem \ref{ww}, spectral convergence of $\Delta_{H, n-1}, \Delta_{C, n-1}, \Delta_{H, n}=\Delta_{C, n}$ hold.
\end{corollary}

\end{document}